\documentclass[11pt,reqno]{amsart}
\usepackage[letterpaper,margin=1in,footskip=0.25in]{geometry}
\usepackage{amssymb}
\usepackage{mathrsfs}
\usepackage{microtype}
\usepackage{mathtools}
\usepackage{mathscinet}
\usepackage{enumitem}
\usepackage[noadjust]{cite}
\renewcommand{\citepunct}{;\penalty\citemidpenalty\ }

\usepackage{calc}
\usepackage{trimclip}
\usepackage{tikz-cd}

\PassOptionsToPackage{pdfusetitle,colorlinks,pagebackref,bookmarksdepth=3,linktoc=all}{hyperref}
\usepackage{bookmark}
\hypersetup{citecolor=[HTML]{2E7E2A},linkcolor=[HTML]{800006},urlcolor=[HTML]{8A0087}}
\usepackage[hyphenbreaks]{breakurl}

\newtheoremstyle{cited}{.5\baselineskip\@plus.2\baselineskip\@minus.2\baselineskip}{.5\baselineskip\@plus.2\baselineskip\@minus.2\baselineskip}{\itshape}{}{\bfseries}{\bfseries .}{5pt plus 1pt minus 1pt}{\thmname{#1}\thmnumber{ #2}\thmnote{ \normalfont#3}}
\theoremstyle{cited}
\newtheorem{theorem}{Theorem}[section]
\newtheorem{corollary}[theorem]{Corollary}
\newtheorem{lemma}[theorem]{Lemma}
\newtheorem{proposition}[theorem]{Proposition}
\newtheorem{step}{Step}[theorem]

\newtheorem{claim}{Claim}[theorem]

\newtheorem{alphthm}{Theorem}

\newtheorem{innercustomthm}{Theorem}
\newenvironment{customthm}[1]
  {\renewcommand\theinnercustomthm{#1}\innercustomthm}
  {\endinnercustomthm}

\newtheoremstyle{citeddef}{.5\baselineskip\@plus.2\baselineskip\@minus.2\baselineskip}{.5\baselineskip\@plus.2\baselineskip\@minus.2\baselineskip}{}{}{\bfseries}{\bfseries .}{5pt plus 1pt minus 1pt}{\thmname{#1}\thmnumber{ #2}\thmnote{ \normalfont#3}}
\theoremstyle{citeddef}
\newtheorem{definition}[theorem]{Definition}
\newtheorem{assumptions}[theorem]{Assumptions}
\newtheorem{setup}[theorem]{Setup}

\newtheoremstyle{citedrem}{.5\baselineskip\@plus.2\baselineskip\@minus.2\baselineskip}{.5\baselineskip\@plus.2\baselineskip\@minus.2\baselineskip}{}{}{\itshape}{\itshape .}{5pt plus 1pt minus 1pt}{\thmname{#1}\thmnumber{ \normalfont#2}\thmnote{ \normalfont#3}}
\theoremstyle{citedrem}
\newtheorem{remark}[theorem]{Remark}

\DeclareMathOperator{\AId}{AId}
\DeclareMathOperator{\Bl}{Bl}

\DeclareMathOperator{\Exc}{Exc}
\DeclareMathOperator{\Hom}{Hom}
\DeclareMathOperator{\HHom}{\mathscr{H}\kern-3pt\mathit{om}}
\DeclareMathOperator{\RRHHom}{\mathbf{R}\mathscr{H}\kern-3pt\mathit{om}}
\DeclareMathOperator{\Proj}{Proj}
\DeclareMathOperator{\Spec}{Spec}
\DeclareMathOperator{\ZR}{ZR}
\DeclareMathOperator{\depth}{depth}
\DeclareMathOperator{\snc}{snc}

\newcommand{\CC}{\mathbf{C}}
\newcommand{\FF}{\mathbf{F}}
\newcommand{\HH}{\mathbf{H}}
\newcommand{\PP}{\mathbf{P}}
\newcommand{\QQ}{\mathbf{Q}}
\newcommand{\RR}{\mathbf{R}}
\newcommand{\kk}{\mathbf{k}}
\newcommand{\cO}{\mathcal{O}}
\newcommand{\fb}{\mathfrak{b}}
\newcommand{\fm}{\mathfrak{m}}
\newcommand{\fp}{\mathfrak{p}}
\newcommand{\sA}{\mathscr{A}}
\newcommand{\sE}{\mathscr{E}}
\newcommand{\sF}{\mathscr{F}}
\newcommand{\sI}{\mathscr{I}}
\newcommand{\sL}{\mathscr{L}}
\newcommand{\sM}{\mathscr{M}}

\newcommand{\id}{\mathrm{id}}
\newcommand{\pr}{\mathrm{pr}}
\newcommand{\red}{\mathrm{red}}
\newcommand{\sm}{\mathrm{sm}}

\newcommand{\hooklongrightarrow}{\lhook\joinrel\longrightarrow}

\DeclarePairedDelimiterX\Set[1]\{\}{#1}

\makeatletter
\newcommand{\longtwoheadrightarrow}{\mathrel{\text{\longtwo@rightarrow}}}
\newcommand{\longtwo@rightarrow}{%
  \sbox0{$\m@th\longrightarrow$}%
  \smash{\rlap{\kern0.175\wd0 \clipbox{{.3\width} {-\height} 0pt {-\height}}{$\m@th\longrightarrow$}}}%
  $\m@th\longrightarrow$%
}
\makeatother

\begin{document}

\title{Relative vanishing theorems for \textnormal{\textbf{Q}}-schemes}
\author{Takumi Murayama}
\address{Department of Mathematics\\Purdue University\\150 N. University Street\\West Lafayette, IN 47907-2067\\USA}
\email{\href{mailto:murayama@purdue.edu}{murayama@purdue.edu}}
\urladdr{\url{https://www.math.purdue.edu/~murayama/}}

\thanks{This material is based upon work supported by the National Science
        Foundation under Grant Nos.\ DMS-1902616 and DMS-2201251.}
\subjclass[2020]{Primary 14F17; Secondary 14E15, 14A15, 14B15, 13F40, 14B05}

\keywords{vanishing theorems, excellent schemes, the Grothendieck limit theorem,
          Zariski--Riemann spaces, rational singularities}

\makeatletter
  \hypersetup{
     pdfsubject = {Primary 14F17; Secondary 14E15, 14A15, 14B15, 13F40, 14B05},
    pdfkeywords = {vanishing theorems, excellent schemes, the Grothendieck limit theorem,
                   Zariski--Riemann spaces, rational singularities}
  }
\makeatother

\begin{abstract}
  We prove the relative Grauert--Riemenschneider vanishing, Kawamata--Viehweg
  vanishing, and Koll\'ar injectivity theorems for proper morphisms of schemes
  of equal characteristic zero,
  solving conjectures of Boutot and Kawakita.
  Our proof uses the Grothendieck limit theorem for sheaf cohomology and
  Zariski--Riemann spaces.
  We also show these vanishing and injectivity theorems hold for locally
  Moishezon (resp.\ projective)
  morphisms of quasi-excellent algebraic spaces admitting dualizing
  complexes and semianalytic germs of complex analytic spaces
  (resp.\ quasi-excellent formal schemes admitting dualizing complexes,
  rigid analytic spaces, Berkovich spaces, and adic spaces
  locally of weakly finite type over a field), all in equal characteristic zero.
  \par We give many applications of our vanishing results.
  For example, we extend Boutot's theorem to all Noetherian
  $\mathbf{Q}$-algebras by showing that if $R \to R'$ is a cyclically pure map
  of $\mathbf{Q}$-algebras and $R'$ is pseudo-rational, then
  $R$ is pseudo-rational.
  This solves a conjecture of Boutot and affirmatively answers a question of
  Schoutens.
  The proof of this Boutot-type result
  uses a new characterization of pseudo-rationality and rational
  singularities using
  Zariski--Riemann spaces.
  This characterization is also used in the proofs of our vanishing and
  injectivity theorems and is of independent interest.
\end{abstract}

\maketitle
\vfil
\setcounter{tocdepth}{1}
{\hypersetup{hidelinks}\tableofcontents}
\newpage

\section{Introduction}\label{sect:intro}
Let $X$ be a smooth complex projective variety.
Kodaira's vanishing theorem \cite[Theorem 2]{Kod53} says that if $\sL$ is an ample
invertible sheaf on $X$, then $H^i(X,\omega_X \otimes \sL) = 0$
for all $i > 0$.
Kodaira's theorem and its generalizations have since become indispensable tools
in algebraic geometry over fields of characteristic zero,
in particular in birational geometry and the minimal model program (see, e.g.,
\cite{KMM87,EV92,KM98,Laz04a,Laz04b,Fuj17}).
\par While the central goal of birational geometry is to study birational
equivalences between projective varieties, this often requires working with
more general schemes.
For example,
Hironaka's original proof of resolution of singularities over
fields of characteristic zero uses an inductive strategy involving
schemes of finite type over quasi-excellent local $\QQ$-algebras (see
\cite[p.\ 162]{Hir64}).
In \citeleft\citen{dFEM10}\citepunct \citen{dFEM11}\citeright, de Fernex, Ein,
and Musta\c{t}\u{a} work with schemes of finite type over formal power series
rings to prove Shokurov's ACC conjecture for log canonical thresholds on
complex algebraic varieties whose singularities belong to a bounded family
(Shokurov's conjecture has since
been proved in general \cite{HMX14}).
In \cite{Kaw15}, Kawakita also works over formal power series rings to prove
a special case of Shokurov's ACC conjecture for minimal log discrepancies
on smooth threefolds.\medskip
\par A problem in these more general contexts is the lack of
Kodaira-type vanishing theorems.
One of the most fundamental generalizations of Kodaira's theorem
for the minimal model program is the Kawamata--Viehweg vanishing theorem
\citeleft\citen{Kaw82}\citemid Theorem 1\citepunct \citen{Vie82}\citemid Theorem
I\citeright, relative versions of which are known to hold for proper
morphisms of varieties over an algebraically closed field of characteristic
zero \cite[Theorem 1-2-3]{KMM87}, or for proper morphisms of
complex analytic spaces that are Moishezon, i.e., bimeromorphic to a projective
morphism \cite[Theorem 3.7]{Nak87}.
In particular, this has been an issue in non-Archimedean geometry, where
the Kawamata--Viehweg vanishing theorem is only known for proper
morphisms to curves
\citeleft\citen{BFJ16}\citemid Appendix B\citepunct \citen{MN15}\citemid
\S5\citeright.\medskip
\par Our main result is the following generalization of the Kawamata--Viehweg
vanishing theorem to proper morphisms of excellent schemes of equal characteristic zero
with dualizing
complexes, which resolves conjectures of Boutot \cite[Remarque 1 on p.\ 67]{Bou87}
and Kawakita \cite[Conjecture 1.1]{Kaw15}.
In fact, we also show a version of Koll\'ar's injectivity theorem \cite[Theorem
2.2]{Kol86} for klt pairs, which for varieties is due to Kawamata \cite[Theorem
3.2]{Kaw85} (see also \cite[Corollaire 1.11]{EV87}).
Below, a morphism $f\colon X \to Y$ is \textsl{maximally dominating} if every
generic point of an irreducible component of $X$ maps to a generic point of an
irreducible component of $Y$ \cite[Expos\'e II, D\'efinition 1.1.2]{ILO14}.
Proper surjective morphisms of integral schemes are maximally dominating.
\begin{alphthm}\label{thm:kvvanishing}
  Let $f\colon X \to Y$ be a proper maximally dominating morphism of
  Noetherian schemes of equal characteristic zero
  such that $Y$ has a dualizing complex $\omega_Y^\bullet$.
  Let $\Delta$ be an effective $\QQ$-Weil divisor on $X$.
  Suppose one of the following conditions holds:
  \begin{enumerate}[label=$(\alph*)$]
    \item $X$ is regular, $\Delta$ has simple normal crossings support, and
      $\lfloor \Delta \rfloor = 0$.
    \item $X$ is normal, $(X,\Delta)$ is klt, and $Y$ is locally excellent.
  \end{enumerate}
  Denote by $\omega_X$ the unique nonzero cohomology sheaf of
  $f^!\omega_Y^\bullet$ (after possibly applying shifts on each connected
  component of $X$) and denote by $K_X$ an associated canonical divisor.
  Consider a Cartier divisor $N$ on $X$ such that $N \sim_\QQ K_X + M + \Delta$ for a
  $\QQ$-Cartier divisor $M$ on $X$.
  \begin{enumerate}[label=$(\roman*)$,ref=\roman*]
    \item\label{thm:kvvanishingvan} Suppose $M$ is $f$-nef and $f$-big.
      Then, we have
      \[
        R^if_*\bigl(\cO_X(N)\bigr) = 0
      \]
      for all $i > 0$.
    \item\label{thm:kvvanishinginj} Suppose $M$ is $f$-semi-ample.
      Let $D$ be an effective Weil divisor on $X$ for which there exists an
      integer $n > 0$ such that $nM$ is Cartier and an
      effective Weil divisor $D'$ on $X$ such that $\cO_X(D+D')
      \simeq \cO_X(nM)$.
      Then,
      the canonical morphisms
      \[
        R^if_*\bigl(\cO_X(N)\bigr) \longrightarrow
        R^if_*\bigl(\cO_X(N+D)\bigr)
      \]
      induced by the inclusion $\cO_X \hookrightarrow \cO_X(D)$ are injective
      for all $i$.
  \end{enumerate}
\end{alphthm}
\par Since Kodaira-type vanishing theorems are false in both positive \cite{Ray78}
and mixed characteristic (Totaro; see \cite[Footnote 1 on p.\ 70]{BMPSTWW}), Theorem
\ref{thm:kvvanishing} and the methods of this paper yield
the most general versions of the Kawamata--Viehweg vanishing
theorem possible for proper morphisms of schemes of arbitrary dimension.
\par We note that on each connected component of $X$, the exceptional pullback
$f^!\omega_Y^\bullet$ of the dualizing complex $\omega_Y^\bullet$
is concentrated in one degree by
local duality \cite[Chapter V, Corollary 6.3]{Har66} since $X$ is
Cohen--Macaulay by assumption in the regular case or by \cite[Theorem 3.1]{BK}
in the klt case (see Theorem \ref{thm:bk31} and Remark \ref{rem:bkcm}).
\par As far as we are aware, the only previously known cases of Theorem
\ref{thm:kvvanishing}$(\ref{thm:kvvanishingvan})$
outside of the context of algebraic varieties or analytic
spaces are when $\dim(X) \le 3$ or $\dim(Y) = 1$.
The case when $\dim(X) = 2$ (which also holds in arbitrary characteristic)
is essentially due to Lipman \cite[Theorem 2.4]{Lip78} when $f$ is
generically finite (see \cite[Theorem 10.4]{Kol13}).
Tanaka \cite[\S3.1]{Tan18} used Lipman's methods to prove the case when $\dim(X)
= 2$ and $f$ is arbitrary (again in arbitrary characteristic).
When $\dim(X) = 3$ and $f$ is birational, Bernasconi and Koll\'ar showed that a
version of Theorem \ref{thm:kvvanishing}$(\ref{thm:kvvanishingvan})$
holds when the residue fields of $Y$
are perfect fields of characteristic $\notin \{2,3,5\}$ \cite[Theorem
2]{BK}.
Some cases when $Y$ is the spectrum of a complete DVR are due to
Boucksom--Favre--Jonsson \cite[Theorem B.3]{BFJ16} and
Musta\c{t}\u{a}--Nicaise \cite[Theorem 5.2.3 and Remark 5.3]{MN15}.
Musta\c{t}\u{a} and Nicaise also proved a version of Theorem
\ref{thm:kvvanishing}$(\ref{thm:kvvanishinginj})$, again under the assumption
that $Y$ is the spectrum of a complete DVR \cite[Theorem 5.3.1 and Remark
5.4]{MN15}.\medskip
\par To prove Theorem \ref{thm:kvvanishing} in the regular case, after replacing $Y$ by
$\Spec(\hat{\cO}_{Y,y})$ for each $y \in Y$, we can use cyclic covers and
log resolutions to reduce to the case when $L$ is $f$-ample and $\Delta = 0$.
We then show the following:
\begin{alphthm}\label{thm:mainvanishing}
  Let $f\colon X \to Y$ be a proper maximally dominating morphism of
  Noetherian schemes of equal characteristic zero
  such that $X$ is locally pseudo-rational and such that $Y$
  has a dualizing complex $\omega_Y^\bullet$.
  Denote by $\omega_X$ the unique nonzero cohomology sheaf of
  $f^!\omega_Y^\bullet$ (after possibly applying shifts on each connected
  component of $X$).
  Consider an invertible sheaf $\sL$ on $X$.
  \begin{enumerate}[label=$(\roman*)$,ref=\roman*]
    \item\label{thm:mainvanishingkv}
      Suppose $\sL$ is $f$-big and $f$-semi-ample.
      Then, we have
      \[
        R^if_*(\omega_X \otimes_{\cO_X} \sL) = 0
      \]
      for all $i > 0$.
    \item\label{thm:mainvanishinggr}
      Suppose $\sL$ is $f$-semi-ample.
      Let $D$ be an effective Weil divisor on $X$ for which there exists an
      integer $n > 0$ and an
      effective Weil divisor $D'$ on $X$ such that $\cO_X(D+D')
      \simeq \sL^{\otimes n}$.
      Then, the canonical morphisms
      \[
        R^if_*(\omega_X \otimes_{\cO_X} \sL^{\otimes k}) \longrightarrow
        R^if_*\bigl(\omega_X \otimes_{\cO_X} \sL^{\otimes k}(D)\bigr)
      \]
      induced by the inclusion $\cO_X \hookrightarrow \cO_X(D)$ are injective
      for all $i$ and for all $k > 0$.
  \end{enumerate}
\end{alphthm}
Pseudo-rationality is a characteristic-free version of rational
singularities introduced by Lipman and Teissier \cite{LT81}
that does not require resolutions of singularities, quasi-excellence, or the
existence of dualizing complexes.
Note that regular rings are locally pseudo-rational \cite[\S4]{LT81}.
The two statements $(\ref{thm:mainvanishingkv})$ and
$(\ref{thm:mainvanishinggr})$ are relative versions of the
Grauert--Riemenschneider vanishing theorem
\cite[Satz 2.1]{GR70} and Koll\'ar's injectivity theorem \cite[Theorem
2.2]{Kol86}, respectively.
We also show dual versions of Theorems \ref{thm:kvvanishing} and
\ref{thm:mainvanishing} analogous to Hartshorne and Ogus's dual
formulation \cite[Proposition 2.2]{HO74} of the relative Grauert--Riemenschneider
vanishing theorem \cite[Satz 2.3]{GR70} (see Theorem
\ref{thm:maindualvanishing})
and Koll\'ar's local version of
Kawamata--Viehweg vanishing \cite[Corollary 20]{Kol11} (see Theorem
\ref{thm:kvvanishingdual}).
These dual statements have the advantage of not requiring that $Y$
has a dualizing complex.
\subsection*{Outline of proof}
\par We outline the proof of Theorem \ref{thm:mainvanishing}.
For simplicity, we consider the case when $f\colon X \to Y$ is a proper
surjective morphism of integral schemes.
The proof of Theorem \ref{thm:mainvanishing} proceeds by
approximating the morphism $f\colon X \to Y$ by proper surjective morphisms of
varieties over $\QQ$,
and then deducing the vanishing in Theorem 
\ref{thm:mainvanishing} from the usual statements for
varieties over a field of characteristic zero.
This approach runs into two major difficulties.
\begin{itemize}
  \item While we can write $f$ as the limit of proper surjective morphisms
    $f_\lambda\colon X_\lambda \to Y_\lambda$ of varieties over $\QQ$ using the
    method of relative Noetherian approximation \cite[\S8]{EGAIV3}, we
    cannot ensure that the $X_\lambda$ are smooth, even if $X$ is regular.
  \item Even though direct images behave well under limits by
    the Grothendieck limit theorem \cite[Expos\'e VI, Th\'eor\`eme 8.7.3]{SGA42} (see
    Theorem \ref{thm:fklimitsdirectimage}), the
    sheaves $\omega_X$ are not known to behave well under limits.
\end{itemize}
\par To fix the smoothness of $X_\lambda$, we want to replace the inverse system
$\{X_\lambda\}_{\lambda \in \Lambda}$ by an inverse system of
resolutions of singularities of
the $X_\lambda$ using Hironaka's resolutions of singularities \cite[Chapter 0,
\S3, Main Theorem I]{Hir64}.
We were unable to choose resolutions of singularities compatibly as $\lambda \in
\Lambda$ varies, so instead we take the inverse system consisting of
\emph{all possible} resolutions of singularities of the $X_\lambda$.
\par A major technical difficulty is to show that the resulting inverse limit is a
familiar locally ringed space, called the \textsl{Zariski--Riemann space} $\ZR(X)$
of $X$.
The Zariski--Riemann space was defined by Zariski for varieties
\citeleft\citen{Zar40}\citemid Definition A.II.5\citepunct \citen{Zar44}\citemid
\S2\citeright\ and by Nagata for Noetherian separated schemes \cite[\S3]{Nag63}.
Thus, even though our initial interest was to show vanishing theorems for
schemes, a surprising and novel aspect of our proof is that we
must leave the world of schemes and consider more general locally
ringed spaces.\medskip
\par To fix the issues with direct images and limits and with the sheaves
$\omega_X$, we use duality to prove statements about local cohomology instead.
The reason this works is that while $\omega_X$ is not known to behave well under
limits, structure sheaves do behave well under limits.
Put together, the proof of Theorems \ref{thm:kvvanishing} and
\ref{thm:mainvanishing} proceeds as follows.
Below, we concentrate on Theorem
\ref{thm:mainvanishing}$(\ref{thm:mainvanishingkv})$.
\begin{enumerate}[label=$(\textup{\Roman*})$,ref=\textup{\Roman*}]
  \item We replace $Y$ with $\Spec(\hat{\cO}_{Y,y})$ to assume that $Y$ is
    the spectrum of an excellent local domain $(R,\fm)$ containing $\QQ$.
  \item Using Lipman's local-global duality \cite[Theorem on p.\ 188]{Lip78} (see
    Lemma \ref{lem:lipmandual} and Proposition \ref{prop:aisastar}), we
    translate the vanishing statements in Theorem \ref{thm:mainvanishing} to
    vanishing statements about the local cohomology modules
    $H^i_Z(\sL^{-1})$ where $Z = f^{-1}(\fm)$.
    These dual statements are analogous to \cite[Proposition
    2.2]{HO74}.
  \item\label{intro:stepzrx}
    We prove a new characterization of pseudo-rational rings in equal
    characteristic zero via Zariski--Riemann spaces (Theorem
    \ref{thm:kempflike}), which is of independent interest.
    This characterization
    shows that the higher direct images of the structure sheaf under the
    projection morphism $\pi\colon\ZR(X) \to X$ from the Zariski--Riemann space of $X$
    vanish.
    It then suffices to show that vanishing holds for the composition
    \[
      \ZR(X) \overset{\pi}{\longrightarrow} X \overset{f}{\longrightarrow}
      \Spec(R).
    \]
    See Definition \ref{def:zr} for the definition of the Zariski--Riemann space
    $\ZR(X)$ associated to $X$.
  \item\label{intro:stepapprox}
    Using the method of relative Noetherian approximation \cite[\S8]{EGAIV3}
    and Hironaka's resolution of singularities \cite[Chapter 0, \S3, Main
    Theorem
    I]{Hir64}, we write the composition in $(\ref{intro:stepzrx})$ as the
    limit of morphisms
    \[
      W_{\lambda,p} \xrightarrow{g_{\lambda,p}} X_\lambda
      \overset{f_\lambda}{\longrightarrow} \Spec(R_\lambda)
    \]
    of varieties over $\QQ$ where the $W_{\lambda,p}$ are smooth (Lemma
    \ref{lem:approx}).
    By a version of the Grothendieck limit theorem 
    \cite[Expos\'e VI, Th\'eor\`eme 8.7.3]{SGA42} for local cohomology modules (Theorem
    \ref{lem:lccolimits}),
    the usual vanishing statements for varieties over $\QQ$ then imply vanishing
    and injectivity theorems for Zariski--Riemann spaces (Theorem
    \ref{thm:zrmaindualvanishing}) that require no assumptions on the
    singularities of $X$ apart from integrality.
    Using the characterization of pseudo-rational rings in
    $(\ref{intro:stepzrx})$, we obtain Theorem \ref{thm:mainvanishing} as a
    consequence.
\end{enumerate}
\par While the technique of applying the Grothendieck limit theorem on one hand
and passing to covers of $X$ that satisfy vanishing theorems on the other have
been applied before, as far as we are aware, the idea to pass to the
Zariski--Riemann space to show vanishing theorems is new.\medskip
\par We describe two sources of inspiration for the Grothendieck limit theorem and
for passing to covers.
The inspiration to use the Grothendieck limit theorem comes from Panin's
proof of the equicharacteristic case of Gersten's conjecture in algebraic
$K$-theory \cite[Theorem A]{Pan03}.
Since Gersten's conjecture is a statement about regular local rings, however,
Panin was able to use N\'eron--Popescu desingularization
\citeleft\citen{Pop86}\citemid Theorem 2.4\citepunct \citen{Pop90}\citemid p.\
45\citepunct\citen{Swa98}\citemid Theorem 1.1\citeright\ to approximate the
regular local rings containing a field $k$ that appear in this special case
of Gersten's conjecture with essentially smooth
algebras over that field $k$ and stay in the category of rings.
By doing so, Panin reduced the equicharacteristic case of Gersten's conjecture
to the geometric case already shown by Quillen \cite[Theorem 5.11]{Qui73}.
This strategy was also applied to related questions in mixed characteristic in
\cite{Ska}.\medskip
\par Our approach is also related to the theory of big Cohen--Macaulay
modules and big Cohen--Macaulay algebras as introduced by Hochster
\cite{Hoc73deep,Hoc75queens,Hoc75cbms}.
In equal characteristic $p > 0$, Hochster and Huneke showed that
the integral closure $X^+$ of $X$ in an algebraic
closure of its function field satisfies a Kodaira-type vanishing theorem
\cite[Theorem 1.2]{HH92}.
Thus, the Zariski--Riemann space $\ZR(X)$ plays a similar role in equal
characteristic zero that the scheme $X^+$ does in equal characteristic $p > 0$.
Smith later asked whether analogues of the Grauert--Riemenschneider or
Kawamata--Viehweg vanishing theorems for $X^+$ hold in equal characteristic $p > 0$
\citeleft\citen{Smi97}\citemid Theorem 2.1\citepunct \citen{Smi97err}\citeright.
Bhatt answered Smith's question in \cite[\S7]{Bha12}.
This approach was recently applied in mixed characteristic
\citeleft\citen{Bha}\citemid Theorem 5.1\citepunct \citen{TY}\citemid
\S3.1\citepunct \citen{BMPSTWW}\citemid \S3\citeright, where $X^+$ was
shown to satisfy Kodaira-type vanishing theorems.
\par For rings $R$ essentially of finite type over the complex numbers,
Roberts showed that the derived pushforward $\RR f_*\cO_{X}$ of the
structure sheaf of a resolution of singularities $f\colon X \to \Spec(R)$ is a
complex which is Cohen--Macaulay in a suitable sense \cite[Corollary on p.\
224]{Rob80} (see also \cite[Proposition 4.17]{IMSW21}).
Our vanishing theorem for Zariski--Riemann spaces (Theorem
\ref{thm:zrmaindualvanishing}) can be seen as a version of Roberts's
result that holds for all integral Noetherian local $\QQ$-algebras.
See \cite[Proposition 4.17]{IMSW21}, which uses Theorems \ref{thm:kvvanishing} and
\ref{thm:mainvanishing} in this paper to show that every excellent $\QQ$-algebra
with a dualizing complex has a maximal Cohen--Macaulay complex that is
equivalent to a graded-commutative dg algebra.
\subsection*{Applications}
Theorems \ref{thm:kvvanishing} and
\ref{thm:mainvanishing} allow one to extend many results to the setting of
excellent rings and schemes of equal characteristic
zero that for varieties rely on Kodaira-type vanishing theorems and resolutions
of singularities.
We describe some examples of these applications.\medskip
\par In joint work with Shiji Lyu \cite[Theorems A and B]{LM}, we use Theorem
\ref{thm:kvvanishing} to prove finite generation of the relative adjoint ring
for proper morphisms of excellent schemes of equal characteristic zero
with dualizing complexes
in the vein of \cite[Theorem 1.2]{BCHM10}.
Our proof uses the approach of Cascini and Lazi\'c \cite{CL12}.
We then prove that one can run the relative minimal model program with scaling
in the sense of \cite{BCHM10,HM10} in this setting,
thereby resolving a question of Koll\'ar \cite[(23)]{Kolhol}.
Using Theorem \ref{thm:kvvanishing} and new GAGA theorems for Grothendieck
duality and dualizing complexes, we then extend the relative minimal model program
with scaling to many categories at once, namely to the categories of
quasi-excellent algebraic spaces and formal schemes with dualizing complexes,
semianalytic germs of complex analytic spaces, Berkovich spaces, rigid analytic
spaces, and adic spaces locally of weakly finite type over a field, all in equal
characteristic zero.
The case for algebraic spaces of finite type over a field was previously shown
by Villalobos-Paz \cite{VP} and the case for
complex analytic spaces was previously shown in \cite{Fuj,DHP}.
In addition, \cite{BMPSTWW} uses Theorem \ref{thm:kvvanishing} to establish
the minimal model program for threefolds of mixed characteristic for
residue characteristics $\notin \{2,3,5\}$, where the characteristic zero fibers
are often excellent schemes of equal characteristic zero and are not
necessarily of finite type over a field.
Lyu and the author of the present paper also extend these results to other
categories of spaces \cite[Theorem $\textup{A}^p$]{LM}.
\par In this paper, using the GAGA-type results obtained in our joint work with
Lyu \cite{LM}, we can show the following vanishing and injectivity
theorems for algebraic spaces, formal schemes, complex analytic spaces, and
non-Archimedean analytic spaces.
\begin{customthm}{\ref*{thm:kvvanishing}\tprime}\label{thm:kvvanishingothercats}
  Let $f\colon X \to Y$ be a proper surjective morphism in one of the following
  categories where $X$ and $Y$ are integral and $X$ is normal:
  \begin{enumerate}[label=$(\textup{\Roman*})$,ref=\textup{\Roman*}]
    \item[$(0)$]
    \makeatletter
    \protected@edef\@currentlabel{0}
    \phantomsection
    \label{setup:introalgebraicspaces}
    \makeatother
      The category of Noetherian algebraic spaces of equal
      characteristic zero
      over a scheme $S$ admitting dualizing complexes.
    \item\label{setup:introformalqschemes}
      The category of Noetherian formal schemes of equal
      characteristic zero
      admitting $c$-dualizing complexes.
    \item\label{setup:introcomplexanalyticgerms} The category of semianalytic
      germs of complex analytic spaces.
    \item\label{setup:introberkovichspaces} The category of $k$-analytic spaces
      over a complete non-Archimedean field $k$ of characteristic zero.
    \makeatletter
    \item[{$(\ref*{setup:introberkovichspaces}')$}]
    \protected@edef\@currentlabel{\ref*{setup:introberkovichspaces}'}
    \phantomsection
    \label{setup:introrigidanalyticspaces}
    \makeatother
      The category of rigid
      $k$-analytic spaces over a complete non-trivially valued
      non-Archimedean field $k$ of characteristic zero.
    \item\label{setup:introadicspaces}
      The category of adic spaces locally of weakly finite type over a complete
      non-trivially valued non-Archimedean field $k$ of characteristic zero.
  \end{enumerate}
  In cases $(\ref{setup:introformalqschemes})$
  and $(\ref{setup:introadicspaces})$,
  we assume that $f$ is projective.
  In cases $(\ref{setup:introberkovichspaces})$ and
  $(\ref{setup:introrigidanalyticspaces})$, we assume that either $f$ is
  projective or that $Y$ is a point.
  \par Let $K_X$ be a canonical divisor on $X$ chosen
  compatibly with a dualizing complex on $Z$,\footnote{For example, when $Z$ is a
  variety over $k$ or in cases $(\ref{setup:introcomplexanalyticgerms})$,
  $(\ref{setup:introberkovichspaces})$,
  $(\ref{setup:introrigidanalyticspaces})$, and
  $(\ref{setup:introadicspaces})$,
  we can choose $K_X$ so that
  $\cO_{X_\sm}( (K_X)_{\vert X_\sm})
  = \det(\Omega_{X_\sm/k})$ where $X_\sm$ is the smooth locus of $X$.
  See \cite[Theorem 24.4$(iii)$, Theorem 24.6$(iii)$, and Theorem 24.8$(iv)$]{LM}.}
  and let $\Delta$ be an effective $\QQ$-Weil divisor on $X$.
  Suppose one of the following conditions holds:
  \begin{enumerate}[label=$(\alph*)$]
    \item $X$ is regular, $\Delta$ has simple normal crossings support, and
      $\lfloor \Delta \rfloor = 0$.
    \item $X$ is normal and $(X,\Delta)$ is klt.
      In cases $(\ref{setup:introalgebraicspaces})$ and
      $(\ref{setup:introformalqschemes})$, we also assume that $Y$ is
      quasi-excellent.
  \end{enumerate}
  Consider a Cartier divisor $N$ on $X$ such that $N \sim_\QQ K_X + M + \Delta$ for a
  $\QQ$-Cartier divisor $M$ on $X$.
  \begin{enumerate}[label=$(\roman*)$,ref=\roman*]
    \item\label{thm:kvvanishingothercatsvan} Suppose $M$ is $f$-nef and $f$-big.
      Then, we have
      \[
        R^if_*\bigl(\cO_X(N)\bigr) = 0
      \]
      for all $i > 0$.
    \item\label{thm:kvvanishingothercatsinj} Suppose $M$ is $f$-semi-ample and
      $f$ is locally Moishezon.
      Let $D$ be an effective Weil divisor on $X$ for which there exists an
      integer $n > 0$ such that $nM$ is Cartier and an
      effective Weil divisor $D'$ on $X$ such that $\cO_X(D+D')
      \simeq \cO_X(nM)$.
      Then,
      the canonical morphisms
      \[
        R^if_*\bigl(\cO_X(N)\bigr) \longrightarrow
        R^if_*\bigl(\cO_X(N+D)\bigr)
      \]
      induced by the inclusion $\cO_X \hookrightarrow \cO_X(D)$ are injective
      for all $i$.
  \end{enumerate}
\end{customthm}
As far as we are aware, in arbitrary dimension,
Theorem \ref{thm:kvvanishingothercats} was previously
only known for varieties and for complex analytic spaces.
In low dimensions, Theorem \ref{thm:kvvanishingothercats} was also known
for non-Archimedean analytic spaces when $\dim(Y)=1$.
For complex analytic spaces, Theorem
\ref{thm:kvvanishingothercats}$(\ref{thm:kvvanishingothercatsvan})$ gives an
alternative proof of Nakayama's version of the Kawamata--Viehweg vanishing
theorem \cite[Theorem 3.7]{Nak87}, and Theorem
\ref{thm:kvvanishingothercats}$(\ref{thm:kvvanishingothercatsinj})$ recovers
the special case of Nakayama's version of Koll\'ar's injectivity theorem
\cite[Theorem 3.10(B)]{Nak87} when $f$ is locally Moishezon.
For non-Archimedean analytic spaces, the case when $\dim(Y) =
1$ follows from the work of Boucksom--Favre--Jonsson \cite[Theorem B.3]{BFJ16} and
Musta\c{t}\u{a}--Nicaise \cite[Theorem 5.2.3, Remark 5.3, Theorem 5.3.1, and
Remark 5.4]{MN15}.
For formal schemes, Smith showed a special
case of Kodaira vanishing \cite[Theorem 4.2.1]{Smi17} and showed that in
general, Kodaira vanishing is false for formal schemes that are smooth and
\emph{pseudo-}projective over the complex numbers \cite[Proposition
4.3.1]{Smi17}.\medskip
\par In \S\ref{sect:rationalsingularities}, we use our vanishing results
to study rational singularities.
First, we show the following version of Boutot's
theorem \cite[Th\'eor\`eme on p.\ 65]{Bou87}, which solves a conjecture of Boutot
\cite[Remarque 1 on p.\ 67]{Bou87} and gives a complete, positive answer to
a question of Schoutens \cite[(2) on p.\ 611]{Sch08}.
This result generalizes a result of Schoutens \cite[Main
Theorem A]{Sch08}, who showed that if $R'$ is regular in the statement below,
then $R$ is locally pseudo-rational.
For the statement below, a ring map $R \to R'$ is \textsl{cyclically pure} if
$IR' \cap R = I$ for every ideal $I \subseteq R$ \cite[p.\ 463]{Hoc77}.
Split or faithfully flat ring maps are cyclically pure \cite[p.\
136]{HR74}.
\begin{customthm}{C}\label{thm:boutot}
  Let $R \to R'$ be a cyclically pure map of Noetherian
  $\QQ$-algebras.
  If $R'$ is locally pseudo-rational, then $R$
  is locally pseudo-rational.
  In particular, if $R'$ is regular, then $R$ is locally pseudo-rational.
\end{customthm}
The proof of Theorem \ref{thm:boutot}
uses techniques from our recent joint work with Charles Godfrey \cite{GM},
where we showed that Du Bois singularities descend under cyclically pure
maps for rings essentially of finite type over the complex numbers.
One key aspect of the proof of Theorem \ref{thm:boutot}
is that our vanishing theorem for Zariski--Riemann
spaces (Theorem \ref{thm:zrmaindualvanishing}) can be interpreted as an
injectivity theorem for the derived pushforward
$\RR\pi_*\cO_{\ZR(X)}$ of the structure sheaf of the Zariski--Riemann space
analogous to Kov\'acs and Schwede's injectivity
theorem for the $0$-th graded piece $\underline{\Omega}^0_X$ of the Deligne--Du
Bois complex \cite[Theorem 3.3]{KS16}.
\par We also prove the following generalizations of results on rational
singularities which were
previously known for rings essentially of finite type over a field of
characteristic zero:
\begin{enumerate}
  \item Pseudo-rationality deforms in equal characteristic zero (Theorem
    \ref{thm:elkik}).
    This extends a result of Elkik \cite[Th\'eor\`eme 5]{Elk78} (for
    rings essentially of finite type over a field of characteristic zero) and the
    author \cite[Proposition 4.17]{Mur} (for quasi-excellent local
    $\QQ$-algebras).
    Note that \cite[Proposition 4.17]{Mur} relies on Theorem
    \ref{thm:mainvanishing} of this paper.
  \item Derived splinters and rational singularities coincide for
    quasi-excellent schemes of equal characteristic zero (Theorem \ref{thm:derivedsplinters}).
    This extends a theorem of Kov\'acs and Bhatt for schemes of finite type over a field of
    characteristic zero \citeleft\citen{Kov00}\citemid Theorem 3\citepunct
    \citen{Bha12}\citemid Theorem 2.12\citeright.
  \item A criterion for Cohen--Macaulayness of Rees algebras over
    quasi-excellent local $\QQ$-algebras with
    rational singularities (Theorem \ref{thm:sdslip}).
    This extends theorems of Sancho de Salas \cite[Theorem 1.7]{SdS87} and
    Lipman \cite[Theorems 4.1 and 4.3]{Lip94} for rings essentially of finite
    type over a field of characteristic zero.
  \item A Brian\c{c}on--Skoda-type theorem for quasi-excellent $\QQ$-algebras with
    rational singularities (Corollary \ref{cor:bsrational}).
    This extends a result of Huneke \cite[Corollary 4.8]{Hun00} for rings
    essentially of finite type over a field of characteristic zero, and provides
    a stronger bound for integral closures of ideals compared to the results in
    \cite{LT81} which apply in the more general context of pseudo-rational rings
    of arbitrary characteristic.
\end{enumerate}
\par We mention that by \cite[Theorem 3.1]{BK},
Theorem \ref{thm:kvvanishing} also implies that excellent dlt pairs of equal
characteristic zero satisfy all local rationality properties known to hold for varieties in
characteristic zero (see Theorem \ref{thm:bk31}).\medskip
\par Additionally, we give an application of our results that is not related to
rational singularities.
\begin{enumerate}[resume]
  \item An adaptation of Hartshorne and Ogus's proof \cite[Corollary 2.6]{HO74}
    that complete local UFD's
    $(R,\fm)$ of dimension $\le 4$ with $R/\fm \simeq \CC$ are Gorenstein that
    removes the algebraizablility condition in \cite{HO74}.
    The result itself is not new to this paper since Raynaud
    (unpublished), Danilov \cite[Theorem 2]{Dan70}, and Boutot \cite[Corollaire
    on p.\ 693]{Bou73} had given different proofs of the steps in \cite{HO74}
    that require algebraizability assumptions.
    Nevertheless, Hartshorne and Ogus's strategy yields an interesting
    proof of this result using analytic methods, which was limited to the
    algebraizable case prior to this paper.
    See Theorem \ref{thm:ho}.
\end{enumerate}
\par We mention that Theorem \ref{thm:kvvanishing} also implies that the theory of
multiplier ideals as developed in \cite[Part Three]{Laz04b}
is now available for all excellent $\QQ$-algebras with
dualizing complexes.
In \cite[\S4.2]{Mursymb}, we used this consequence of Theorem
\ref{thm:kvvanishing} to give a multiplier ideal-theoretic proof of the
uniform comparison between symbolic and ordinary powers of ideals in
regular $\QQ$-algebras due to Ein--Lazarsfeld--Smith \cite[Theorem 2.2 and
Variant on p.\ 251]{ELS01} (for smooth $\CC$-algebras) and Hochster--Huneke
\cite[Theorem 4.4$(a)$]{HH02} (for regular rings of equal characteristic) that
does not rely on reduction modulo $p$ or N\'eron-type desingularization
theorems.
\subsection*{Notation}
All rings are commutative with identity, and all ring maps are unital.
If $k$ is a field, then a \textsl{variety over $k$} is an integral scheme that is
separated and of finite type over $k$.
\par Intersection products on schemes that are proper over a field are defined
using Euler characteristics as in \citeleft\citen{Kle66}\citemid Chapter
I\citepunct \citen{Kle05}\citemid Appendix B\citeright.
Weil and Cartier divisors are defined as in \cite[(21.6.2)]{EGAIV4} and
\cite[D\'efinition 21.1.2]{EGAIV4}, respectively.
The two notions coincide on locally Noetherian schemes that are locally
factorial \cite[Th\'eor\`eme 21.6.9$(ii)$]{EGAIV4}.
A \textsl{$\QQ$-Weil divisor} (resp.\ \textsl{$\QQ$-Cartier divisor}) is a
formal $\QQ$-linear combination of Weil divisors (resp.\ Cartier divisors).
\par We also use the following terminology.
A point $x$ in a topological space $X$ is \textsl{maximal} if it
is the generic point of an irreducible component of $X$ \cite[Chapitre 0,
(2.1.1)]{EGAInew}.
A morphism $X \to Y$ of schemes is \textsl{maximally dominating} if every
maximal point of $X$ maps to a maximal point of $Y$ \cite[Expos\'e II,
D\'efinition 1.1.2]{ILO14}.
\subsection*{Acknowledgments}
We would like to thank
Donu Arapura,
Bhargav Bhatt,
Rankeya Datta,
Yajnaseni Dutta,
Charles Godfrey,
Mattias Jonsson,
J\'anos Koll\'ar,
Linquan Ma,
Mircea Musta\c{t}\u{a},
Yusuke Nakamura,
Giovan Battista Pignatti Morano di Custoza,
Hans Schoutens,
Karl Schwede,
Kazuma Shimomoto,
Chris Skalit,
and
Farrah Yhee
for helpful conversations.
We are grateful to Rankeya Datta for allowing us to include the proof of Theorem
\ref{thm:derivedsplinters}, which arose out of discussions with him, and to
Shiji Lyu for pointing out that Theorems
\ref{thm:kvvanishing} and \ref{thm:mainvanishing} hold without excellence
assumptions.
Finally, we thank the anonymous referee for their comments and suggestions.

\section{Preliminaries}\label{sect:prelim}
\subsection{Excellence and quasi-excellence}
We begin with the notion of excellence.
Grothendieck and Dieudonn\'e conjectured that resolutions of singularities exist
for all quasi-excellent schemes \cite[Remarque 7.9.6]{EGAIV2}.
Temkin proved their conjecture for schemes of equal characteristic zero
\cite[Theorem 1.1]{Tem08}.
\begin{definition}[{\citeleft\citen{EGAIV2}\citemid D\'efinition 7.8.2 and
  (7.8.5)\citepunct \citen{Mat80}\citemid Definition 34.A\citeright}]
  Let $R$ be a ring.
  We say that $R$ is \textsl{excellent} (resp.\ \textsl{quasi-excellent}) if the
  following conditions (resp.\ conditions $(\ref{def:excellentnoeth})$,
  $(\ref{def:excellentgring})$, and $(\ref{def:excellentj2})$ below) are
  satisfied.
  \begin{enumerate}[label=$(\roman*)$,ref=\roman*]
    \item\label{def:excellentnoeth} $R$ is Noetherian.
    \item\label{def:excellentuc} $R$ is universally catenary.
    \item\label{def:excellentgring}
      $R$ is a \textsl{G-ring}, i.e., for every prime ideal $\fp \subseteq R$,
      the $\fp$-adic completion map $R_\fp \to \hat{R}_\fp$ has
      geometrically regular fibers.
    \item\label{def:excellentj2}
      $R$ is \textsl{J-2}, i.e., for every $R$-algebra $S$ of finite type, the
      regular locus in $\Spec(S)$ is open.
  \end{enumerate}
  A locally Noetherian scheme $X$ is \textsl{excellent} (resp.\
  \textsl{quasi-excellent}) if it admits an open affine covering $X = \bigcup_i
  \Spec(R_i)$ such that every $R_i$ is excellent (resp.\ quasi-excellent).
  \par A locally Noetherian scheme $X$ or a ring $R$ is \textsl{locally
  excellent} (resp.\ \textsl{locally quasi-excellent}) if all of its local rings
  are excellent (resp.\ quasi-excellent).
\end{definition}
Noetherian complete local rings are excellent, and the classes of excellent and
quasi-excellent
schemes are stable under morphisms locally essentially of finite type
\citeleft\citen{EGAIV2}\citemid Scholie 7.8.3\citepunct \citen{Mat80}\citemid
(34.A)\citeright.
\begin{remark}
  While excellence (resp.\ quasi-excellence) implies local excellence
  (resp.\ local quasi-excellence), the converse is false in general
  \cite[Example 1]{Hoc73}.
  In fact, by \cite[(34.A)]{Mat80}, a Noetherian ring is locally excellent
  (resp.\ locally quasi-excellent) if and only if it is a universally catenary
  G-ring (resp.\ it is a G-ring).
\end{remark}
\subsection{Dualizing complexes}
We will also need the notion of a dualizing complex.
\begin{definition}[{\citeleft\citen{Har66}\citemid Chapter V, Definition on p.\
  258\citepunct \citen{Con00}\citemid p.\ 118\citeright}]
  Let $X$ be a locally Noetherian scheme.
  A \textsl{dualizing complex} on $X$ is a complex $\omega_X^\bullet$ in
  $D^b_{\textup{coh}}(X)$ that has finite injective dimension such that the
  natural morphism
  \[
    \id_{(-)} \longrightarrow
    \RRHHom_{\cO_X}\bigl(\RRHHom_{\cO_X}(-,\omega_X^\bullet),
    \omega_X^\bullet\bigr)
  \]
  of $\delta$-functors on $D^b_{\textup{coh}}(X)$ is an isomorphism.
\end{definition}
\begin{remark}\label{rem:whendualizingexist}
  We will cite results about dualizing complexes and Grothendieck duality
  as we use them.
  To ensure the existence of dualizing complexes, we recall
  the following (see \cite[p.\ 299]{Har66}):
  \begin{enumerate}[label=$(\roman*)$,ref=\roman*]
    \item\label{rem:whendualizingexistft}
      If $f\colon X \to Y$ is a morphism of finite type between Noetherian
      schemes and $\omega_Y^\bullet$ is a dualizing complex for $Y$, then the
      exceptional pullback $f^!\omega_Y^\bullet$ is a dualizing complex for $X$.
    \item\label{rem:whendualizingexistcomplete}
      If $X$ is regular (or more generally, Gorenstein) of finite Krull
      dimension, then $\cO_X$ is a dualizing complex for $X$.
      Thus, combining $(\ref{rem:whendualizingexistft})$ with the Cohen
      structure theorem \cite[Theorem 29.4$(ii)$]{Mat89},
      spectra of Noetherian complete local rings have dualizing
      complexes.
  \end{enumerate}
  We also note that locally Noetherian schemes with dualizing complexes have
  finite Krull dimension and are universally catenary \cite[(1) and (2) on p.\
  300]{Har66}.
\end{remark}
\subsection{Relative ampleness conditions}\label{sect:relativeampleness}
We define relative ampleness conditions for invertible sheaves and
$\QQ$-Cartier divisors.
While most of these definitions exist in the literature, the definition of
 $f$-big is more general than what is usually used.
We have made this definition to facilitate our limit arguments in
\S\ref{sect:mainapprox}.
\begin{definition}[(see {\citeleft\citen{EGAII}\citemid D\'efinitions 4.4.2 and
  4.6.1\citepunct \citen{KMM87}\citemid Definitions 0-1-4, 0-3-2, and 0-1-1\citepunct
  \citen{Fuj17}\citemid \S\S2.1--2.2\citeright})]
  \label{def:fpositive}
  Let $f\colon X \to Y$ be a morphism of schemes and let $\sL$ be an invertible
  sheaf on $X$.
  \begin{enumerate}[label=$(\roman*)$,ref=\roman*]
    \item We say that $\sL$ is \textsl{$f$-very ample} if there exists a
      quasi-coherent $\cO_Y$-module $\sE$ and an immersion $i\colon X
      \hookrightarrow \PP(\sE)$ over $Y$ such that $\sL \simeq
      i^*(\cO_{\PP(\sE)}(1))$.
    \item Suppose $f$ is quasi-compact.
      We say that $\sL$ is \textsl{$f$-ample} if 
      there exists an open affine cover $Y = \bigcup_i U_i$ such that
      $\sL\rvert_{f^{-1}(U_i)}$ is ample for all $i$.
    \item We say that $\sL$ is \textsl{$f$-generated} if the
      adjunction morphism $f^*f_*\sL \to \sL$ is surjective.
      We say that $\sL$ is \textsl{$f$-semi-ample} if there exists an integer $n
      > 0$ such that $\sL^{\otimes n}$ is $f$-generated.
    \item\label{def:fbig} Suppose that $f$ is a proper maximally dominating morphism.
      We say that $\sL$ is \textsl{$f$-big} if there exists an integer $n > 0$
      such that for every maximal point $\eta \in Y$, the pullback
      $\sL^{\otimes n}_\eta$ of
      $\sL^{\otimes n}$ to the fiber $X_\eta$
      induces a rational map
      \[
        \begin{tikzcd}[column sep=5.75em]
          \phi_{\lvert \sL_\eta^{\otimes n} \rvert}\colon
          X_\eta \rar[dashed]{\lvert H^0(X_\eta,\sL^{\otimes n}_\eta) \rvert} &
          \PP\bigl( H^0(X_\eta,\sL^{\otimes n}_\eta) \bigr)
        \end{tikzcd}
      \]
      that is generically finite onto its image in the sense of \cite[Expos\'e
      II, Proposition 1.1.7]{ILO14}.
    \item Suppose that $f$ is a proper morphism.
      We say that $\sL$ is \textsl{$f$-nef} if $\sL\rvert_{f^{-1}(y)}$ is nef
      for every $y \in Y$, i.e., if for every one-dimensional integral closed
      subscheme $C \subseteq f^{-1}(y)$, we have
      \[
        \bigl(\sL\rvert_{f^{-1}(y)} \cdot C\bigr) \ge 0.
      \]
  \end{enumerate}
  We can extend these definitions to Cartier divisors $L$ on $X$ by asking
  that their associated invertible sheaves $\cO_X(L)$ satisfy these conditions.
  If $D$ is a $\QQ$-Cartier divisor, then we say that $D$ is
  \textsl{$f$-ample} (resp.\ \textsl{$f$-semi-ample}, \textsl{$f$-big},
  \textsl{$f$-nef}) if some positive integer multiple of $D$ satisfies this
  property.
\end{definition}
\begin{remark}
  When $X$ is integral,
  the definition for $f$-big in Definition \ref{def:fpositive}$(\ref{def:fbig})$
  is equivalent to saying that the volume of $\sL_\eta$ is positive for every
  maximal point $\eta \in Y$ by \cite[Theorems 8.2 and 10.7]{Cut14}.
\end{remark}
\subsection{Rational singularities and pseudo-rational rings}
We adopt the following definition for rational singularities.
\begin{definition}[(cf.\ {\citeleft\citen{KKMSD73}\citemid pp.\ 50--51\citepunct
  \citen{Kol13}\citemid Definition 2.76\citeright})]\label{def:ourratsings}
  Let $(R,\fm)$ be a quasi-excellent local $\QQ$-algebra.
  We say that $R$ has \textsl{rational singularities} if $R$ is normal and if
  for every proper birational morphism $f\colon X \to \Spec(R)$ such that $X$ is
  regular, we have $R^if_*\cO_X = 0$ for all $i > 0$.
  \par Now let $Y$ be a normal locally quasi-excellent locally Noetherian
  scheme of equal characteristic zero.
  We say that $Y$ has \textsl{rational singularities} if every local ring
  $\cO_{Y,y}$ has rational singularities.
  If $R$ is a locally quasi-excellent Noetherian ring, we say that $R$ has
  \textsl{rational singularities} if $\Spec(R)$ does.
\end{definition}
\par The condition in Definition \ref{def:ourratsings} is not vacuous since
resolutions of singularities exist for quasi-excellent local $\QQ$-algebras by
\cite[Chapter I, \S3, Main Theorem I$(n)$]{Hir64}.
We check that this definition localizes.
\begin{lemma}\label{lem:ourratsingslocal}
  Let $(R,\fm)$ be a quasi-excellent local $\QQ$-algebra with rational
  singularities.
  Then, $R_\fp$ has rational singularities for every prime ideal $\fp \subseteq
  R$, and hence $\Spec(R)$ has rational singularities.
\end{lemma}
\begin{proof}
  Since normality is a local condition, it suffices to
  consider the condition on proper birational morphisms.
  \par
  Let $f_\fp\colon X_\fp \to \Spec(R_\fp)$ be a proper birational morphism from a
  regular scheme $X_\fp$.
  We can then find a Cartesian diagram
  \[
    \begin{tikzcd}
      X_\fp \rar\dar[swap]{f_\fp} & X\dar{g}\\
      \Spec(R_\fp) \rar & \Spec(R)
    \end{tikzcd}
  \]
  where the horizontal morphisms are localizing immersions in the sense of
  \cite[Definition 2.7]{Nay09} and $g$ is proper by
  Nayak's version of Nagata compactification \cite[Theorem 4.1]{Nay09}.
  Set $\overline{X_\fp}$ to be the scheme-theoretic closure of $X_\fp$ in $X$.
  Since $X_\fp \hookrightarrow X$ is quasi-compact, the underlying set of
  $\overline{X_\fp}$ is equal to the set-theoretic closure of $X_\fp$ in $X$ by
  \cite[Corollaire 6.10.6(1)]{EGAInew}.
  The morphism $\overline{X_\fp} \to \Spec(R)$ is therefore birational.
  \par We now let $\pi_1\colon \tilde{X} \to X$ be a resolution of
  singularities that is an isomorphism along the set-theoretic image of $X_\fp$
  in $\overline{X_\fp}$, which exists by
  \cite[Chapter I, \S3, Main Theorem I$(n)$]{Hir64}.
  We then obtain the proper birational morphism
  \[
    f \coloneqq (g \circ \pi_1) \colon \tilde{X} \longrightarrow \Spec(R)
  \]
  from the regular scheme $\tilde{X}$ whose base change
  to $\Spec(R_\fp)$ is the morphism $f_\fp$.
  By assumption, we have $R^if_*\cO_{\tilde{X}} = 0$
  for all $i > 0$, and hence $R^if_{\fp*}\cO_{X_\fp} = 0$ for all $i > 0$ by
  flat base change.
\end{proof}
Even without assuming the existence of resolutions of singularities, we can make
the following definition:
\begin{definition}[{\cite[\S2]{LT81}}]\label{def:pseudorational}
  Let $(R,\fm)$ be a Noetherian local ring of dimension $d$.
  We say that $R$ is \textsl{pseudo-rational} if the following conditions are
  satisfied.
  \begin{enumerate}[label=$(\roman*)$,ref=\roman*]
    \item $R$ is normal.
    \item $R$ is Cohen--Macaulay.
    \item $R$ is \textsl{analytically unramified}, i.e., the $\fm$-adic completion
      $\hat{R}$ of $R$ is reduced.
    \item\label{def:pseudoratbiratcond}
      For every proper birational morphism $f\colon W \to \Spec(R)$
      with $W$ normal, denoting the closed fiber by $E = f^{-1}(\{\fm\})$,
      the natural map
      \[
        H^d_\fm(R) = H^d_{\{\fm\}}\bigl(\Spec(R),f_*\cO_W\bigr)
        \xrightarrow{\delta^d_f(\cO_W)}
        H^d_E(W,\cO_W)
      \]
      appearing as the edge map in the Leray--Serre spectral sequence for the
      composition of functors $\Gamma_{\{\fm\}} \circ f_* = \Gamma_E$
      is injective.
  \end{enumerate}
  If $X$ is a locally Noetherian scheme or $R$ is a Noetherian ring, we say that
  $X$ is \textsl{locally pseudo-rational} if all its local rings are
  pseudo-rational.
\end{definition}

\section{Limits and local cohomology}\label{sect:limits}
In this section, we review some preliminaries on 
limits of ringed spaces and schemes and their sheaf cohomology.
The main new result is that local cohomology is well-behaved under limits of
ringed spaces (Theorem \ref{lem:lccolimits}).
We set our conventions for limits of spaces in \S\ref{sect:limitsrssch} and
define Zariski--Riemann spaces as an example of such a limit
in \S\ref{sect:zariskiriemann}.
The behavior of sheaf cohomology under limits of ringed spaces is
reviewed in \S\ref{sect:cohlimits} and we deduce the analogous result for local
cohomology as a consequence in \S\ref{sect:lclimits}.
\subsection{Limits of ringed spaces and schemes}\label{sect:limitsrssch}
We fix our notation for limits of ringed spaces mostly following Fujiwara and
Kato \cite[Chapter 0, \S4]{FK18}, which in turn draws
on the topos-theoretic formulation of this material in \cite[Expos\'es VI and
VII]{SGA42}.
In the scheme-theoretic context, some of this material appears in
\cite[\S8]{EGAIV3}.
\begin{setup}[(see {\cite[Chapter 0, \S$4.1.(e)$ and \S$4.2.(a)$]{FK18}})]
  \label{setup:limitlrs}
  Let $\{(X_\lambda,v_{\lambda\mu})\}_{\lambda \in \Lambda}$
  be an inverse system of ringed spaces (resp.\ locally ringed spaces)
  indexed by a filtered preordered set $\Lambda$.
  By \cite[Chapter 0, Proposition 4.1.10]{FK18}, the limit
  \[
    X = \varprojlim_{\lambda \in \Lambda} X_\lambda
  \]
  exists in the category of ringed spaces (resp.\ locally ringed spaces), and
  is preserved under the forgetful functor from the category of locally ringed
  spaces to the category of ringed spaces.
  The underlying topological space of $X$ is
  the limit of the underlying topological spaces of the $X_\lambda$'s.
  The structure sheaf $\cO_X$ on $X$ can be described as follows.
  Denote by $v_{\lambda\mu}\colon X_\mu \to X_\lambda$ the transition morphisms
  in our inverse system.
  Note that such a transition morphism $v_{\lambda\mu}$
  yields a pullback map
  \[
    v_{\lambda\mu}^{-1}\cO_{X_\lambda} \longrightarrow \cO_{X_\mu}
  \]
  on structure sheaves, where $v_{\lambda\mu}^{-1}$ is the pullback
  as Abelian sheaves.
  Pulling back these maps to $X$ along the canonical projection morphisms
  $v_\lambda\colon X \to X_\lambda$, we obtain a direct
  system of Abelian sheaves $\{v_\lambda^{-1}\cO_{X_\lambda}\}_{\lambda \in
  \Lambda}$ on $X$.
  The structure sheaf on $X$ can then be described as the colimit
  \begin{align*}
    \cO_X &= \varinjlim_{\lambda \in \Lambda}v_\lambda^{-1}
    \cO_{X_\lambda}
    \intertext{of this direct system of Abelian sheaves on $X$.
    Moreover, by \cite[Chapter 0, Lemma 4.2.7]{FK18}
    this sheaf can be described as a colimit of
    $\cO_X$-modules:}
    \cO_X &\simeq \varinjlim_{\lambda \in \Lambda}
    v_\lambda^* \cO_{X_\lambda}.
  \end{align*}
\end{setup}
\begin{setup}[(see {\cite[\S8.2]{EGAIV3}})]\label{setup:limitsch}
  With notation as in Setup \ref{setup:limitlrs}, suppose the inverse
  system $\{(X_\lambda,v_{\lambda\mu})\}_{\lambda \in \Lambda}$ lies in the
  category of schemes.
  If the transition morphisms $v_{\lambda\mu}\colon X_\mu \to X_\lambda$ are
  affine for all $\lambda \le \mu$, then the limit $X$ in the category of
  locally ringed spaces is a scheme by \cite[Proposition 8.2.3 and Remarque
  8.2.14]{EGAIV3}, and the projection morphisms $v_\lambda\colon X \to
  X_\lambda$ are affine for all $\lambda \in \Lambda$.
\end{setup}
\begin{remark}\label{rem:directedvsfiltered}
  Our setup is more general than that in \cite{FK18} since our
  index sets $\Lambda$ are only assumed to be filtered and preordered.
  However, given an inverse system (resp.\ direct system) indexed
  by such a filtered preordered set $\Lambda$, there always exists a directed
  set $\Lambda'$ and an inverse system (resp.\ direct system) indexed by
  $\Lambda'$ using the same objects and morphisms as the
  original system, together with an initial (resp.\ final) morphism
  between the two inverse systems (resp.\ direct systems) by \cite[Theorem
  1]{AN82}.
  Since this morphism of inverse systems (resp.\ direct systems) is initial
  (resp.\ final), they have the same limit (resp.\ colimit).
  We will therefore allow ourselves to index inverse systems and direct systems
  by filtered preordered sets, and will state results from \cite{FK18} in this
  generality.
\end{remark}
In order for cohomology to behave well with respect to limits, we need to make
some additional assumptions on our inverse systems
$\{(X_\lambda,v_{\lambda\mu})\}_{\lambda \in \Lambda}$.
\begin{assumptions}\label{assum:fornicelimits}
  With notation as in Setup \ref{setup:limitlrs}, we will assume the following
  conditions hold.
  \begin{enumerate}[label=$(\alph*)$,ref=\alph*]
    \item\label{assum:spectralspaces}
      For every $\lambda \in \Lambda$, the underlying topological space of
      $X_\lambda$ is spectral.
      Following \cite[p.\ 43]{Hoc69},
      a topological space $Y$ is \textsl{spectral} if it is $T_0$ and
      quasi-compact, the quasi-compact open subsets in $Y$ are stable under
      finite intersection and form an open basis, and every nonempty irreducible
      closed subset in $Y$ has a generic point.
    \item\label{assum:qctransition}
      For all $\lambda \le \mu$, the underlying continuous maps of
      $v_{\lambda\mu}\colon X_\mu \to X_\lambda$ are quasi-compact.
  \end{enumerate}
  By \cite[Theorem 7]{Hoc69} (see also \cite[Chapter 0, Theorem 2.2.10(1)]{FK18}),
  these assumptions together imply that the underlying topological space of $X$ is
  spectral, and that the underlying continuous maps of the projection morphisms
  $v_\lambda\colon X \to X_\lambda$ are quasi-compact.
\end{assumptions}
\begin{remark}\label{rem:qcqsspectral}
  If $X_\lambda$ is a scheme as in Setup \ref{setup:limitsch}, then
  the underlying topological space of $X_\lambda$ is spectral
  if and only if $X_\lambda$ is
  quasi-compact and quasi-separated by \cite[Propositions 2.1.5 and 6.1.12]{EGAInew}
  (see also \cite[Chapter 0, Example 2.2.2(2)]{FK18}).
  Thus, if the $X_\lambda$ are all Noetherian (or more generally, quasi-compact and
  quasi-separated) schemes, then both
  $(\ref{assum:spectralspaces})$ and $(\ref{assum:qctransition})$ hold by
  \cite[Corollaire 6.1.13 and Proposition 6.1.5$(v)$]{EGAInew}.
\end{remark}
\begin{remark}
  A topological space is spectral if and only if it is coherent and sober
  \cite[Chapter 0, Remark 2.2.4(1)]{FK18}.
  Quasi-compact maps of spectral spaces are called \textsl{spectral} in
  \cite[p.\ 43]{Hoc69}.
\end{remark}
\subsection{Zariski--Riemann spaces}\label{sect:zariskiriemann}
An important example of a limit of an inverse system of schemes is
the \textsl{Zariski--Riemann} space
defined by Zariski for varieties \citeleft\citen{Zar40}\citemid Definition
A.II.5\citepunct \citen{Zar44}\citemid \S2\citeright\ and by Nagata for
Noetherian separated schemes \cite[\S3]{Nag63}.
\begin{definition}[(see {\citeleft\citen{FK06}\citemid Definition 5.9\citepunct
  \citen{Tem10}\citemid \S3.2\citepunct
  \citen{FK18}\citemid Chapter II, Definitions E.2.2 and E.2.3\citeright})]\label{def:zr}
  Let $X$ be a quasi-compact quasi-separated scheme.
  Denote by $\AId_X$ the set of quasi-coherent ideal sheaves $\sI$ of finite
  type on $X$ such that $X
  - V(\sI)$ contains all maximal points of $X$.
  The \textsl{Zariski--Riemann space} of $X$ is the limit
  \[
    \ZR(X) \coloneqq \varprojlim_{\sI \in \AId_X} X_\sI
  \]
  over the inverse system of blowups $X_\sI \to X$ along $\sI \in \AId_X$ in the
  category of locally ringed spaces.
  By Remark \ref{rem:qcqsspectral} and \cite[Theorem 7]{Hoc69}, 
  the underlying topological space of $\ZR(X)$ is
  spectral, and the underlying continuous maps of the projection morphisms
  $\ZR(X) \to X_\sI$ are quasi-compact.
\end{definition}
We note that the formation of Zariski--Riemann spaces commutes with
base change by quasi-compact separated \'etale morphisms by
\cite[\href{https://stacks.math.columbia.edu/tag/087B}{Tag
087B}]{stacks-project}.
\begin{remark}
  In \cite[\S3.2]{Tem10}, Temkin defines the Zariski--Riemann space for integral
  schemes using all proper birational morphisms $X' \to X$.
  If $X$ is quasi-compact and quasi-separated, then 
  the limit over such an inverse system coincides with $\ZR(X)$ since all
  proper birational morphisms can be dominated by a blowup along an ideal sheaf
  in $\AId_X$ by \cite[Premi\`ere partie, Corollaire 5.7.12]{RG71} (see also
  \cite[Theorem 2.11]{Con07}).
  We have chosen our definition to ensure that our inverse system is indexed by
  a directed set instead of a directed category.
\end{remark}
\subsection{Sheaf cohomology on limits of ringed spaces}\label{sect:cohlimits}
We will need to understand the behavior of sheaf cohomology on limits of ringed
spaces.
To do so, we set our notation for sheaves on inverse systems of ringed spaces.
Again, much of this material also appears in \cite[Expos\'es VI and VII]{SGA42}
in the language of topos theory.
\begin{setup}[(see {\cite[Chapter 0, \S4.4]{FK18}})]\label{setup:limitschlc}
  With notation as in Setup \ref{setup:limitlrs},
  for every $\lambda \in \Lambda$, we also fix an $\cO_{X_\lambda}$-module
  $\sF_\lambda$,
  together with morphisms
  \[
    \varphi_{\lambda\mu}\colon v_{\lambda\mu}^*\sF_\lambda \longrightarrow
    \sF_\mu
  \]
  of $\cO_{X_\mu}$-modules
  for every $\lambda \le \mu$, such that $\varphi_{\lambda\nu} =
  \varphi_{\mu\nu} \circ v^*_{\mu\nu}\varphi_{\lambda\mu}$ whenever $\lambda \le
  \mu \le \nu$.
  We then have a direct system $\{v_\lambda^*\sF_\lambda\}_{\lambda \in
  \Lambda}$ of $\cO_X$-modules whose colimit is the $\cO_X$-module
  \begin{align}
    \sF &= \varinjlim_{\lambda \in \Lambda} v_\lambda^*\sF_\lambda.\nonumber
    \intertext{We have the following canonical isomorphism from \cite[Chapter 0,
    Proposition 4.2.7]{FK18}:}
    \sF &\simeq \varinjlim_{\lambda \in \Lambda}
    v_\lambda^{-1}\sF_\lambda.\label{eq:fk0427}
  \end{align}
\end{setup}
We also make the following definition, which will simplify the statements of
some of the results below.
\begin{definition}
  Let $A$ be a ring, and consider the ringed space $(\{*\},A)$ whose underlying
  topological space is a point and whose structure sheaf is the constant sheaf
  $A$.
  The \textsl{category of ringed spaces over $A$} is the slice category of ringed
  spaces over $(\{*\},A)$.
\end{definition}
With this notation, we have the following statements about the behavior of sheaf
cohomology under limits of ringed spaces, which are special cases of
\cite[Expos\'e VI, Th\'eor\`eme 8.7.3]{SGA42}.
The terminology ``the Grothendieck limit theorem'' is from \cite[p.\ 169]{Pan03}.
\begin{theorem}[(The Grothendieck limit theorem
  {\cite[Chapter 0, Proposition 4.4.1]{FK18}})]\label{thm:fklimitscoh}
  Let $A$ be a ring.
  Let $\{(X_\lambda,v_{\lambda\mu})\}_{\lambda \in \Lambda}$
  be an inverse system of spectral ringed spaces over $A$
  indexed by a filtered preordered
  set $\Lambda$ with quasi-compact transition morphisms, and
  let
  \[
    X \coloneqq \varprojlim_{\lambda\in\Lambda} X_\lambda
  \]
  be the inverse limit of this inverse system with canonical projection
  morphisms $v_\lambda\colon X \to X_\lambda$.
  \par For each $\lambda \in \Lambda$, fix an
  $\cO_{X_\lambda}$-module $\sF_\lambda$ on each
  $X_\lambda$, together with morphisms
  \[
    \varphi_{\lambda\mu}\colon v_{\lambda\mu}^*\sF_\lambda \longrightarrow
    \sF_\mu
  \]
  of $\cO_{X_\mu}$-modules for every $\lambda \le \mu$, such that $\varphi_{\lambda\nu} =
  \varphi_{\mu\nu} \circ v^*_{\mu\nu}\varphi_{\lambda\mu}$ whenever $\lambda \le
  \mu \le \nu$.
  Consider the direct system $\{v_\lambda^*\sF_\lambda\}_{\lambda \in
  \Lambda}$ of $\cO_X$-modules whose colimit is the $\cO_X$-module
  \[
    \sF \coloneqq \varinjlim_{\lambda\in\Lambda} v_\lambda^*\sF_\lambda.
  \]
  Then, the canonical map
  \[
    \varinjlim_{\lambda \in \Lambda} H^i(X_\lambda,\sF_\lambda)
    \longrightarrow H^i(X,\sF)
  \]
  is an isomorphism of $A$-modules for all $i \ge 0$.
\end{theorem}
\begin{theorem}[{\cite[Chapter 0, Corollary 4.4.4]{FK18}}]\label{thm:fklimitsdirectimage}
  Let $\{(X_\lambda,v_{\lambda\mu})\}_{\lambda \in \Lambda}$ and
  $\{(Y_\lambda,w_{\lambda\mu})\}_{\lambda \in \Lambda}$
  be inverse systems of spectral ringed spaces
  indexed by a filtered preordered
  set $\Lambda$ with quasi-compact transition morphisms, and
  let
  \[
    X \coloneqq \varprojlim_{\lambda\in\Lambda} X_\lambda
    \qquad \text{and} \qquad
    Y \coloneqq \varprojlim_{\lambda\in\Lambda} Y_\lambda
  \]
  be the inverse limits of these inverse systems with canonical projection
  morphisms $v_\lambda\colon X \to X_\lambda$ and $w_\lambda\colon Y \to
  Y_\lambda$, respectively.
  Consider a system of morphisms $\{f_\lambda\colon X_\lambda \to
  Y_\lambda\}_{\lambda \in \Lambda}$ such that the diagrams
  \[
    \begin{tikzcd}
      X_\mu \rar{f_\mu}\dar[swap]{v_{\lambda\mu}}
      & Y_\mu\dar{w_{\lambda\mu}}\\
      X_\lambda \rar{f_\lambda} & Y_\lambda
    \end{tikzcd}
  \]
  commute for all $\lambda \le \mu$, and set $f =
  \varprojlim_{\lambda \in \Lambda}f_\lambda\colon X \to Y$.
  \par For each $\lambda \in \Lambda$, fix an
  $\cO_{X_\lambda}$-module $\sF_\lambda$ on each
  $X_\lambda$, together with morphisms
  \[
    \varphi_{\lambda\mu}\colon v_{\lambda\mu}^*\sF_\lambda \longrightarrow
    \sF_\mu
  \]
  of $\cO_{X_\mu}$-modules for every $\lambda \le \mu$, such that $\varphi_{\lambda\nu} =
  \varphi_{\mu\nu} \circ v^*_{\mu\nu}\varphi_{\lambda\mu}$ whenever $\lambda \le
  \mu \le \nu$.
  Consider the direct system $\{v_\lambda^*\sF_\lambda\}_{\lambda \in
  \Lambda}$ of $\cO_X$-modules whose colimit is the $\cO_X$-module
  \[
    \sF \coloneqq \varinjlim_{\lambda\in\Lambda} v_\lambda^*\sF_\lambda.
  \]
  Then, the canonical morphism
  \[
    \varinjlim_{\lambda \in \Lambda} w_\lambda^{-1}R^if_{\lambda*}(\sF_\lambda)
    \longrightarrow R^if_*(\sF)
  \]
  is an isomorphism of $\cO_Y$-modules for all $i \ge 0$.
\end{theorem}
\subsection{Local cohomology on limits of ringed spaces}\label{sect:lclimits}
We now show that local cohomology is well-behaved under limits.
See \cite[Expos\'e VI, Corollaire 5.5]{SGA42} and \cite[Lemma 5.16]{HO08} for related
results.
\begin{theorem}\label{lem:lccolimits}
  Let $A$ be a ring.
  Let $\{(X_\lambda,v_{\lambda\mu})\}_{\lambda \in \Lambda}$
  be an inverse system of spectral ringed spaces over $A$
  indexed by a filtered preordered
  set $\Lambda$ with quasi-compact transition morphisms, and
  let
  \[
    X \coloneqq \varprojlim_{\lambda\in\Lambda} X_\lambda
  \]
  be the inverse limit of this inverse system with canonical projection
  morphisms $v_\lambda\colon X \to X_\lambda$.
  \par For each $\lambda \in \Lambda$, fix an
  $\cO_{X_\lambda}$-module $\sF_\lambda$ on each
  $X_\lambda$, together with morphisms
  \[
    \varphi_{\lambda\mu}\colon v_{\lambda\mu}^*\sF_\lambda \longrightarrow
    \sF_\mu
  \]
  of $\cO_{X_\mu}$-modules for every $\lambda \le \mu$, such that $\varphi_{\lambda\nu} =
  \varphi_{\mu\nu} \circ v^*_{\mu\nu}\varphi_{\lambda\mu}$ whenever $\lambda \le
  \mu \le \nu$.
  Consider the direct system $\{v_\lambda^*\sF_\lambda\}_{\lambda \in
  \Lambda}$ of $\cO_X$-modules whose colimit is the $\cO_X$-module
  \[
    \sF \coloneqq \varinjlim_{\lambda\in\Lambda} v_\lambda^*\sF_\lambda.
  \]
  \par Fix $\alpha \in \Lambda$, a quasi-compact open subset $U_\alpha \subseteq
  X_\alpha$, and a closed subset $Z_\alpha \subseteq U_\alpha$.
  For each $\lambda \ge \alpha$, set $U_\lambda =
  v_{\alpha\lambda}^{-1}(U_\alpha)$ and $Z_\lambda =
  X_\lambda - U_\lambda$.
  Then, the canonical map 
  \[
    \varinjlim_{\lambda \ge \alpha} H^i_{Z_\lambda}(X_\lambda,\sF_\lambda)
    \longrightarrow H^i_Z(X,\sF)
  \]
  is an isomorphism of $A$-modules for all $i \ge 0$.
\end{theorem}
\begin{proof}
  By Excision \cite[Proposition 1.3]{Gro67}, we may replace $X$ by $U$ to assume that
  $Z$ is closed in $X$.
  Consider $\lambda,\mu \in \Lambda$ such that $\lambda \le \mu$.
  By \cite[Functoriality II.9.7]{Ive86} applied to the maps $v_\lambda\colon
  X \to X_\lambda$, $v_\mu\colon X \to X_\mu$, and $v_{\lambda\mu}\colon X_\mu
  \to X_\lambda$, there is a commutative diagram
  \[
    \begin{tikzcd}
      \cdots \rar
      & H^i_{Z_\lambda}(X_\lambda,\sF_\lambda) \rar\dar
      & H^i(X_\lambda,\sF_\lambda) \rar\dar
      & H^i\bigl(U_\lambda,\sF_\lambda\rvert_{U_\lambda}\bigr) \rar\dar
      & \cdots\\
      \cdots \rar
      & H^i_{Z_\mu}(X_\mu,v_{\lambda\mu}^{-1}\sF_\lambda) \rar\dar
      & H^i(X_\mu,v_{\lambda\mu}^{-1}\sF_\lambda) \rar\dar
      & H^i\bigl(U_\mu,(v_{\lambda\mu}^{-1}\sF_\lambda)\rvert_{U_\mu}\bigr) \rar\dar
      & \cdots\\
      \cdots \rar
      & H^i_{Z_\mu}(X_\mu,\sF_\mu) \rar\dar
      & H^i(X_\mu,\sF_\mu) \rar\dar
      & H^i\bigl(U_\mu,\sF_\mu\rvert_{U_\mu}\bigr) \rar\dar
      & \cdots\\
      \cdots \rar
      & H^i_Z(X,v_\mu^{-1}\sF_\mu) \rar\dar
      & H^i(X,v_\mu^{-1}\sF_\mu) \rar\dar
      & H^i\bigl(U,(v_\mu^{-1}\sF_\mu)\rvert_U\bigr) \rar\dar
      & \cdots\\
      \cdots \rar
      & H^i_Z(X,\sF) \rar
      & H^i(X,\sF) \rar
      & H^i\bigl(U,\sF\rvert_U\bigr) \rar
      & \cdots
    \end{tikzcd}
  \]
  of Abelian groups with exact rows, where the vertical arrows in the top and third
  rows are induced by pulling back along $v_{\lambda\mu}$ and $v_\mu$
  and the vertical arrows in the second and
  bottom rows are obtained from the maps
  \[
    v_{\lambda\mu}^{-1}\sF_\lambda \longrightarrow v_{\lambda\mu}^*\sF_\lambda
    \xrightarrow{\varphi_{\lambda\mu}} \sF_\mu
  \]
  and the description of $\sF$ as a colimit of Abelian sheaves as in
  \eqref{eq:fk0427}.
  The commutative diagram is in fact a commutative diagram of $A$-modules by the
  argument in \cite[Chapter 0, \S4.3.$(c)$]{FK18}.
  Now taking colimits over all $\mu \ge \lambda \ge \alpha$, the middle and right
  arrows in the diagram yield isomorphisms of $A$-modules by Theorem
  \ref{thm:fklimitscoh}.
  The five lemma \cite[Chapter I, Proposition 1.1]{CE56} then implies the desired
  isomorphisms.
\end{proof}

\begingroup
\makeatletter
\renewcommand{\@secnumfont}{\bfseries}
\part{Relative vanishing and injectivity theorems}\label{part:vanishing}
\makeatother
\endgroup

\section{Approximating morphisms of schemes}\label{sect:mainapprox}
As outlined in \S\ref{sect:intro}, the idea in our proof of Theorem
\ref{thm:mainvanishing} is to approximate the morphism $f\colon X \to Y$ by
morphisms of varieties over $\QQ$.
Since this approximation construction takes up the bulk of the proof of Theorem
\ref{thm:mainvanishing}, we state and prove it separately below.
In \S\ref{sect:approxpositivity}, we prove that many ampleness conditions
on invertible sheaves behave well under limits.
We prove our approximation result (Lemma \ref{lem:approx}) in
\S\ref{sect:proofmainapprox}.
\subsection{Relative ampleness conditions and
limits}\label{sect:approxpositivity}
We prove that all ampleness conditions defined in Definition
\ref{def:fpositive} (except for $f$-nefness) behave well
under limits.
\begin{lemma}[(cf.\ {\cite[Lemme 8.10.5.2]{EGAIV3}})]\label{lem:positivitylimits}
  Let $\{S_\lambda\}_{\lambda \in \Lambda}$ be an inverse
  system of quasi-compact quasi-separated schemes with affine transition
  morphisms and limit $S$.
  Let $f_\alpha\colon X_\alpha \to Y_\alpha$ be a morphism of
  $S_\alpha$-schemes of finite presentation for some $\alpha \in \Lambda$, and
  let $\sL_\alpha$ be an invertible sheaf on $X_\alpha$.
  For every $\lambda \ge \alpha$, let
  \[
    f_\lambda \colon X_\lambda \longrightarrow Y_\lambda
  \]
  be the base change of $f_\alpha$ along $S_\lambda \to
  S_\alpha$, and denote by $\sL_\lambda$ the pullback of $\sL_\alpha$ to
  $X_\lambda$.
  Denote by $f\colon X \to Y$ the limit of the morphisms $f_\lambda$, and denote
  by $\sL$ the pullback of $\sL_\alpha$ to $X$.
  \begin{enumerate}[label=$(\roman*)$,ref=\roman*]
    \item\label{lem:positivitylimitsample}
      If $\sL$ is $f$-very ample (resp.\ $f$-ample), then
      there exists an index $\lambda \in \Lambda$ such that $\sL_\mu$ is
      $f_\mu$-very ample (resp.\ $f_\mu$-ample) for all
      $\mu \ge \lambda$.
    \item\label{lem:positivitylimitssemiample}
      Suppose $f_\alpha$ is quasi-separated.
      If $\sL$ is $f$-generated (resp.\ $f$-semi-ample), then
      there exists an index $\lambda \in \Lambda$ such that $\sL_\mu$ is
      $f_\mu$-generated (resp.\ $f_\mu$-semi-ample) for all
      $\mu \ge \lambda$.
    \item\label{lem:positivitylimitsbig}
      Suppose $f$ is a proper morphism of schemes.
      If $\sL$ is $f$-big, then
      there exists an index $\lambda \in \Lambda$ such that $\sL_\mu$ is
      $f_\mu$-big for all $\mu \ge \lambda$.
  \end{enumerate}
\end{lemma}
\begin{proof}
  The statement $(\ref{lem:positivitylimitsample})$ follows from \cite[Lemme
  8.10.5.2]{EGAIV3} since very ampleness (resp.\ ampleness) is stable
  under base change \cite[Propositions 4.4.10$(iii)$ and 4.6.13$(iii)$]{EGAII}.\smallskip
  \par For $(\ref{lem:positivitylimitssemiample})$, after replacing $\sL_\alpha$
  by a positive integer power, it suffices to show the $f$-generated case.
  For all $\mu \ge \alpha$, the pullback of the morphism $f_\mu^*f_{\mu*}\sL_\mu
  \to \sL_\mu$ is
  \[
    f^*w_\mu^*f_{\mu*}\sL_\mu =
    v_\mu^*f_\mu^*f_{\mu*}\sL_\mu \longrightarrow
    v_\mu^*\sL_\mu
  \]
  where $w_\mu\colon Y \to Y_\mu$ and $v_\mu\colon X \to X_\mu$
  are the canonical projection morphisms.
  The colimit of these morphisms is the adjunction morphism
  $f^*f_*\sL \to \sL$, since the left
  adjoint $f^*$ commutes with colimits and then by applying Theorem
  \ref{thm:fklimitsdirectimage} together with the
  isomorphism \cite[Chapter 0, Proposition 4.2.7]{FK18} that allows us to replace
  $w_\mu^{-1}$ with $w_\mu^*$.
  Now since $f_{\mu*}\sL_\mu$ is quasi-coherent \cite[Proposition
  6.7.1]{EGAInew}, we can apply \cite[Corollaire 8.5.7]{EGAIV3} to say that there
  exists $\lambda \in \Lambda$ such that
  \[
    f_\mu^*f_{\mu*}\sL_\mu \longrightarrow v_\mu^*\sL_\mu
  \]
  is surjective for all $\mu \ge \lambda$, as required.\smallskip
  \par We now show $(\ref{lem:positivitylimitsbig})$.
  We first note that the $f_\mu$ are proper 
  for
  large enough $\mu$ \cite[Th\'eor\`eme 8.10.5$(xii)$]{EGAIV3}.
  Moreover, by \cite[Proposition 8.4.2$(a)$$(i)$]{EGAIV3}, the morphisms $Y \to
  Y_\mu$ and $X \to X_\mu$ induce bijections on maximal points for large enough
  $\mu$.
  Note that $Y$ has only finitely many maximal points by the
  quasi-compactness of $S$.
  Thus, choosing $\lambda \ge \alpha$ large enough we may assume that
  the morphisms $X_\mu \to Y_\mu$ are proper and maximally dominating
  for all $\mu \ge \lambda$, and we
  may replace $Y_\mu$ by the spectra $\Spec(\kappa(\eta_\mu))$ of
  its residue fields at maximal points to assume that the $Y_\mu$ are
  spectra of fields $k_\mu$ with colimit $k$.
  \par Now let $n > 0$ be an integer such that $\sL^{\otimes n}$ induces a
  generically finite morphism onto its image.
  Then, there exists an open subset $U \subseteq \PP_{k}(H^0(X,\sL^{\otimes n}))$
  such that the rational map
  \[
    \begin{tikzcd}[column sep=5.5em]
      \phi_{\lvert \sL^{\otimes n} \rvert}
      \colon X \rar[dashed]{\lvert H^0(X,\sL^{\otimes n}) \rvert}
      & \PP_{k}\bigl( H^0(X,\sL^{\otimes n}) \bigr)
    \end{tikzcd}
  \]
  induced by $\sL^{\otimes n}$ restricts to a finite morphism over $U$.
  By \cite[Corollaire 8.6.4, Th\'eor\`eme 8.8.2$(i)$, and Corollaire
  8.8.2.5]{EGAIV3}, \cite[(4.1.3) and (4.2.10)]{EGAII}, and flat base change,
  after possibly replacing $\lambda$ by a larger index, we may assume there 
  exists an open subset
  \[
    U_\lambda \subseteq
    \PP_{k_\lambda}\bigl(H^0(X_\lambda,\sL_\lambda^{\otimes n})\bigr)
  \]
  such that
  $v_\lambda^{-1}(U_\lambda) = U$, and such that the rational map $\phi_{\lvert
  \sL^{\otimes n} \rvert}$ restricted to $U$ is the base change of the rational
  maps
  \[
    \begin{tikzcd}[column sep=5.875em]
      \phi_{\lvert \sL_\mu^{\otimes n} \rvert}
      \colon X_\mu \rar[dashed]{\lvert H^0(X_\mu,\sL_\mu^{\otimes n}) \rvert}
      & \PP_{k_\mu}\bigl( H^0(X_\mu,\sL_\mu^{\otimes n}) \bigr)
    \end{tikzcd}
  \]
  restricted to $U_\mu = v_{\lambda\mu}^{-1}(U_\lambda)$ for all $\mu
  \ge \lambda$.
  Moreover, we may assume that the maps $\phi_{\lvert \sL_\mu^{\otimes n}
  \rvert}$ restricted to $U_\lambda$ are finite for all $\mu \ge \lambda$
  by \cite[Th\'eor\`eme 8.10.5$(x)$]{EGAIV3}.
  We therefore see that $\phi_{\lvert \sL_\mu^{\otimes n} \rvert}$ induces a
  generically finite morphism onto its image for all $\mu \ge \lambda$, and
  hence $\sL_\mu$ is $f_\mu$-big for all $\mu \ge \lambda$.
\end{proof}
\subsection{The approximation lemma}\label{sect:proofmainapprox}
We now show our main approximation result.
\begin{lemma}\label{lem:approx}
  Let $\kk$ be a Noetherian ring and let
  $(R,\fm)$ be an integral Noetherian local $\kk$-algebra.
  Consider a proper surjective morphism $f\colon X \to \Spec(R)$ from an
  integral scheme $X$.
  Write $R$ as the colimit
  \begin{equation}\label{lem:approxrj}
    R \simeq \varinjlim_{\lambda \in \Lambda} R_\lambda
  \end{equation}
  of a direct system of sub-$\kk$-algebras of finite type indexed by
  a directed set $\Lambda$ and partially ordered by inclusion.
  We then have the following:
  \begin{enumerate}[label=$(\roman*)$,ref=\roman*]
    \item\label{lem:approxxj} There exists $\alpha \in \Lambda$ and a proper
      surjective morphism $f'_\alpha\colon X'_\alpha \to \Spec(R_\alpha)$ from a
      reduced scheme $X'_\alpha$ for which the diagram
      \[
        \begin{tikzcd}
          X \rar{f}\dar[swap]{v_\alpha'} & \Spec(R)\dar\\
          X'_\alpha \rar{f'_\alpha} & \Spec(R_\alpha)
        \end{tikzcd}
      \]
      is Cartesian.
      Moreover, $\alpha$ can be chosen such that
      denoting by $f'_\lambda\colon X'_\lambda \to \Spec(R_\lambda)$ the base
      change of $f'_\alpha$ to $\Spec(R_\lambda)$, there exist integral closed
      subschemes $X_\lambda \subseteq X'_\lambda$ for all $\lambda \ge \alpha$
      such that the following hold:
      \begin{itemize}
        \item Setting $\fm_\lambda \coloneqq \fm \cap R_\lambda$, we have
          \[
            \dim(X) \le \dim\bigr(X_\lambda \otimes_{R_\lambda}
            (R_\lambda)_{\fm_\lambda}\bigr)
          \]
          for all $\lambda \ge \alpha$.
        \item The limit of the morphisms 
          \[
            f_\lambda\colon X_\lambda \hooklongrightarrow X_\lambda'
            \overset{f'_\lambda}{\longrightarrow} \Spec(R_\lambda)
          \]
          with transition morphisms $v_{\lambda\mu}\colon X_\mu \to X_\lambda$
          is the morphism
          $f\colon X \to \Spec(R)$.
      \end{itemize}
    \item\label{lem:approxlj}
      Let $\sL$ be an invertible sheaf on $X$.
      Then, after possibly replacing $\alpha$ with a larger index, we can write
      \[
        \sL \simeq v_\alpha^*\sL_\alpha
      \]
      for an invertible sheaf $\sL_\alpha$ on $X_\alpha$, where
      $v_\alpha\colon X \to X_\alpha$ are the canonical projection morphisms.
      Moreover, if $\sL$ is $f$-very ample (resp.\ $f$-ample, $f$-semi-ample,
      $f$-big), then we
      may assume that the invertible sheaves $\sL_\lambda \coloneqq
      v_{\alpha\lambda}^*\sL_\alpha$ are $f_\lambda$-very ample (resp.\ $f_\lambda$-ample,
      $f_\lambda$-semi-ample, $f_\lambda$-big) for all $\lambda \ge \alpha$.
    \item\label{lem:approxwj}
      For each $\lambda \ge \alpha$,
      the inverse system
      \begin{equation}\label{eq:inversesystemwithsmooth}
        \Set[\big]{g_{\lambda,p}\colon W_{\lambda,p} \longrightarrow
        X_\lambda}_{p \in P_\lambda}
      \end{equation}
      of all projective birational morphisms from integral schemes
      $W_{\lambda,p}$ that are separated and of finite type over $\kk$
      such that $f_\lambda \circ g_{\lambda,p}$ is projective is
      nonempty and indexed by a directed set $P_\lambda$.
      Moreover, if projective resolutions of singularities (resp.\
      normalizations, Macaulayfications) exist for all
      integral schemes that are separated and of finite type over $\kk$,
      then we may assume that the schemes $W_{\lambda,p}$ are regular (resp.\
      normal, Cohen--Macaulay).
    \item\label{lem:approxlimit} Consider the set
      \[
        J = \bigsqcup_{\lambda \in \Lambda} P_\lambda
      \]
      with the preorder where
      $(\lambda,p) \le (\mu,q)$ if and only $\lambda \le \mu$ and
      the morphism $g_{\mu,q}$ fits into a commutative diagram
      \[
        \begin{tikzcd}
          W_{\mu,q} \rar{g_{\mu,q}} \dar
          & X_\mu \rar \dar{v_{\lambda\mu}}
          & \Spec(R_\mu) \dar\\
          W_{\lambda,p} \rar{g_{\lambda,p}} & X_\lambda \rar &
          \Spec(R_\lambda)\mathrlap{.}
        \end{tikzcd}
      \]
      Then, the set $J$ is filtered, and
      the morphism $\pi\colon \ZR(X) \to X$ of locally ringed spaces from
      the Zariski--Riemann space of $X$ is the limit of the inverse
      systems \eqref{eq:inversesystemwithsmooth} as $\lambda \in \Lambda$ also
      varies.
  \end{enumerate}
\end{lemma}
\begin{remark}\label{rem:maincasesforapprox}
  Let $\kk$ be a quasi-excellent Noetherian $\QQ$-algebra, in which case
  projective resolutions of singularities exist by
  \cite[Theorem 1.1]{Tem08}.
  In this case, given $f\colon X \to \Spec(R)$ as above, we have the commutative
  diagram
  \[
    \begin{tikzcd}
      \ZR(X) \rar{\pi}\dar & X \rar{f}\dar{v_\lambda} & \Spec(R)\dar\\
      W_{\lambda,p} \rar{g_{\lambda,p}} & X_\lambda \rar{f_\lambda} &
      \Spec(R_\lambda)
    \end{tikzcd}
  \]
  of locally ringed spaces, where all but $\ZR(X)$ are Noetherian
  $\kk$-schemes and the
  schemes in the bottom row are integral schemes that are
  separated and of finite type over $\kk$ with $W_{\lambda,p}$ regular,
  such that the morphisms in the top row are the limits of the morphisms in the
  bottom row.
  We can also localize the $R_\lambda$ at $\fm_\lambda = \fm \cap R_\lambda$
  without affecting the inverse limit (since the inverse systems satisfy the
  same universal property) in order to assume that the $R_\lambda$
  are local (although the schemes in the bottom row of the diagram above
  are now \emph{essentially} of finite type over $\kk$).
  \par The necessary normalizations exist in $(\ref{lem:approxwj})$
  when $\kk$ is a Nagata ring in the sense of \cite[Definition 31.A]{Mat80}, and
  the necessary Macaulayfications exist in $(\ref{lem:approxwj})$ when $\kk$ is
  CM-quasi-excellent in the sense of \cite[Definition 1.2]{Ces} by \cite[Remark
  1.4, Theorem 5.3, and Remark 5.4]{Ces}.
\end{remark}
\begin{proof}[Proof of Lemma \ref{lem:approx}]
  \par We construct the morphisms in $(\ref{lem:approxxj})$.
  The first part of the statement follows from \cite[Th\'eor\`eme
  8.8.2$(ii)$]{EGAIV3}.
  By \cite[Proposition 8.7.2 and Th\'eor\`eme 8.10.5$(vi)$,\allowbreak$(xii)$]{EGAIV3}, after possibly
  replacing $\alpha$ by a larger index, we may assume that $X'_\lambda$ is
  reduced and that $f'_\lambda$ is proper and surjective for all $\lambda \ge
  \alpha$.
  Denote by $\eta$ and $\eta_\lambda$ the generic points of $Y$ and $Y_\lambda$,
  respectively.
  By transitivity of fibers \cite[Corollaire 3.4.9]{EGAInew} and applying
  \cite[(4.4.1)]{EGAIV2}, the generic fibers
  $f_\lambda^{\prime-1}(\eta_\lambda)$ are also irreducible.
  \par Next, we show that we can replace the morphisms $f'_\lambda$ by some
  $f_\lambda\colon X_\lambda \to \Spec(R_\lambda)$ for integral closed schemes
  $X_\lambda \subseteq X'_\lambda$.
  For each $\lambda \ge \mu \ge \alpha$, we will construct the following
  commutative diagram:
  \begin{equation}\label{eq:constructingxlambdas}
    \begin{tikzcd}
      X \rar[equal] \dar[swap]{v_\mu} & X \rar{f} \dar{v'_\mu} & \Spec(R) \dar\\
      X_\mu \rar[hook] \dar[swap]{v_{\lambda\mu}} & X'_\mu \rar{f'_\mu}
      \dar{v'_{\lambda\mu}}
      & \Spec(R_\mu) \dar\\
      X_\lambda \rar[hook] & X'_\lambda \rar{f'_\lambda} &
      \Spec(R_\lambda)\mathrlap{.}
    \end{tikzcd}
  \end{equation}
  Here, the squares in the right column are Cartesian.
  The scheme $X_\lambda$ is the scheme-theoretic closure of
  $f_\lambda^{\prime-1}(\eta_\lambda)$, which coincides with the set-theoretic
  closure with reduced scheme structure since $X_\lambda'$ is reduced
  \cite[Corollaire 6.10.6]{EGAInew}.
  We define $X_\mu$ in a similar fashion.
  These schemes $X_\lambda$ and $X_\mu$ are irreducible by \cite[Chapitre 0,
  Proposition 2.1.13]{EGAInew}, hence integral.
  The morphisms in the rightmost column induce bijections
  on generic points for all $\mu \ge \lambda \ge \alpha$, since all rings in the
  direct system \eqref{lem:approxrj} are domains.
  Thus, by transitivity of scheme-theoretic images \cite[Proposition
  6.10.3]{EGAInew}, morphisms in the leftmost column exist in a way that makes
  the squares in left column commute.
  Now for every $\lambda \ge \alpha$, consider the composition
  \[
    f_\lambda \colon X_\lambda \hooklongrightarrow X'_\lambda
    \overset{f'_\lambda}{\longrightarrow} \Spec(R_\lambda).
  \]
  This morphism is proper since it is the composition of proper morphisms.
  The base change of $f_\lambda$ to the generic point is
  \begin{equation}\label{eq:isoalonggenfib}
    f_\lambda^{-1}(\eta_\lambda) \overset{\sim}{\longrightarrow}
    f_\lambda^{\prime-1}(\eta_\lambda) \longrightarrow
    \Spec\bigl(\kappa(\eta_\lambda)\bigr),
  \end{equation}
  since $X_\lambda$ is the set-theoretic closure of
  $f_\lambda^{\prime-1}(\eta_\lambda)$ in $X'_\lambda$, and hence $X_\lambda$
  and $X'_\lambda$ are isomorphic over the generic point $\eta_\lambda$.
  Thus, we see that the morphism $f_\lambda$ is surjective with irreducible
  generic fiber.
  By \cite[Lemme 13.1.2]{EGAIV3}, the squares
  \[
    \begin{tikzcd}
      X \rar{f}\dar[swap]{v_\lambda} & \Spec(R)\dar\\
      X_\lambda \rar{f_\lambda} & \Spec(R_\lambda)
    \end{tikzcd}
  \]
  are Cartesian for all $\lambda \ge \alpha$.
  We therefore see that $f\colon X \to \Spec(R)$ satisfies the universal
  property for the limit of the morphisms $f_\lambda$.
  \par Finally, to ensure that
  \[
    \dim(X) \le \dim\bigl(X_\lambda \otimes_{R_\lambda}
    (R_\lambda)_{\fm_\lambda}\bigr)
  \]
  for all $\lambda \ge
  \alpha$, we choose a maximal chain
  \begin{equation}\label{eq:chainofclosed}
    Z_0 \subsetneq Z_1 \subsetneq \cdots \subsetneq Z_{\dim(X)} = X
  \end{equation}
  of irreducible closed subsets in $X$, which exists since $\dim(X) < \infty$ by
  \cite[Corollaire 5.6.6]{EGAIV2}.
  Since limits commute with fiber products and since
  \[
    R \simeq \varinjlim_{\lambda \ge \alpha} (R_\lambda)_{\fm_\lambda}
  \]
  by \cite[Chapitre 0, Proposition 6.1.6$(ii)$]{EGAInew}, we have
  \[
    X \simeq \varprojlim_{\lambda \ge \alpha} \bigl(X_\lambda
    \otimes_{R_\lambda} (R_\lambda)_{\fm_\lambda}\bigr)
  \]
  since they satisfy the same universal property.
  Next, \cite[Proposition 8.6.3]{EGAIV3}
  says that the partially ordered set of closed subschemes in $X$ is the colimit
  of the partially ordered sets of closed subschemes in $X_\lambda
  \otimes_{R_\lambda} (R_\lambda)_{\fm_\lambda}$ as $\lambda
  \in \Lambda$ varies.
  Thus, after possibly replacing $\alpha$ by a larger index, we may assume that the
  chain \eqref{eq:chainofclosed}
  is the preimage of a chain of closed subsets in
  \[
    X_\alpha \otimes_{R_\alpha} (R_\alpha)_{\fm_\alpha}.
  \]
  Moreover, since the chain \eqref{eq:chainofclosed} is a chain of \emph{strict}
  inclusions of \emph{irreducible} closed subsets, we can apply
  \cite[Proposition 8.6.3]{EGAIV3} again to say that after possibly replacing
  $\alpha$ by a larger index, we have a chain
  \[
    Z_{0,\alpha} \subsetneq Z_{1,\alpha} \subsetneq \cdots \subsetneq
    Z_{\dim(X),\alpha} = X_\alpha \otimes_{R_\alpha}
    (R_\alpha)_{\fm_\alpha}
  \]
  with strict inclusions whose preimage in $X$ is the chain
  \eqref{eq:chainofclosed}, and that each closed subset in this chain is
  irreducible \cite[Proposition 8.4.2$(a)(i)$]{EGAIV3}.
  Since this chain of inclusions must base change to a chain of strict
  inclusions in $X$, we see that the preimage of this chain in
  \[
    X_\lambda \otimes_{R_\lambda} (R_\lambda)_{\fm_\lambda}
    \simeq X_\alpha \otimes_{R_\alpha} (R_\lambda)_{\fm_\lambda}
  \]
  is
  still a chain of closed subsets with strict inclusions for each
  $\lambda \ge \alpha$, which are still irreducible by
  \cite[Proposition 8.4.2$(a)(i)$]{EGAIV3}.\smallskip
  \par We now show $(\ref{lem:approxlj})$.
  By \cite[Th\'eor\`eme 8.5.2$(ii)$]{EGAIV3}, after possibly replacing
  $\alpha$ with a larger index, there exists a coherent sheaf $\sL'_\alpha$ on
  $X'_\alpha$ such that
  $v_\alpha^{\prime*}\sL'_\alpha \simeq \sL$.
  By \cite[Proposition 8.5.5]{EGAIV3}, after possibly replacing $\alpha$ by a larger
  index again, we may assume that the inverse image $\sL'_\lambda \coloneqq
  v_{\alpha\lambda}^{\prime*}\sL'_\alpha$ on $X'_\lambda$
  is invertible for all $\lambda \ge \alpha$.
  We now set
  \[
    \sL_\lambda \coloneqq \sL_\lambda'\rvert_{X_\lambda},
  \]
  which is invertible and  satisfies 
  $\sL \simeq v_\alpha^{\prime*}\sL'_\alpha \simeq v_\alpha^*\sL_\alpha$
  by the commutativity of the squares in the left column of
  \eqref{eq:constructingxlambdas}.
  \par We now show that if $\sL$ is $f$-very ample (resp.\ $f$-ample,
  $f$-semi-ample, $f$-big), then we may assume the same holds for
  $\sL_\lambda$ for all $\lambda \ge \alpha$.
  This holds for $\sL'_\lambda$ instead of $\sL_\lambda$ by Lemma
  \ref{lem:approx}.
  Restricting $\sL'_\lambda$ to $X_{\lambda}'$ preserves these properties in
  each case by \cite[Propositions 4.4.10$(i\,\textit{bis})$ and
  4.6.13$(i\,\textit{bis})$]{EGAII} for $f$-very ample and $f$-ample,
  \cite[Lemma 2.11$(i)$]{CT20} and its proof for $f$-generated and
  $f$-semi-ample, and the fact that $X_\lambda \hookrightarrow X_\lambda'$
  induces an
  isomorphism over the generic point $\eta_\lambda$ by construction for
  $f$-big (see \eqref{eq:isoalonggenfib}).\smallskip
  \par Next, we show $(\ref{lem:approxwj})$.
  For each $\lambda \in \Lambda$, consider the inverse system
  \begin{equation}\label{eq:inversesystemwithnonsmooth}
    \Set[\big]{g_{\lambda,p'}\colon W_{\lambda,p'} \longrightarrow
    X_\lambda}_{p' \in P'_\lambda}
  \end{equation}
  of all projective birational morphisms where the $W_{\lambda,p'}$ are integral
  schemes that are separated and of finite type over $\kk$.
  The inverse system \eqref{eq:inversesystemwithsmooth}
  is nonempty since it contains the identity and is indexed
  by a subset of $\AId_{X_\lambda}$ since every projective birational morphism to
  $X_\lambda$ is a blowup \cite[Corollaire 2.3.7]{EGAIII1}.
  This inverse system is coinitial with the inverse system
  \eqref{eq:inversesystemwithsmooth} by Chow's lemma \cite[Corollaire
  5.6.2]{EGAII}.
  \par If projective resolutions of singularities (resp.\ normalizations,
  Macaulayfications) exist for all integral schemes that are separated and of
  finite type over $\kk$, then the inverse system of morphisms in
  \eqref{eq:inversesystemwithsmooth} is coinitial with the subsystem consisting of
  morphisms from regular (resp.\ normal, Cohen--Macaulay) schemes
  $W_{\lambda,p}$.
  This proves the ``moreover'' statement.\smallskip
  \par Finally, it remains to show $(\ref{lem:approxlimit})$.
  Set
  \[
    J' = \bigsqcup_{\lambda \in \Lambda} P_\lambda'
  \]
  with the preorder where
  $(\lambda,p') \le (\mu,q')$ if and only $\lambda \le \mu$ and
  the morphism $g_{\mu,q'}$ fits into a commutative diagram
  \[
    \begin{tikzcd}
      W_{\mu,q'} \rar{g_{\mu,q'}} \dar
      & X_\mu \rar \dar{v_{\lambda\mu}}
      & \Spec(R_\mu) \dar\\
      W_{\lambda,p'} \rar{g_{\lambda,p'}} & X_\lambda \rar &
      \Spec(R_\lambda)\mathrlap{.}
    \end{tikzcd}
  \]
  By the argument in $(\ref{lem:approxwj})$, the two inverse systems
  \begin{alignat}{4}
    \nonumber
    &\bigl\{g_{\lambda,p} &{}\colon{}& W_{\lambda,p} &{}\longrightarrow{}&
    X_\lambda\bigr\}_{(\lambda,p) \in J}\\
    \label{eq:glambdapprimesystem}
    &\bigl\{g_{\lambda,p'} &{}\colon{}& W_{\lambda,p'} &{}\longrightarrow{}&
    X_\lambda\bigr\}_{(\lambda,p') \in J'}
  \end{alignat}
  are coinitial.
  It therefore suffices to show that $J'$ is filtered and that
  the morphism $\pi\colon \ZR(X) \to X$ is the limit of the morphisms in
  \eqref{eq:glambdapprimesystem}.
  \par To show that $J'$ is filtered, let $(\lambda_1,p_1')$ and
  $(\lambda_2,p_2')$ be two indices in $J'$.
  Since $\Lambda$ is directed, there exists $\mu \in \Lambda$ such that
  $\lambda_1 \le \mu$ and $\lambda_2 \le \mu$.
  We now claim we can construct a commutative diagram of the form
  \[
    \begin{tikzcd}[row sep=small]
      & W'_{\lambda_1,p_1'} \rar[hook] & W_{\lambda_1,p_1'}
      \times_{X_{\lambda_1}} X_\mu \arrow{dr}{\pr_2}\\
      W_{\mu,q'} \arrow{ur}\arrow{dr} & & & X_\mu\\
      & W'_{\lambda_2,p_2'} \rar[hook] & W_{\lambda_2,p_2'}
      \times_{X_{\lambda_2}} X_\mu \arrow{ur}[swap]{\pr_2}
    \end{tikzcd}
  \]
  where the composition $W_{\mu,q'} \to X_\mu$ is projective and birational, and
  where $W_{\mu,q'}$ is integral.
  We set $W'_{\lambda_1,p_1'}$ to be the closure of the inverse image of the
  open set in $X_\mu$
  over which the second projection $\pr_2\colon W_{\lambda_1,p_1'}
  \times_{X_{\lambda_1}} X_\mu \to X_\mu$ is an isomorphism, and
  similarly for $W'_{\lambda_2,p_2'}$.
  Since $W'_{\lambda_1,p_1'} \to X_\mu$ is a projective and birational morphism
  from an integral scheme by
  construction, it is the blowup along some ideal $\sI_1 \in \AId_{X_\mu}$ by
  \cite[Corollaire 2.3.7]{EGAIII1}, and similarly $W'_{\lambda_2,p_2'} \to X_\mu$ is
  the blowup along some ideal $\sI_2 \in \AId_{X_\mu}$.
  We can therefore consider the blowup $W_{\mu,q'} \to X_\mu$ along
  $\sI_1\sI_2$, which factors uniquely through $W'_{\lambda_1,p_1'}$ and
  $W'_{\lambda_2,p_2'}$ by the universal property of blowups
  \cite[\href{https://stacks.math.columbia.edu/tag/0806}{Tag
  0806}]{stacks-project}.
  Note that $W_{\mu,q'}$ is integral by \cite[Proposition 8.1.4$(i)$]{EGAII}.
  \par It remains to show that the limit of the morphisms in
  \eqref{eq:glambdapprimesystem} is indeed the morphism $\pi\colon \ZR(X) \to
  X$.
  We claim that the limit of the inverse system
  \begin{equation}\label{eq:glambdapprimesystemonx}
    \Set[\big]{g_{\lambda,p'} \times_{X_\lambda} \id_X \colon W_{\lambda,p'}
      \times_{X_\lambda} X \longrightarrow X_\lambda \times_{X_\lambda}
    X}_{(\lambda,p') \in J'}
  \end{equation}
  coincides with the limit of the inverse system \eqref{eq:glambdapprimesystem}.
  This follows since the squares
  \[
    \begin{tikzcd}[column sep=5em]
      W_{\lambda,p'} \times_{X_\lambda} X
      \rar{g_{\lambda,p'} \times_{X_\lambda} \id_X} \dar
      & X_\lambda \times_{X_\lambda} X \dar{v_\lambda}\\
      W_{\lambda,p'} \rar{g_{\lambda,p'}} & X_\lambda
    \end{tikzcd}
  \]
  are Cartesian, and hence the limits of the two inverse systems
  \eqref{eq:glambdapprimesystem} and \eqref{eq:glambdapprimesystemonx} satisfy
  the same universal property.
  \par We now show that the limit of the inverse system
  \eqref{eq:glambdapprimesystemonx} is indeed the morphism $\pi\colon \ZR(X) \to
  X$.
  It suffices to show that the the inverse system
  \eqref{eq:glambdapprimesystemonx} is coinitial with the inverse system
  defining $\ZR(X)$.
  Let $g_{\lambda,p'}\colon W_{\lambda,p'} \to X_\lambda$ be morphism in
  \eqref{eq:glambdapprimesystem}.
  As before, we know that $g_{\lambda,p'}$ is the blowup along some ideal
  $\sI \in \AId_{X_\lambda}$ by \cite[Corollaire 2.3.7]{EGAIII1}.
  We then have the commutative diagram
  \[
    \begin{tikzcd}[column sep=4em]
      X_{v_\lambda^{-1}\sI\cdot\cO_X} \rar{\pi_{v_\lambda^{-1}\sI\cdot\cO_X}}\dar
      & X\dar{v_\lambda}\\
      W_{\lambda,p'} \rar{g_{\lambda,p'}} & X_\lambda
    \end{tikzcd}
  \]
  by the universal property of blowups
  \cite[\href{https://stacks.math.columbia.edu/tag/0806}{Tag
  0806}]{stacks-project},
  where the top horizontal arrow is the blowup along
  $v_\lambda^{-1}\sI\cdot\cO_X$.
  Note that $v_\lambda^{-1}\sI\cdot\cO_X \in \AId_{X}$.
  By the universal property of fiber products, we see that
  $X_{v_\lambda^{-1}\sI\cdot\cO_X}$ factors through the base change of
  $g_{\lambda,p'}$.
  \par Conversely, suppose $\pi_\sI\colon X_\sI \to X$ is an admissible blowup.
  Then, by \cite[Th\'eor\`eme 8.8.2$(ii)$]{EGAIV3} (here we use the
  Noetherianity of $X$ to say that the blowup $\pi_\sI$ is finitely presented),
  there exists an index $\alpha \in
  \Lambda$ and a morphism $h'_\alpha\colon W'_\alpha \to X_\alpha$ for which the
  diagram
  \[
    \begin{tikzcd}
      X_\sI \rar{\pi_\sI}\dar & X\dar{v_\alpha}\\
      W'_\alpha \rar{h'_\alpha} & X_\alpha
    \end{tikzcd}
  \]
  is Cartesian.
  For each $\lambda \ge \alpha$,
  denote by $h'_\lambda\colon W'_\lambda \to X_\lambda$ the base change of
  $h'_\alpha$ to $X_\lambda$.
  By \cite[Proposition 8.7.2 and Th\'eor\`eme 8.10.5$(i)$,$(xiii)$]{EGAIV3}, for large
  enough $\lambda \ge \alpha$, the scheme $W'_\lambda$ is reduced, the morphism
  $h'_\lambda$ is projective, and the restriction of $h'_\lambda$ to an open
  subset $U_\lambda$ of $X_\lambda$ induces an isomorphism.
  Denote by $\xi$ and $\xi_\lambda$ the generic points of $X$ and
  $X_\lambda$, respectively.
  By transitivity of fibers \cite[Corollaire 3.4.9]{EGAInew} and
  \cite[(4.4.1)]{EGAIV2}, the generic fibers
  $h^{\prime-1}_\lambda(\xi_\lambda)$ are also irreducible.
  Now let $W_\lambda$ be the scheme-theoretic closure of
  $h^{\prime-1}_\lambda(\xi_\lambda)$, which coincides with the set-theoretic
  closure with reduced scheme structure since $W_\lambda'$ is reduced
  \cite[Corollaire 6.10.6]{EGAInew}.
  The scheme $W_\lambda$ is irreducible by \cite[Chapitre 0, Proposition
  2.1.13]{EGAInew}, hence integral.
  Now consider the composition
  \[
    h_\lambda \colon W_\lambda \hooklongrightarrow W'_\lambda
    \overset{h'_\lambda}{\longrightarrow} X_\lambda.
  \]
  This morphism is projective since it is the composition of projective
  morphisms, and is birational since its restriction to $U_\lambda$ is still an
  isomorphism.
  By \cite[Lemme 13.1.2]{EGAIV3}, the square
  \[
    \begin{tikzcd}
      X_\sI \rar{\pi_\lambda}\dar & X\dar{v_\lambda}\\
      W_\lambda \rar{h_\lambda} & X_\lambda
    \end{tikzcd}
  \]
  is Cartesian.
  Thus, the inverse system \eqref{eq:glambdapprimesystemonx} is coinitial with
  the inverse system defining $\ZR(X)$, and hence their limits
  coincide.
\end{proof}

\section{Relative vanishing and injectivity theorems for Zariski--Riemann spaces}
\label{sect:grvanishing}
Our goal in this section is to prove the following relative vanishing and
injectivity theorem for Zariski--Riemann spaces.
This theorem is stated using local cohomology
following the dual formulation of Grauert--Riemenschneider
vanishing due to Hartshorne and Ogus \cite[Proposition 2.2]{HO74}.
This dual formulation allows us to prove these statements for
Zariski--Riemann spaces.
This statement also has the advantage of not requiring the existence of dualizing
complexes or canonical sheaves $\omega_X$, which are not
known to behave well under limits.
\begin{theorem}\label{thm:zrmaindualvanishing}
  Let $(R,\fm)$ be an integral Noetherian local $\QQ$-algebra
  and set $Y \coloneqq \Spec(R)$.
  Let $f\colon X \to Y$ be a proper surjective
  morphism from an integral scheme $X$.
  Set $Z = f^{-1}(\{\fm\})$, and
  denote by $\pi\colon \ZR(X) \to X$ the canonical projection morphism from the
  Zariski--Riemann space of $X$.
  \par Consider an invertible sheaf $\sL$ on $X$.
  \begin{enumerate}[label=$(\roman*)$,ref=\roman*]
    \item\label{thm:zrmaindualvanishingkv}
      Suppose $\sL$ is $f$-big and $f$-semi-ample.
      Then, we have
      \[
        H^i_{\pi^{-1}(Z)}\bigl(\ZR(X),\pi^*\sL^{-1}\bigr) = 0
      \]
      for all $i < \dim(X)$.
    \item\label{thm:zrmaindualvanishinggr}
      Suppose $\sL$ is $f$-semi-ample.
      Let $D$ be an effective Weil divisor on $X$ for which there exists an
      integer $n > 0$ and an
      effective Weil divisor $D'$ on $X$ such that $\cO_X(D+D')
      \simeq \sL^{\otimes n}$.
      Then, the canonical morphisms
      \[
        H^i_{\pi^{-1}(Z)}\bigl(\ZR(X),\pi^*\bigl(\sL^{-k}(-D)\bigr)\bigr)
        \longrightarrow
        H^i_{\pi^{-1}(Z)}\bigl(\ZR(X),\pi^*\sL^{-k}\bigr)
      \]
      induced by the inclusion $\cO_{X}(-D) \hookrightarrow \cO_{X}$ are
      surjective for all $i$ and for all $k > 0$.
  \end{enumerate}
\end{theorem}
\par In \S\ref{sect:proofmaingr}, we prove that for Noetherian schemes $X$,
vanishing and injectivity can be stated in terms of higher direct images and
$\omega_X$ or in terms of local cohomology modules (Proposition
\ref{prop:aisastar}).
This will be used after reducing to the case of varieties over $\QQ$ to prove
Theorem \ref{thm:zrmaindualvanishing} and will also be used later to prove
Theorems \ref{thm:kvvanishing} and \ref{thm:mainvanishing}.
The key ingredient for showing the two formulations
are equivalent is a combination of Grothendieck local duality
and Grothendieck duality for proper morphisms, which is called the
\textsl{local-global duality of Lipman} in \cite[p.\ 283]{HHK98}.
We then prove Theorem \ref{thm:zrmaindualvanishing} in
\S\ref{sect:grvanishingzr} using our approximation results from
\S\ref{sect:mainapprox}.
\subsection{Lipman's local-global duality}\label{sect:proofmaingr}
We prove that Theorems \ref{thm:mainvanishing} and
\ref{thm:maindualvanishing} are equivalent when dualizing complexes exist.
The key ingredient is the following duality statement due to Lipman
\cite{Lip78}.
See \cite[Definition on p.\ 276]{Har66} for the notion of a normalized dualizing
complex used below.
Hartshorne and Ogus give a different approach using formal
duality in the proof of \cite[Proposition 2.2]{HO74}.
If $\sL$ is a locally free sheaf of finite rank on a ringed space $X$,
the \textsl{dual} of $\sL$ is $\sL^\vee\coloneqq \HHom_{\cO_X}(\sL,\cO_X)$.
\begin{lemma}[(see {\cite[Theorem on p.\ 188]{Lip78}})]\label{lem:lipmandual}
  Let $f\colon X \to \Spec(R)$ be a proper morphism where $(R,\fm)$ is a
  Noetherian local ring with a normalized dualizing complex $\omega_R^\bullet$.
  Set $Z = f^{-1}(\{\fm\})$, and let $\sL$ be a locally free sheaf of finite
  rank on $X$.
  Then, there is a quasi-isomorphism
  \begin{align*}
    \RR\Gamma_Z(X,\sL^{\vee}) &\simeq \Hom_R\bigl(\RR f_*
    (\omega_X^\bullet\otimes_{\cO_X}\sL),E\bigr),
    \intertext{functorial in $\sL$, where $\omega_X^\bullet =
    f^!\omega_R^\bullet$ and where $E$ is the injective
    hull of the residue field of $R$.
    In particular, if $X$ is Cohen--Macaulay of pure dimension $n$, we have
    an isomorphism}
    H^i_Z(X,\sL^{\vee}) &\simeq \Hom_R\bigl( R^{n-i} f_*
    (\omega_X \otimes_{\cO_X} \sL),E\bigr)
  \end{align*}
  for every $i$,
  where $\omega_X$ denotes the unique cohomology sheaf of
  $f^!\omega_Y^\bullet$.
\end{lemma}
\begin{proof}
  We follow the proof in \cite[Theorem on p.\ 188]{Lip78}, keeping track of
  morphisms $\sL \to \sM$ of locally free sheaves of finite rank along the way.
  We have the commutative diagram
  \begin{equation}\label{eq:lipmandual1}
    \begin{tikzcd}
      \RR f_*(\omega_X^\bullet \otimes \sL)\dar & \lar[swap]{\sim}
      \RR f_*\HHom_{\cO_X}(\sL^\vee,\omega_X^\bullet) \dar
      \rar{\sim} & \Hom_R(\RR f_*\sL^\vee,\omega_R^\bullet) \dar\\
      \RR f_*(\omega_X^\bullet \otimes \sM) & \lar[swap]{\sim}
      \RR f_*\HHom_{\cO_X}(\sM^\vee,\omega_X^\bullet)
      \rar{\sim} & \Hom_R(\RR f_*\sM^\vee,\omega_R^\bullet)
    \end{tikzcd}
  \end{equation}
  where the horizontal arrows are isomorphisms induced by the isomorphism of
  functors
  \[
    \RRHHom_{\cO_X}(\sL^{\vee},-) \overset{\sim}{\longrightarrow} -
    \otimes_{\cO_X} \sL
  \]
  coming from \cite[Chapter V, Corollary 6.3]{Har66} and the evaluation at $1$
  map in the left square, and are induced by Grothendieck duality
  \citeleft\citen{Har66}\citemid Chapter VII, Corollary 3.4$(c)$\citepunct
  \citen{Con00}\citemid Theorem 3.4.4\citeright\ in the right square.
  Since $\RR f_*\sL^\vee$ is quasi-coherent with coherent cohomology (by the
  assumption that $f$ is proper \cite[Th\'eor\`eme 3.2.1]{EGAIII1}), we can apply local
  duality \cite[Chapter V, Corollary 6.3]{Har66} to obtain the commutative diagram
  \begin{equation}\label{eq:lipmandual2}
    \begin{tikzcd}
      \RR\Gamma_\fm(R,\RR f_*\sL^\vee)
      \rar{\sim} & \Hom_R\bigl(\Hom_R(\RR f_*\sL^\vee,\omega_R^\bullet),E\bigr)
      \\
      \RR\Gamma_\fm(R,\RR f_*\sM^\vee) \uar
      \rar{\sim} & \Hom_R\bigl(\Hom_R(\RR f_*\sM^\vee,\omega_R^\bullet),E\bigr)
      \uar
    \end{tikzcd}
  \end{equation}
  where the horizontal arrows are isomorphisms.
  Using the isomorphism of functors $\RR\Gamma_\fm \circ \RR f_* \simeq
  \RR\Gamma_Z$, we can identify the objects in the left column with
  $\RR\Gamma_Z(X,\sL^\vee)$ and $\RR\Gamma_Z(X,\sM^\vee)$, respectively.
  We can then combine the diagrams we have obtained so far to obtain the
  following commutative diagram:
  \[
    \begin{tikzcd}
      \RR\Gamma_Z(X,\sL^\vee)
      \rar{\sim} & \Hom_R\bigl(\Hom_R(\RR f_*\sL^\vee,\omega_R^\bullet),E\bigr)
      \rar[dash]{\sim} & \Hom_R\bigl(\RR f_*(\omega_X^\bullet \otimes \sL),E\bigr)\\
      \RR\Gamma_Z(X,\sM^\vee) \uar
      \rar{\sim} & \Hom_R\bigl(\Hom_R(\RR f_*\sM^\vee,\omega_R^\bullet),E\bigr)
      \uar \rar[dash]{\sim} & \Hom_R\bigl(\RR f_*(\omega_X^\bullet \otimes
      \sM),E\bigr)\mathrlap{.}\uar
    \end{tikzcd}
  \]
  Here, the left square is \eqref{eq:lipmandual2} with the identification
  $\RR\Gamma_\fm \circ \RR f_* \simeq \RR\Gamma_Z$ made above, and the right
  square is obtained from \eqref{eq:lipmandual1} and applying $\Hom_R(-,E)$
  which has no higher Ext modules since $E$ is injective.
  Finally,
  the ``in particular'' statement follows from the first statement after taking
  $i$-th cohomology, since in this case $\omega_X \simeq
  \omega_X^\bullet[-n]$ by local duality \cite[Chapter V, Corollary 6.3]{Har66}.
\end{proof}
We now show that Theorems \ref{thm:mainvanishing} and
\ref{thm:maindualvanishing} are equivalent when dualizing complexes exist.
\begin{proposition}\label{prop:aisastar}
  Let $f\colon X \to Y$ be a proper morphism of
  Noetherian schemes, and
  suppose that $X$ is Cohen--Macaulay and equidimensional and that $Y$ has a dualizing
  complex $\omega_Y^\bullet$.
  Denote by $\omega_X$ the unique cohomology sheaf of
  $f^!\omega_Y^\bullet$ (after possibly applying shifts on each connected
  component of $X$).
  \par Consider an invertible sheaf $\sL$ on $X$ and fix $y \in Y$.
  Denote by $f_y\colon X_y \to \Spec(\cO_{Y,y})$ the base
  change of $f$ along $\Spec(\cO_{Y,y}) \to Y$, denote by $\sL_y$ the pullback of
  $\sL$ to $X_y$, and set $Z_y = f_y^{-1}(y)$.
  For all $i$, we have the following:
  \begin{enumerate}[label=$(\roman*)$,ref=\roman*]
    \item\label{prop:aisastarkv}
      $R^i f_{y*}(\omega_{X_y} \otimes_{\cO_{X_y}} \sL_y) = 0$ if and only
      if $H^{\dim(X_y)-i}_{Z_y}(X_y,\sL_y^{-1}) = 0$.
    \item\label{prop:aisastargr}
      Let $D$ be an effective Cartier divisor on $X$, and denote by $D_y$
      the pullback of $D$ to $X_y$.
      Then, the canonical morphism
      \[
        R^i f_{y*}(\omega_{X_y} \otimes_{\cO_{X_y}} \sL_y) \longrightarrow R^i
        f_{y*}\bigl(\omega_{X_y} \otimes_{\cO_{X_y}} \sL_y(D_y)\bigr)
      \]
      induced by the inclusion $\cO_{X_y} \to \cO_{X_y}(D_y)$ is
      injective if and only if the canonical morphism
      \[
        H^{\dim(X_y)-i}_{Z_y}\bigl(X_y,\sL_y^{-1}(-D_y)\bigr) \longrightarrow
        H^{\dim(X_y)-i}_{Z_y}(X_y,\sL_y^{-1})
      \]
      induced by the inclusion $\cO_{X_y}(-D_y) \to \cO_{X_y}$ is surjective.
  \end{enumerate}
\end{proposition}
\begin{proof}
  Since all statements are local by flat base change,
  we may replace $Y$ with $\Spec(\cO_{Y,y})$ to assume that
  $Y$ is the spectrum of a Noetherian local ring $(R,\fm)$ with a dualizing
  complex, since dualizing complexes localize \cite[Chapter V, Corollary 2.3]{Har66}.
  After translating the dualizing complex, we may assume it is normalized.
  \par We first consider $(\ref{prop:aisastarkv})$.
  Let $E$ denote the injective hull of the residue field of $R$.
  Since the functor $\Hom_R(-,E)$ is faithfully exact \cite[Corollary 3.2(2)]{Ish64}, we see
  that $R^if_*(\omega_X \otimes_{\cO_X} \sL) = 0$ if and only if
  \[
    \Hom_R\bigl(R^if_*(\omega_X \otimes_{\cO_X} \sL),E\bigr) = 0.
  \]
  By local-global duality (Lemma \ref{lem:lipmandual}), this is equivalent to
  $H^{\dim(X)-i}_Z(X,\sL^{-1}) = 0$.
  \par We now consider $(\ref{thm:maindualvanishinggr})$.
  Since $\Hom_R(-,E)$ is faithfully exact \cite[Corollary 3.2(2)]{Ish64}, 
  the morphism
  \begin{align*}
    R^i f_*(\omega_X \otimes_{\cO_X} \sL) &\longrightarrow R^i
    f_*\bigl(\omega_X \otimes_{\cO_X} \sL(D)\bigr)
    \intertext{is injective if and only if}
    \Hom_R\bigl(R^if_*\bigl(\omega_X \otimes_{\cO_X} \sL(D)\bigr),E\bigr)
    &\longrightarrow
    \Hom_R\bigl(R^if_*(\omega_X \otimes_{\cO_X} \sL),E\bigr)
    \intertext{is surjective.
    This is equivalent to the surjectivity of}
    H^{\dim(X)-i}_Z\bigl(X,\sL^{-1}(-D)\bigr) &\longrightarrow
    H^{\dim(X)-i}_Z(X,\sL^{-1})
  \end{align*}
  by local-global duality (Lemma \ref{lem:lipmandual}).
\end{proof}
\subsection{Relative vanishing and injectivity theorem for Zariski--Riemann
spaces}\label{sect:grvanishingzr}
In this subsection, we prove our relative vanishing and injectivity theorem for
Zariski--Riemann spaces using our approximation results in
\S\ref{sect:mainapprox}.
As outlined in \S\ref{sect:intro}, the idea is to approximate the local
cohomology modules in question by approximating the morphism $f\colon X \to Y$ by
a morphism of $\QQ$-varieties.
We will show later that this vanishing descends to $X$
using relative vanishing
for the canonical morphism $\pi\colon \ZR(X) \to X$ from the Zariski--Riemann space
associated to $X$ (Theorem \ref{thm:kempflike}).
\begin{proof}[Proof of Theorem \ref{thm:zrmaindualvanishing}]
  For $(\ref{thm:zrmaindualvanishinggr})$,
  since the map $\cO_X \to \cO_X(D)$ factors the map $\cO_X \to \cO_X(D+D')$, we
  can replace $D$ by $D+D'$ to assume that $\cO_X(D) \simeq \sL^{\otimes n}$,
  and in particular, we may assume that $D$ is Cartier.
  \par We now proceed in a sequence of steps.
  \begin{step}\label{step:zrstsspecialcaseoverq}
    It suffices to show that
    for morphisms $f\colon X \to Y$ fitting into a Cartesian
    diagram
    \[
      \begin{tikzcd}
        X_y \rar{f_y}\dar & \Spec(\cO_{Y,y})\dar\\
        X \rar{f} & Y
      \end{tikzcd}
    \]
    where $f\colon X \to Y$ is a morphism of varieties over
    $\QQ$, $X$ is smooth, $Y$ is affine, and $y \in Y$ is a point, we have
    \[
      H^i_{f_y^{-1}(\{y\})}(X_y,\sL_y^{-1}) = 0
    \]
    for all $i < \dim(X_y)$ for $(\ref{thm:zrmaindualvanishingkv})$
    and the morphisms
    \[
      H^i_{f_y^{-1}(\{y\})}\bigl(X_y,\sL_y^{-k}(-D)\bigr)
      \longrightarrow
      H^i_{f_y^{-1}(\{y\})}(X_y,\sL_y^{-k})
    \]
    are surjective for all $i$ and all $k > 0$ for $(\ref{thm:zrmaindualvanishinggr})$,
    where $\sL_y$ is the pullback of
    $\sL$ to $X_y$.
  \end{step}
  By Lemma \ref{lem:approx} and Remark \ref{rem:maincasesforapprox} applied to
  $f\colon X \to \Spec(R)$ and $\kk = \QQ$, we have the commutative diagram
  \begin{equation}\label{eq:zrapproxlemapp}
    \begin{tikzcd}
      \ZR(X) \rar{\pi}\dar & X \rar{f}\dar & \Spec(R)\dar\\
      \tilde{W}_{\lambda,p} \rar{\tilde{g}_{\lambda,p}}\dar
      & \tilde{X}_\lambda \rar{\tilde{f}_\lambda}\dar &
      \Spec\bigl((R_\lambda)_{\fm_\lambda}\bigr)\dar\\
      W_{\lambda,p} \rar{g_{\lambda,p}} & X_\lambda \rar{f_\lambda} &
      \Spec(R_\lambda)
    \end{tikzcd}
  \end{equation}
  of locally ringed spaces, where all but $\ZR(X)$ are Noetherian schemes of equal
  characteristic zero, the bottom squares are Cartesian,
  and where the schemes in the bottom row are varieties over $\QQ$ with
  $W_{\lambda,p}$ smooth, such that the
  morphisms in the top row are the limits of the morphisms in the bottom row.
  As in the proof of Lemma \ref{lem:approx}$(\ref{lem:approxxj})$ and in Remark
  \ref{rem:maincasesforapprox}, the inverse limits of the morphisms in the
  middle and bottom row satisfy the same universal property, and hence both
  limit to the morphisms in the top row.
  By Lemma \ref{lem:approx}$(\ref{lem:approxlj})$, we also have that
  \[
    \sL \simeq v_\lambda^*\sL_\lambda \simeq \varinjlim_{\lambda \in
    \Lambda}v_\lambda^*\sL_\lambda
  \]
  for $f_\lambda$-big and $f_\lambda$-semi-ample (resp.\ $f_\lambda$-semi-ample)
  invertible sheaves $\sL_\lambda$, where $v_\lambda \colon X \to X_\lambda$ are
  the canonical projection morphisms.
  We denote by $\tilde{\sL}_\lambda$ the pullback of $\sL_\lambda$ to
  $\tilde{X}_\lambda$, which is $\tilde{f}_\lambda$-semi-ample by \cite[Lemma
  2.12$(i)$]{CT20}, and is also $\tilde{f}_\lambda$-big in situation
  $(\ref{thm:zrmaindualvanishingkv})$ by the fact that the base change
  $R_\lambda \to (R_\lambda)_{\fm_\lambda}$ does not affect generic fibers.
  We also know that $g_{\lambda,p}^*\sL_\lambda$ is $(f_\lambda \circ
  g_{\lambda,p})$-semi-ample by \cite[Lemma 2.11]{CT20}, and is also
  $(f_\lambda \circ g_{\lambda,p})$-big in situation
  $(\ref{thm:zrmaindualvanishingkv})$ by the fact that
  $g_{\lambda,p}$ induces a birational morphism along the generic fiber of
  $f_\lambda$.
  The same reasoning applies to $\tilde{g}_{\lambda,p}^*\tilde{\sL}_\lambda$.
  \par To show Theorem
  \ref{thm:zrmaindualvanishing}$(\ref{thm:zrmaindualvanishingkv})$, the
  hypothesis in Step \ref{step:zrstsspecialcaseoverq} implies
  \[
    H^i_{(\tilde{f}_\lambda \circ \tilde{g}_{\lambda,p})^{-1}(\{\fm_\lambda\})}
    \bigl(\tilde{W}_{\lambda,p},\tilde{g}_{\lambda,p}^*\tilde{\sL}_\lambda^{-1}\bigr) = 0
  \]
  for all $i < \dim(X_\lambda \otimes_{R_\lambda} (R_\lambda)_{\fm_\lambda})$
  since $\tilde{g}_{\lambda,p}^*\tilde{\sL}_\lambda^{-1}$
  is $(\tilde{f}_\lambda \circ \tilde{g}_{\lambda,p})$-big and
  $(\tilde{f}_\lambda \circ \tilde{g}_{\lambda,p})$-semi-ample.
  Since
  \[
    \dim(X) \le \dim\bigl(X_\lambda \otimes_{R_\lambda}
    (R_\lambda)_{\fm_\lambda}\bigr)
  \]
  for all $\lambda \in \Lambda$ (see Lemma
  \ref{lem:approx}$(\ref{lem:approxxj})$), this implies in particular that the
  vanishing holds for all $i < \dim(X)$.
  Taking colimits over all $(\lambda,p) \in J$, Theorem
  \ref{lem:lccolimits} implies
  \[
    H^i_{\pi^{-1}(Z)}\bigl(\ZR(X),\pi^*\sL^{-1}\bigr) =
    H^i_{(f \circ \pi)^{-1}(\{\fm\})}\bigl(\ZR(X),\pi^*\sL^{-1}\bigr) = 0
  \]
  for all $i < \dim(X)$,
  where the colimit of the inverse images of the sheaves
  $\tilde{g}_{\lambda,p}^*\tilde{\sL}_\lambda$ on $\ZR(X)$ is $\pi^*\sL$ by the
  commutativity of the diagram \eqref{eq:zrapproxlemapp}.
  \par It remains to show Theorem
  \ref{thm:zrmaindualvanishing}$(\ref{thm:zrmaindualvanishinggr})$.
  By \cite[\href{https://stacks.math.columbia.edu/tag/0B8W}{Tag
  0B8W}(3)]{stacks-project}, we can find $\alpha \in \Lambda$ such that
  $\cO_X(-D) \simeq v_\alpha^{-1}\sI_\alpha\cdot\cO_{X}$ for an ideal sheaf $\sI_\alpha
  \subseteq \cO_{X_\alpha}$.
  After replacing $X_\alpha$ by the blowup of $X_\alpha$ along $\sI_\alpha$,
  we may assume that $\sI_\alpha$ is invertible.
  Let $D_\alpha$ be the effective Cartier divisor corresponding to $\sI_\alpha$.
  By \cite[Corollaire 8.5.2.5]{EGAIV3}, after possibly replacing $\alpha$ by a
  larger index and $D_\alpha$ by its pullback (which exist since the $v_{\alpha\lambda}$
  are surjective morphisms of integral Noetherian schemes \cite[Proposition
  21.4.5$(iii)$]{EGAIV4}),
  we may assume that $\cO_{X_\alpha}(-D_\alpha) \simeq
  \sL_\alpha^{\otimes-n}$.
  We can therefore write
  the injection $\sL^{-k}(-D) \hookrightarrow \sL^{-k}$ as the colimit of
  the injections
  \begin{align*}
    v_\lambda^*\bigl(\sL_\lambda^{-k}(-D_\lambda)\bigr)
    \hooklongrightarrow{}& v_\lambda^*\bigl(\sL_\lambda^{-k}\bigr).
  \intertext{Denoting by $\tilde{D}_\lambda$ the restriction of $D_\lambda$ to
  $\tilde{X}_\lambda$, we see that}
    H^i_{(\tilde{f}_\lambda \circ \tilde{g}_{\lambda,p})^{-1}(\{\fm_\lambda\})}
    \bigl(\tilde{W}_{\lambda,p},\tilde{g}_{\lambda,p}^*\tilde{\sL}_\lambda^{-k}
    (-\tilde{D}_\lambda)\bigr)
    \longrightarrow{}&
    H^i_{(\tilde{f}_\lambda \circ \tilde{g}_{\lambda,p})^{-1}(\{\fm_\lambda\})}
    (\tilde{W}_{\lambda,p},\tilde{g}_{\lambda,p}^*\tilde{\sL}_\lambda^{-k})
  \intertext{is surjective for all $i$ since
  $\tilde{g}_{\lambda,p}^*\tilde{\sL}_\lambda$ is
  $f$-semi-ample by \cite[Lemma 2.11]{CT20}.
  Taking colimits over all $(\lambda,p) \in J$, Theorem
  \ref{lem:lccolimits} implies}
    H^i_{(f \circ \pi)^{-1}(\{\fm\})}\bigl(\ZR(X),\pi^*\bigl(\sL^{-k}(-D)\bigr)
    \bigr) \longrightarrow{}&
    H^i_{(f \circ \pi)^{-1}(\{\fm\})}\bigl(\ZR(X),\pi^*\sL^{-k}\bigr),
  \end{align*}
  is surjective for all $i$, where the colimit of the inverse
  images of the sheaves $g_{\lambda,p}^*(\cO_{X_\lambda}(-D_\lambda))$ on
  $\ZR(X)$ is $\pi^*\cO_X(-D)$ by the commutativity of the diagram
  \eqref{eq:zrapproxlemapp}.
  \begin{step}\label{step:specialcaseoverqzr}
    Conclusion of proof.
  \end{step}
  \par We start by proving the special case of Theorem
  \ref{thm:zrmaindualvanishing}$(\ref{thm:zrmaindualvanishingkv})$
  stated in Step \ref{step:zrstsspecialcaseoverq}.
  By Proposition \ref{prop:aisastar}$(\ref{prop:aisastarkv})$ and flat base
  change, it suffices to show that
  \[
    R^if_*(\omega_{X} \otimes_{\cO_X} \sL) = 0
  \]
  for all $i > 0$.
  Consider the Stein factorization
  \[
    X \overset{f'}{\longrightarrow} Y' \overset{g}{\longrightarrow} Y
  \]
  of $f$.
  We can replace $f\colon X \to Y$ by $f'\colon X \to Y'$ to assume that
  $Y$ is normal, since the relative normalization morphism $g$ 
  is affine and hence
  $R^if'_*(\omega_X \otimes_{\cO_X} \sL)$ vanishes if and only if
  $R^if_*(\omega_X \otimes_{\cO_X} \sL)$ does.
  By flat base change and the fact that
  the formation of $\omega_X$ is compatible with ground field extensions
  \cite[Chapter V, Corollary 3.4$(a)$]{Har66}, it suffices to show that denoting by
  \[
    f_\CC\colon X_\CC \longrightarrow Y_\CC
  \]
  the base change of $f$ along the field extension $\QQ\subseteq\CC$, we have
  \[
    R^if_{\CC*}(\omega_{X_\CC} \otimes_{\cO_{X_\CC}} \sL) = 0
  \]
  for all $i > 0$.
  We note that $f_\CC$ is maximally dominating by the flatness of $\QQ \subseteq
  \CC$ \cite[Expos\'e II, Proposition 1.1.5]{ILO14}.
  Since $Y_\CC$ is normal, it is the disjoint union of normal varieties.
  We claim we may work one irreducible component at a time to assume that
  $f_\CC$ is a projective surjective morphism of complex varieties.
  Note that $\sL_\CC$ is $f_\CC$-semiample by \cite[Lemma 2.12$(i)$]{CT20}.
  By transitivity of fibers \cite[Corollaire 3.4.9]{EGAInew}, the compatibility
  of maps induced by linear systems and flat base change \cite[(4.2.10)]{EGAII}, and
  the fact that generically finite morphisms are stable under flat base change
  (combine \cite[Expos\'e II, Proposition 1.1.5]{ILO14} and the characterization
  of generically finite morphisms in \cite[Expos\'e II, Proposition 1.1.7]{ILO14}),
  we see that the restriction of large enough powers of $\sL_\CC$ induce
  generically finite morphisms on each fiber of $f_\CC$ over the maximal points
  of $Y$.
  We therefore see that $\sL_\CC$ is $f_\CC$-big.
  Now the required vanishing holds in situation $(\ref{thm:zrmaindualvanishingkv})$ by
  relative Kawamata--Viehweg vanishing for complex algebraic varieties
  \cite[Theorem 1-2-3]{KMM87}.
  \par It remains to show Theorem
  \ref{thm:zrmaindualvanishing}$(\ref{thm:zrmaindualvanishinggr})$
  holds in the special case
  stated in Step \ref{step:zrstsspecialcaseoverq}.
  By Proposition \ref{prop:aisastar}$(\ref{prop:aisastargr})$,
  flat base change, and the fact that
  the formation of $\omega_X$ is compatible with ground field extensions
  \cite[Chapter V, Corollary 3.4$(a)$]{Har66}, it suffices to show that denoting by
  \[
    f_\CC\colon X_\CC \longrightarrow Y_\CC
  \]
  the base change of $f$ along the field extension $\QQ\subseteq\CC$, the
  morphisms
  \[
    R^if_*(\omega_{X_\CC} \otimes_{\cO_{X_\CC}} \sL^{\otimes k}_\CC) \longrightarrow
    R^if_*\bigl(\omega_{X_\CC} \otimes_{\cO_{X_\CC}} \bigl(\sL^{\otimes k}(D)\bigr)_\CC
    \bigr)
  \]
  are injective for all $i$ and all $k > 0$, where $\sL^{\otimes k}_\CC$ and
  $(\sL^{\otimes k}(D))_\CC$ are the pullbacks of $\sL^{\otimes k}$ and
  $\sL^{\otimes k}(D)$ to $X_\CC$,
  respectively.
  This statement holds by Fujino's version of Koll\'ar's injectivity theorem for
  simple normal crossings pairs \cite[Theorem 5.6.1]{Fuj17}.
\end{proof}
\section{Rational singularities via Zariski--Riemann spaces}\label{sect:kempf}
In this section, we prove a new characterization of rational singularities and
pseudo-rational rings via Zariski--Riemann spaces in equal characteristic zero.
An advantage of this characterization is that we do not need resolutions of
singularities, quasi-excellence, or the existence of dualizing complexes.
This characterization is a version of the characterizations of rational
singularities for varieties over fields of
characteristic zero due to Lipman and Teissier \cite[$(iv)'$ on p.\ 102
and Corollary of $(iii)$ on p.\ 107]{LT81}, Lipman \cite[Lemma 4.2]{Lip94}, and
Kov\'acs \cite[Theorem 1]{Kov00}.
We have modeled the formulation of our characterization after the 
characterization of Du Bois singularities due
to Godfrey and the author of the present paper
\cite[Theorem 2.3]{GM}, where the $0$-th graded piece of the Deligne--Du Bois
complex $\underline{\Omega}^0_X$
takes the role of $\RR\pi_*\cO_{\ZR(X)}$.\medskip
\par We start with the following result for pseudo-rationality.
\begin{lemma}\label{lem:kempflikepseudorat}
  Let $(R,\fm)$ be a Noetherian local ring of dimension $d$,
  and set $X \coloneqq \Spec(R)$.
  Denote by $\pi\colon \ZR(X) \to X$ the canonical projection from the
  Zariski--Riemann space of $X$.
  For every $i$, the map
  \begin{equation}\label{eq:lcmaptozr}
    H^i_{\fm}(R) \overset{\delta^i_\pi}{\longrightarrow}
    H^i_{\pi^{-1}(\{\fm\})}\bigl(\ZR(X),\cO_{\ZR(X)}\bigr)
  \end{equation}
  is injective if and only if for every proper birational morphism $W \to X$,
  the map
  \[
    H^i_{\fm}(R) \overset{\delta^i_f}{\longrightarrow}
    H^i_{f^{-1}(\{\fm\})}\bigl(W,\cO_{W}\bigr)
  \]
  is injective.
  \par In particular, $R$ is pseudo-rational if and only if $\delta^d_\pi$
  is injective and $R$ is normal, Cohen--Macaulay, and analytically
  unramified.
\end{lemma}
\begin{proof}
  By Theorem \ref{lem:lccolimits}, we have a factorization
  \[
    H^i_{\fm}(R) \overset{\delta^i_f}{\longrightarrow} H^i_{f^{-1}(\{\fm\})}
    (W,\cO_W) \longrightarrow
    H^i_{(f \circ g)^{-1}(\{\fm\})} (\tilde{W},\cO_{\tilde{W}})
    \longrightarrow H^i_{\pi^{-1}(\{\fm\})}\bigl(\ZR(X),\cO_{\ZR(X)}\bigr)
  \]
  of $\delta^i_\pi(\cO_X)$ for every proper birational morphism $f\colon W \to
  X$, where $g\colon \tilde{W} \to W$ is a projective birational morphism such
  that $f \circ g$ is a blowup of $X$.
  Such a morphism $g$ exists by Chow's lemma \cite[Th\'eor\`eme 5.6.1]{EGAII}
  and the fact that every projective birational morphism is a blowup
  \cite[Corollaire 2.3.7]{EGAIII1}.
  \par Note that $\Rightarrow$ follows since
  if the composition $\delta^i_\pi$ is injective, then
  $\delta^i_f$ is injective.
  For $\Leftarrow$, 
  the maps
  \[
    H^i_{\fm}(R) \to H^i_{(f \circ g)^{-1}(\{\fm\})}
    (\tilde{W},\cO_{\tilde{W}})
  \]
  are injective for every $\tilde{W}$ as constructed above.
  We can therefore take colimits and apply
  Theorem \ref{lem:lccolimits} to see that $\delta^i_\pi$ is injective.
\end{proof}
\par We are now ready to show our characterization of rational singularities and
pseudo-rational rings.
Below, we will freely use the fact that the formation of $\ZR(X)$ commutes with
base change by quasi-compact separated \'etale morphisms
\cite[\href{https://stacks.math.columbia.edu/tag/087B}{Tag
087B}]{stacks-project}, and hence $\ZR(X) \times_X X_y \simeq
\ZR(X_y)$ where $X_y \coloneqq \Spec(\cO_{X,y})$.
\begin{theorem}\label{thm:kempflike}
  Let $X$ be a Noetherian scheme.
  Denote by $\pi\colon \ZR(X) \to X$ the canonical projection morphism from the
  Zariski--Riemann space of $X$.
  Consider the following conditions:
  \begin{enumerate}[label=$(\roman*)$,ref=\roman*]
    \item\label{thm:kempflikekovacs}
      For every point $y \in X$, setting $X_y \coloneqq
      \Spec(\cO_{X,y})$ and $\pi_y\colon \ZR(X_y) \to X_y$,
      the natural morphism $\cO_{X_y} \to \RR\pi_{y*}\cO_{\ZR(X_y)}$ admits a left
      inverse in the derived category of $\cO_{X_y}$-modules.
    \item\label{thm:kempflikepseudoratlike}
      For every point $y \in X$, setting $X_y \coloneqq
      \Spec(\cO_{X,y})$ and $\pi_y\colon \ZR(X_y) \to X_y$, the natural morphism
      \[
        H^i_{\{y\}}(X_y,\cO_{X_y})
        \xrightarrow{\delta^i_{\pi_y}}
        H^i_{\pi_y^{-1}(\{y\})}\bigl(\ZR(X_y),\cO_{\ZR(X_y)}\bigr)
      \]
      is injective for every $i$.
    \item\label{thm:kempflikepseudorat}
      $X$ is locally pseudo-rational.
    \item\label{thm:kempflikerat}
      $X$ is a locally quasi-excellent scheme of equal characteristic zero
      and has rational singularities.
  \end{enumerate}
  Then, we have the following implications:
  \[
    \begin{tikzcd}[column sep=large]
      (\ref{thm:kempflikekovacs})
      \rar[Rightarrow]
      & (\ref{thm:kempflikepseudoratlike})
      \arrow[bend right=24,Rightarrow,dashed]{l}[above]{\text{equal characteristic zero}}
      \arrow[bend right=24,Rightarrow,dashed]{r}[below]{\substack{\text{equal characteristic
      zero}\\\text{$+$ locally analytically unramified}}}
      & \lar[Rightarrow]
      (\ref{thm:kempflikepseudorat})
      \arrow[Leftrightarrow,dashed]{r}[above,yshift=7]{\substack{\text{equal characteristic
      zero}\\\text{$+$ locally quasi-excellent}}}
      & (\ref{thm:kempflikerat})
    \end{tikzcd}
  \]
\end{theorem}
\begin{proof}
  The implication $(\ref{thm:kempflikekovacs}) \Rightarrow
  (\ref{thm:kempflikepseudoratlike})$ holds because the morphisms
  $\delta^i_{\pi_y}$ also admit left inverses.\smallskip
  \par We now show $(\ref{thm:kempflikepseudorat}) \Rightarrow
  (\ref{thm:kempflikepseudoratlike})$.
  Injectivity for $i < \dim(X_y)$ holds since pseudo-rational implies
  Cohen--Macaulay (see Definition \ref{def:pseudorational}) and hence
  \[
    H^i_{\{y\}}(X_y,\cO_{X_y}) = 0
  \]
  for all $i < \dim(X_y)$.
  For $i = \dim(X_y)$, we apply Lemma \ref{lem:kempflikepseudorat}.\smallskip
  \par We now show $(\ref{thm:kempflikepseudoratlike}) \Rightarrow
  (\ref{thm:kempflikekovacs})$ in equal characteristic zero.
  By Noetherian induction,
  it suffices to show that if
  $\cO_{X_\eta} \to
  \RR\pi_{\eta*}\cO_{\ZR(X_{\eta})}$
  admits a left inverse for every
  proper generization $\eta \rightsquigarrow y$, then
  $\cO_{X_y} \to
  \RR\pi_{y*}\cO_{\ZR(X_y)}$ admits a left inverse.
  \par We first claim that for every $i$, the maps
  \begin{equation}\label{eq:lcvanisheskempf}
    H^i_{\{y\}}(X_y,\cO_{X_y})
    \overset{\delta^i_{\pi_y}}{\longrightarrow}
    H^i_{\pi_y^{-1}(\{y\})}\bigl(\ZR(X_y),\cO_{\ZR(X_y)}\bigr)
  \end{equation}
  are isomorphisms.
  For $i < \dim(X_y)$, the morphisms \eqref{eq:lcvanisheskempf} are isomorphisms
  by combining $(\ref{thm:kempflikepseudoratlike})$ and the
  fact that the modules on the right vanish by Theorem
  \ref{thm:zrmaindualvanishing}$(\ref{thm:zrmaindualvanishingkv})$ applied to
  $\sL = \cO_{X_y}$.
  Moreover, this morphism is also an isomorphism for $i = \dim(X_y)$ since it is
  injective by $(\ref{thm:kempflikepseudoratlike})$ and surjective by
  \cite[Remark $(b)$ on p.\ 103]{LT81}.
  \par Now consider the exact triangle
  \begin{align}
    \cO_{X_y} &\longrightarrow \RR \pi_{y*}\cO_{\ZR(X_y)} \longrightarrow
    C_y^\bullet \overset{w_y}{\longrightarrow} \cO_{X_y}[1].\nonumber
  \intertext{By the inductive hypothesis, we know that for every proper generization $\eta
  \rightsquigarrow y$, the composition}
    \cO_{X_\eta} &\longrightarrow \bigl(\RR
    \pi_{y*}\cO_{\ZR(X_y)}\bigr)_\eta \longrightarrow \RR
    \pi_{\eta*}\cO_{\ZR(X_\eta)}
    \label{eq:triangleslocalized}
  \end{align}
  admits a left inverse.
  We therefore see that the left map in \eqref{eq:triangleslocalized}
  also admits a left inverse.
  Next, consider commutative diagram
  \[
    \begin{tikzcd}
      \cdots \rar & \mathbf{H}^i_{\{y\}}(X_y,C_{y}^\bullet)
      \rar\dar & \mathbf{H}^i(X_y,C_{y}^\bullet) \rar\dar &
      \mathbf{H}^i\bigl(X_y - \{y\}, C_{y}^\bullet\bigr) \rar\dar & \cdots\\
      \cdots \rar & H^{i+1}_{\{y\}}(X_y,\cO_{X_y})
      \rar & H^{i+1}(X_y,\cO_{X_y}) \rar &
      H^{i+1}\bigl(X_y - \{y\}, \cO_{X_y}\bigr) \rar & \cdots
    \end{tikzcd}
  \]
  with exact rows, where the vertical maps are induced by $w_y$.
  We see that the right vertical map is $0$ by the existence of a left inverse
  in \eqref{eq:triangleslocalized} and since $w_y = 0$ if and only if the map
  $\cO_{X_y} \to \RR \pi_{y*}\cO_{\ZR(X_y)}$ admits a left inverse
  \cite[Corollary 1.2.7 and Remark
  1.2.9]{Nee01}.
  Set $d_y \coloneqq \dim(X_y)$.
  If $i \ne d_y-1$, then $H^{i+1}_{\{y\}}(X_y,\cO_{X_y}) = 0$ by
  \eqref{eq:lcvanisheskempf}.
  By the commutativity of the diagram, we see that the middle vertical map is
  also $0$ for $i \ne d_y-1$.
  If $i=d_y-1$, then $H^{d_y}(X_y,\cO_{X_y}) = 0$ since $X_y$ is affine.
  We therefore conclude that the map $w_y$
  induces the $0$ map on hypercohomology, and hence $w_y = 0$.
  By \cite[Corollary 1.2.7 and Remark 1.2.9]{Nee01} again, this shows that
  $\cO_{X_y} \to \RR \pi_{y*}\cO_{\ZR(X_y)}$ admits a left inverse, i.e.,
  $(\ref{thm:kempflikekovacs})$ holds.\smallskip
  \par We now show $(\ref{thm:kempflikepseudoratlike}) \Rightarrow
  (\ref{thm:kempflikepseudorat})$ in equal characteristic zero if the local
  rings of $X$ are analytically unramified.
  Since condition $(\ref{thm:kempflikepseudorat})$ is local, it suffices to
  prove the case when $X$ is the spectrum of a local $\QQ$-algebra.
  First, $(\ref{thm:kempflikekovacs})$ implies the injection
  $\cO_X \hookrightarrow f_*\cO_W$ splits
  for every finite birational morphism $f\colon W \to X$, and hence
  $X$ is normal (see \cite[Example 2.1]{Bha12}).
  Next, $X$ is Cohen--Macaulay since $(\ref{thm:kempflikepseudoratlike})$
  implies we have the injections
  \[
    H^i_{\{y\}}(X_y,\cO_{X_y})
    \overset{\delta^i_{\pi_y}}{\hooklongrightarrow}
    H^i_{\pi_y^{-1}(\{y\})}\bigl(\ZR(X_y),\cO_{\ZR(X_y)}\bigr) = 0
  \]
  where the vanishing holds by Theorem
  \ref{thm:zrmaindualvanishing}$(\ref{thm:zrmaindualvanishingkv})$ applied to
  $\sL = \cO_X$.
  Finally, injectivity at $i = d$ holds by Lemma \ref{lem:kempflikepseudorat}.\smallskip
  \par It remains to show $(\ref{thm:kempflikepseudorat}) \Leftrightarrow
  (\ref{thm:kempflikerat})$ in equal characteristic zero if $X$ is locally
  quasi-excellent.
  Since both conditions are local, we may replace $X$ by the spectrum of one of
  its local rings to assume that $X = \Spec(R)$ for a quasi-excellent local ring
  $(R,\fm)$.
  As in \cite[Example $(b)$ on p.\ 103]{LT81},
  the implication $(\ref{thm:kempflikerat}) \Rightarrow
  (\ref{thm:kempflikepseudorat})$ follows from
  Chow's lemma \cite[Th\'eor\`eme 5.6.1]{EGAII}
  and resolution of singularities \cite[Chapter I,
  \S3, Main Theorem I$(n)$]{Hir64}, which imply any proper birational morphism $W \to
  \Spec(R)$ can be dominated by a projective resolution of singularities $W' \to
  \Spec(R)$.
  For the converse implication $(\ref{thm:kempflikepseudorat}) \Rightarrow
  (\ref{thm:kempflikerat})$, let $f\colon W \to \Spec(R)$ be a proper resolution of
  singularities.
  Then, the injective maps
  \[
    H^i_{\{\fm\}}(X,\cO_{X})
    \overset{\delta^i_{f}}{\hooklongrightarrow}
    H^i_{f^{-1}(\{\fm\})}(W,\cO_W)
  \]
  are in fact isomorphisms, since $H^i_{f^{-1}(\{\fm\})}(W,\cO_W) = 0$ for all
  $i < \dim(X)$ by Theorem 
  \ref{thm:zrmaindualvanishing}$(\ref{thm:zrmaindualvanishingkv})$ applied to
  $\sL = \cO_{X}$.
  Moreover, this morphism is also an isomorphism for $i = \dim(X)$ since it is
  injective by $(\ref{thm:kempflikepseudoratlike})$ and surjective by
  \cite[Remark $(b)$ on p.\ 103]{LT81}.
  The claim now follows from local-global duality (Lemma \ref{lem:lipmandual}).
\end{proof}
Since regular local rings are pseudo-rational \cite[\S4]{LT81}, we see that
$(\ref{thm:kempflikekovacs})$ holds for regular rings in equal characteristic zero.
For excellent schemes, we can prove the stronger statement that
$R^i\pi_*\cO_{\ZR(X)} = 0$ for all $i > 0$ in both equal characteristic zero and
for varieties of dimension
$\le4$ over algebraically closed fields of arbitrary characteristic using
a recent result of Kov\'acs \cite[Theorem 8.6]{Kov}.
The result below will not be used in the sequel.
\begin{theorem}\label{thm:relativevanishingzr}
  Let $X$ be a locally pseudo-rational excellent Noetherian scheme.
  Assume that $X$ is of equal characteristic zero or a quasi-projective
  variety of dimension
  $\le4$ over an algebraically closed field of arbitrary characteristic.
  Denote by
  $\pi\colon \ZR(X) \to X$ the canonical projection morphism from the
  Zariski--Riemann space of $X$.
  Then, we have
  \[
    R^i\pi_*\cO_{\ZR(X)} = 0
  \]
  for all $i > 0$.
\end{theorem}
\begin{proof}
  By Theorem \ref{thm:fklimitsdirectimage}, it
  suffices to show that for a cofinal subset of ideals $\{\sI_\lambda\}_{\lambda
  \in \Lambda} \subseteq \AId_X$, the blowup
  $\pi_{\sI_\lambda}\colon X_{\sI_\lambda} \to X$
  along $\sI_\lambda$ satisfies
  $R^i\pi_{\sI_\lambda*}\cO_{X_{\sI_\lambda}} = 0$
  for all $i > 0$.
  \par Let $X_\sI \to X$ be a blowup where $\sI \in \AId_X$.
  Since $X_\sI$ is an excellent scheme in the equal characteristic zero case
  \cite[Proposition 7.8.6$(i)$]{EGAIV2} or a quasi-projective variety in the
  dimension $\le4$ case, there exists a birational blowup $\tilde{X} \to
  X_\sI$ such that $\tilde{X}$ is regular in the equal characteristic zero
  case \cite[Theorem 1.1]{Tem08} or such that $\tilde{X}$ is normal and
  Cohen--Macaulay in the dimension $\le4$ case \cite[Corollary 1.8]{Bro83}.
  By \cite[Premi\`ere partie, Lemme 5.1.4]{RG71} (see also \cite[Lemma 1.2]{Con07}), the
  composition
  \[
    \tilde{\pi}\colon \tilde{X} \longrightarrow X_\sI \longrightarrow X
  \]
  can be written as the blowup of $X$ along a coherent ideal sheaf $\sI_\lambda
  \in \AId_X$ contained in $\sI$.
  We have $R^i\tilde{\pi}_*\cO_{\tilde{X}} = 0$ by \cite[Theorem 8.6]{Kov}.
\end{proof}
\begin{remark}
  Only the vanishing statements proved so far are used in the proof
  of Theorem \ref{thm:boutot}.
  If one is interested in the proof of Theorem \ref{thm:boutot}, they can
  proceed directly to reading \S\ref{sect:boutot}.
\end{remark}

\section{Relative vanishing and injectivity theorems for schemes}
We are now ready to prove the following dual version of Theorem
\ref{thm:mainvanishing}.
Theorem \ref{thm:maindualvanishing} below implies Theorem
\ref{thm:mainvanishing} by Proposition \ref{prop:aisastar}.
\begin{customthm}{\ref*{thm:mainvanishing}*}\label{thm:maindualvanishing}
  Let $f\colon X \to Y$ be a proper maximally dominating morphism of
  locally Noetherian schemes of equal characteristic zero such that $X$ is
  locally pseudo-rational.
  \par Consider an invertible sheaf $\sL$ on $X$.
  For every $y \in Y$, denote by $f_y\colon X_y \to \Spec(\cO_{Y,y})$ the base
  change of $f$ along $\Spec(\cO_{Y,y}) \to Y$, denote by $\sL_y$ the pullback
  of $\sL$ to $X_y$, and set $Z_y = f_y^{-1}(y)$.
  \begin{enumerate}[label=$(\roman*)$,ref=\roman*]
    \item\label{thm:maindualvanishingkv}
      Suppose $\sL$ is $f$-big and $f$-semi-ample.
      Then, we have
      \[
      H^i_{Z_y}(X_y,\sL_y^{-1}) = 0
      \]
      for every $y \in Y$ and for all $i < \dim(X_y)$.
    \item\label{thm:maindualvanishinggr}
      Suppose $\sL$ is $f$-semi-ample.
      Let $D$ be an effective Weil divisor on $X$ 
      for which there exists an integer $n > 0$ and
      an effective Weil divisor $D'$ such that $\cO_X(D+D')
      \simeq \sL^{\otimes n}$.
      Then, the canonical morphisms
      \[
        H^i_{Z_y}\bigl(X_y,\sL_y^{-k}(-D_y)\bigr) \longrightarrow
        H^i_{Z_y}(X_y,\sL_y^{-k})
      \]
      induced by the inclusion $\cO_{X_y}(-D_y) \hookrightarrow \cO_{X_y}$ are
      surjective for every $y \in Y$ for all $i$ and for all $k > 0$, where $D_y$ is the
      pullback of $D$ to $X_y$.
  \end{enumerate}
\end{customthm}
\begin{proof}
  For $(\ref{thm:maindualvanishinggr})$,
  since the map $\cO_X \to \cO_X(D)$ factors the map $\cO_X \to \cO_X(D+D')$, we
  can replace $D$ by $D+D'$ to assume that $\cO_X(D) \simeq \sL^{\otimes n}$,
  and in particular, we may assume that $D$ is Cartier.
  \par We proceed in a sequence of steps.
  \begin{step}\label{step:integral}
    Reduction to the case when $X$ and $Y$ are integral.
  \end{step}
  Consider the Stein factorization
  \[
    X \overset{f'}{\longrightarrow} Y' \overset{g}{\longrightarrow}
    Y
  \]
  of $f$.
  By Incomparability \cite[Theorem 9.3$(ii)$]{Mat89} applied to the finite
  morphism $g$, the points in $Y'$ lying over maximal points of
  $Y$ must be maximal, and hence $f'$ is maximally
  dominating.
  Since $Y'$ is normal, it is the disjoint union of integral normal schemes.
  We claim we may work with one connected component of $X$ and $Y'$
  at a time to assume that $f$ is a surjective morphism of integral schemes.
  For each $y \in Y$, we have a decomposition
  \begin{align}
    f_y^{-1}\bigl(\{y\}\bigr) &= \bigsqcup_{y' \in g^{-1}(\{y\})}
    f_{y'}^{\prime-1}\bigl(\{y'\}\bigr)\label{eq:decompforintegral}
    \intertext{into connected components by \cite[Corollaire 4.3.3]{EGAIII1}.
    By the Mayer--Vietoris sequence (see \cite[Proof of Proposition 4.2]{Har75}),
    we then have a decomposition of functors}
    H^i_{Z_y}(X_y,-) &\simeq \bigoplus_{y' \in g^{-1}(\{y\})}
    H^i_{Z_{y'}}(X_{y'},-)\nonumber
  \end{align}
  where $X_{y'} \coloneqq X \times_{Y'} \Spec(\cO_{Y',y'})$ and $Z_{y'} =
  f_{y'}^{\prime-1}(\{y'\})$.
  Since each $f_{y'}^{\prime-1}(\{y'\})$ lies in a unique connected component of
  $X \times_{Y'} \Spec(\cO_{Y',y'})$, we can use Excision \cite[Proposition
  1.3]{Gro67} to replace $X$ and $Y$ by the
  connected components of $X$ and $Y'$ to assume that both $X$ and $Y$ are integral.
  We note that $\sL$ is $f'$-semi-ample by \cite[Lemma 2.10]{CT20}, and in case
  $(\ref{thm:maindualvanishingkv})$ is $f'$-big since the decomposition
  \eqref{eq:decompforintegral} also holds for maximal points $y \in Y$.
  \begin{step}
    Conclusion of proof.
  \end{step}
  By Step \ref{step:integral}, we may assume that $X$ and $Y$ are integral.
  Since the desired vanishing is a local condition, we may replace $Y$ by
  $\Spec(\cO_{Y,y})$ and $X$ by $X_y$.
  Denote by $\pi\colon \ZR(X) \to X$ the canonical projection morphism from the
  Zariski--Riemann space of $X$.
  By Theorem \ref{thm:zrmaindualvanishing}, we know that
  \[
    H^i_{\pi^{-1}(Z)}\bigl(\ZR(X),\pi^*\sL^{-1}\bigr) = 0
  \]
  for all $i < \dim(X)$ and that
  \[
    H^i_{\pi^{-1}(Z)}\bigl(\ZR(X),\pi^*\bigl(\sL^{-k}(-D)\bigr)\bigr)
    \longrightarrow H^i_{\pi^{-1}(Z)}\bigl(\ZR(X),\pi^*\sL^{-k}\bigr)
  \]
  is surjective for all $i$, respectively.
  Next, since $\cO_X \to \RR\pi_*\cO_{\ZR(X)}$ admits a left inverse by
  Theorem
  \ref{thm:kempflike}, the projection formula \cite[Chapitre 0, Proposition
  12.2.3]{EGAIII1}
  implies the desired vanishing and surjectivity on
  $X$.
\end{proof}
\section{Relative vanishing and injectivity theorems for klt pairs}\label{sect:kvvanishing}
Our goal in this section is to deduce our version of the Kawamata--Viehweg
vanishing theorem and Koll\'ar's injectivity theorem for klt pairs
(Theorem \ref{thm:kvvanishing}) from
Theorem \ref{thm:mainvanishing}.
To do so, we establish a covering lemma in \S\ref{sect:coveringlemma}.
We then prove the dual version of Theorem \ref{thm:kvvanishing} for normal
crossings pairs (Theorem \ref{thm:kvvanishingdual}) in
\S\ref{sect:kvvanishingproof}, and then prove Theorem \ref{thm:kvvanishing}
itself in \S\ref{sect:proofoftheorema}.
\subsection{A covering lemma}\label{sect:coveringlemma}
We start with the following version of Kawamata's covering lemma
(cf.\ \cite[Theorem 17]{Kaw81}).
\begin{lemma}[(cf.\ {\cite[Lemma 3.19]{EV92}})]\label{lem:evkawamatacover}
  Let $X$ be an integral regular scheme projective over an integral Noetherian local
  $\QQ$-algebra $(R,\fm)$.
  Let
  \[
    D = \sum_{j=1}^r D_j
  \]
  be a reduced simple normal crossings divisor and let $N_1,N_2,\ldots,N_r$ be
  positive integers.
  Then, there exists a finite surjective morphism $\tau \colon W\to X$ from a
  regular integral scheme $W$ such that
  \begin{enumerate}[label=$(\alph*)$,ref=\alph*]
    \item\label{lem:ev319a}
      We have $\tau^*D_j = N_j \cdot (\tau^*D_j)_\red$ for every $j \in
      \{1,2,\ldots,r\}$.
    \item\label{lem:ev319b} $\tau^*D$ is a simple normal crossings divisor.
    \item\label{lem:ev319c}
      The degree of $\tau$ divides some power of $\prod_{j=1}^r N_j$.
  \end{enumerate}
\end{lemma}
\begin{proof}
  We adapt the proof of \cite[Lemma 3.19]{EV92}.
  Write
  \[
    D = \sum_{j=1}^r D_j
  \]
  as a sum of regular divisors that are possibly disconnected.
  We construct $W$ inductively.
  It therefore suffices to prove the case when $N_1 > 1$ and $N_2 = N_3 = \cdots
  = N_r = 1$.
  \par Let $f\colon X \to \Spec(R)$ be the structure morphism for $X$.
  Let $\sA'$ be an $f$-very ample invertible sheaf on $X$, which exists since
  $f$ is projective.
  Then, there exists an integer $n > 0$ such that $\sA^{\prime\otimes
  n}(-D_1)$ is $f$-generated by \cite[Proposition 2.6.8$(i)$]{EGAII}.
  Letting $m > 0$ be an integer such that $N_1 \mid (n+m)$ and setting
  \[
    \sA \coloneqq \sA^{\prime\otimes(n+m)/N_1},
  \]
  we see that $\sA^{\otimes
  N_1}(-D_1)$ is $f$-very ample by \cite[Proposition 4.4.8]{EGAII}.
  \par We now construct divisors $H_1,H_2,\ldots,H_{\dim(X)}$ as follows.
  Since $\sA^{\otimes N_1}(-D_1)$ is $f$-very ample, there is a closed
  immersion $i\colon X \hookrightarrow \PP^N_R$ over $R$ such that
  $i^*\cO(1) \simeq \sA^{\otimes N_1}(-D_1)$.
  We now choose $\dim(X)$ general irreducible divisors
  \[
    H_1,H_2,\ldots,H_{\dim(X)} \in \bigl\lvert \sA^{\otimes N_1}(-D_1) \bigr\rvert
  \]
  as the vanishing of sections in $H^0(\PP^N_R,\cO(1))$ such that $D+\sum_i H_i$
  is a simple normal crossings divisor using the
  Bertini theorem in \cite[Theorem 2.17 and Remark 2.18]{BMPSTWW}.
  Since there are $\dim(X)+1$ divisors in the set
  $\{H_1,H_2,\ldots,H_{\dim(X)},D_1\}$, we know that
  \begin{equation}\label{eq:dimxplus1}
    \Biggl(\,\bigcap_{i=1}^{\dim(X)} H_i\Biggr) \cap D_1
    = \emptyset.
  \end{equation}
  \par Let $\tau_i\colon W_i \to X$ be the cyclic cover obtained by taking the $N_1$-th
  root out of $H_i+D_1$, which satisfies $\cO_X(H_i+D_1) \simeq \sA^{\otimes
  N_1}$ by \cite[(3.5)]{EV92}.
  Then, $(\ref{lem:ev319a})$ and $(\ref{lem:ev319c})$ are satisfied by
  construction, but $W_i$ may be
  singular over $H_i \cap D_1$ and $\tau_i^*(D)$ may not have simple normal
  crossings over $H_i \cap D_1$.
  To fix this, let $W$ be the normalization of
  \[
    W_1 \times_X W_2 \times_X \cdots \times_X W_{\dim(X)}.
  \]
  To show that $W$ is regular and that the pullback of $D$ to $W$ is a simple
  normal crossings divisor, we describe an alternative inductive
  construction of $W$.
  Let $W^{(\nu)}$ be the normalization of
  \[
    W_1 \times_X W_2 \times_X \cdots \times_X W_\nu
  \]
  and consider the composition
  \[
    \tau^{(\nu)}\colon W^{(\nu)} \longrightarrow W_1 \times_X W_1 \times_X
    \cdots \times_X W_\nu \longrightarrow X.
  \]
  Outside of the singular locus of $W^{(\nu)}$, we see that $W^{(\nu+1)}$ is
  obtained from $W^{(\nu)}$ by taking the $N_1$-th root out of
  \[
    \tau^{(\nu)*}(H_{\nu+1}+D_1) = \tau^{(\nu)*}(H_{\nu+1}) + N_1 \cdot
    (\tau^{(\nu)*}D_1)_\red.
  \]
  This is the same as taking the $N_1$-th root out of $\tau^{(\nu)*}(H_{\nu+1})$
  by \cite[Remark 3.3$(b)$ and Corollary 3.11]{EV92}.
  Since $\tau^{(\nu)*}(H_{\nu+1})$ has no singularities, \cite[Lemma 3.15]{EV92}
  implies the singularities of $W^{(\nu+1)}$ lie over the singularities of
  $W^{(\nu)}$, and hence inductively over $H_1 \cap D_1$.
  However, since $W$ is independent of the numbering of the $H_i$, the
  singularities of $W$ in fact lie over
  \[
    \bigcap_{i=1}^{\dim(X)} (H_i \cap D_1) = \Biggl(\,\bigcap_{i=1}^{\dim(X)} H_i\Biggr)
    \cap D_1 = \emptyset
  \]
  by the choice of the $H_i$ in \eqref{eq:dimxplus1}.
\end{proof}
\subsection{The Kawamata--Viehweg vanishing theorem and Koll\'ar's injectivity
theorem for normal crossings pairs}\label{sect:kvvanishingproof}
Before proving Theorem \ref{thm:kvvanishing}, we first
prove the Kawamata--Viehweg vanishing theorem and Koll\'ar's injectivity theorem
for normal crossings divisors on regular schemes of equal characteristic zero.
We follow the proof in
\cite[Theorem 2.64]{KM98}.\medskip
\par In the statement below, $(\ref{thm:kvvanishingdualvan})$ is a generalization of
\cite[Corollary 20]{Kol11}, and $(\ref{thm:kvvanishingdualinj})$ is a
(dual) version of \citeleft\citen{Kaw85}\citemid Theorem
3.2\citepunct \citen{EV87}\citemid Corollaire 1.11\citeright.
Theorem \ref{thm:kvvanishingdual} below implies the regular case of Theorem
\ref{thm:kvvanishing} by Proposition \ref{prop:aisastar}.
\begin{theorem}
  \label{thm:kvvanishingdual}
  Let $f\colon X \to Y$ be a proper maximally dominating morphism of locally
  Noetherian schemes of equal characteristic zero such that $X$ is regular.
  \par Let $\Delta$ be a $\QQ$-divisor on $X$ with normal crossings
  support such that $\lfloor \Delta \rfloor = 0$.
  Consider a divisor $L$ on $X$ such that $L \sim_\QQ M
  + \Delta$ for a $\QQ$-divisor $M$ on $X$.
  For every $y \in Y$, denote by $f_y\colon X_y \to \Spec(\cO_{Y,y})$ the base
  change of $f$ along $\Spec(\cO_{Y,y}) \to Y$, denote by $L_y$ the pullback
  of $L$ to $X_y$, and set $Z_y = f_y^{-1}(y)$.
  \begin{enumerate}[label=$(\roman*)$,ref=\roman*]
    \item\label{thm:kvvanishingdualvan} Suppose $M$ is $f$-nef and $f$-big.
      Then, we have
      \[
        H^i_{Z_y}\bigl(X_y,\cO_{X_y}(-L_y)\bigr) = 0
      \]
      for every $y \in Y$ and for all $i < \dim(X_y)$.
    \item\label{thm:kvvanishingdualinj} Suppose $M$ is $f$-semi-ample.
      Let $D$ be an effective divisor on $X$ for which there exists an integer
      $n > 0$ such that $nM$ is Cartier and an
      effective divisor $D'$ on $X$ such that $\cO_X(D+D')
      \simeq \cO_X(nM)$.
      Then,
      the canonical morphisms
      \[
        H^i_{Z_y}\bigl(X_y,\cO_{X_y}(-L_y-D_y)\bigr) \longrightarrow
        H^i_{Z_y}\bigl(X_y,\cO_{X_y}(-L_y)\bigr)
      \]
      induced by the inclusion $\cO_X \hookrightarrow \cO_X(D)$ are surjective
      for every $y \in Y$ for all $i$, where $D_y$ is the
      pullback of $D$ to $X_y$.
  \end{enumerate}
\end{theorem}
\begin{proof}
  We proceed in a sequence of steps.
  \begin{step}\label{step:kvvanishingexcellent}
    It suffices to show that when $f$ is surjective, $X$ is integral, and $Y =
    \Spec(R)$ for an excellent local domain $(R,\fm)$ with a dualizing complex
    $\omega_R^\bullet$,
    setting $Z = f^{-1}(\{\fm\})$, we have
    \[
      H^i_Z\bigl(X,\cO_{X}(-L)\bigr) = 0
    \]
    for all $i < \dim(X)$ for $(\ref{thm:kvvanishingdualvan})$ and
    \[
      H^i_{Z}\bigl(X,\cO_{X}(-L-D)\bigr) \longrightarrow
      H^i_{Z}\bigl(X,\cO_{X}(-L)\bigr)
    \]
    are surjective for all $i$ for
    $(\ref{thm:kvvanishingdualinj})$.
  \end{step}
  The desired vanishing and surjectivity are local conditions,
  and hence we can fix $y \in Y$.
  We first claim we may replace $f$ by its base change $\hat{f}_y\colon
  \hat{X}_y \to \Spec(\hat{\cO}_{Y,y})$ along the morphism
  $\Spec(\hat{\cO}_{Y,y}) \to Y$ for each $y \in Y$.
  Note that
  the pullback $\hat{\Delta}_y$ of $\Delta_y$ to
  $\hat{X}_y$ satisfies $\lfloor \hat{\Delta} \rfloor = 0$ and has
  normal crossings support
  by applying
  \cite[Lemme 7.9.3.1]{EGAIV2} to each stratum of an \'etale cover of
  $(X,\Delta)$ where the pullback of $\Delta$ has simple normal crossings
  support.
  Here we note that maximal ideals extend to maximal ideals for \'etale local
  maps of local rings \cite[Th\'eor\`eme 17.6.1]{EGAIV4},
  and hence taking \'etale covers is compatible with completion.
  The pullback of $M$ to $\hat{X}_y$ is $\hat{f}_y$-nef \cite[Lemma
  2.20(1)]{Kee03} for $(\ref{thm:kvvanishingdualvan})$ and
  $\hat{f}_y$-semi-ample \cite[Lemma 2.12]{CT20} for
  $(\ref{thm:kvvanishingdualinj})$.
  The morphism $\hat{f}_y$ is proper and maximally dominating by flat base
  change \cite[Expos\'e II, Proposition 1.1.5]{ILO14}, the ring
  $\hat{\cO}_{Y,y}$ is excellent by \cite[Scholie 7.8.3$(iii)$]{EGAIV2}, and
  the vanishing on local cohomology for $f$ descends
  from that on $\hat{f}$ by faithfully flat base change \cite[Theorem
  6.10]{HO08}.
  \par To show the pullback of $M$ to $\hat{X}_y$ is $\hat{f}_y$-big for
  $(\ref{thm:kvvanishingdualvan})$, let $n > 0$ be
  an integer such that $nM$ induces a generically finite morphism on every
  generic fiber of $f$.
  Then, by transitivity of fibers \cite[Corollaire 3.4.9]{EGAInew}, the
  compatibility of maps induced by linear systems and flat base change
  \cite[(4.2.10)]{EGAII}, and
  the fact that generically finite maps are stable under flat base change
  (combine \cite[Expos\'e II, Proposition 1.1.5]{ILO14} and the characterization
  of generically finite morphisms in \cite[Expos\'e II, Proposition
  1.1.7]{ILO14}), we see that the pullback of $nM$ to $\hat{X}_y$ induces
  generically finite morphisms along the generic fibers of
  $\hat{f}_y$.
  \par Finally, we repeat the argument of Step
  \ref{step:integral} of the proof of Theorem \ref{thm:maindualvanishing} to
  reduce to the case when $X$ and $Y$ are integral.
  Note that the excellence of $\cO_{Y,y}$ is not lost by 
  \cite[Scholie 7.8.3]{EGAIV2}.
  \begin{step}\label{step:kvvanishingsncfbigsa}
    The vanishing (resp.\ surjectivity)
    in Step \ref{step:kvvanishingexcellent} holds when
    $\Delta$ has simple normal crossings support and
    $M$ is $f$-big and $f$-semi-ample (resp.\ $f$-semi-ample).
  \end{step}
  By Theorem \ref{thm:maindualvanishing}, it
  suffices to reduce to the case when $\Delta = 0$.
  The idea is to induce on the number of components in $\Delta$, which we do by
  showing the following more general result:
  \begin{claim}\label{claim:km265}
    Let $X$ be an integral regular scheme projective over an integral Noetherian local
    $\QQ$-algebra $(R,\fm)$.
    Let $L$ be a Cartier divisor on $X$ such that
    \[
      L \sim_\QQ M + \sum_{j=1}^r a_jD_j,
    \]
    where the $D_j$ are regular (possibly disconnected)
    divisors, $\sum_j D_j$ is a simple normal crossings divisor, and the
    $a_j$ are rational numbers in $[0,1)$.
    Then, there is a finite surjective morphism $p\colon W \to X$ from a regular
    integral scheme $W$ and a divisor $M_W$ on $W$ such that $M_W \sim_\QQ
    p^*M$ and such that $\cO_X(-L)$ is a direct summand of $p_*\cO_W(-M_W)$.
  \end{claim}
  \begin{proof}[Proof of Claim \ref{claim:km265}]
    We proceed by induction on $r$.
    Write $a_1 = b/m$, where $m$ is a positive integer.
    By Lemma \ref{lem:evkawamatacover} applied to $\sum_j D_j$, $N_1 = m$, and
    $N_2 = N_3 = \cdots = N_r = 1$, there exists a finite surjective morphism
    $p_1\colon X_1 \to X$ such that $p_1^*D_1 \sim mD'$ for some divisor $D'$ on
    $X_1$.
    Moreover, each $p_1^*D_j$ is regular and $\sum_j p_1^*D_j$ is a simple
    normal crossings divisor.
    By \cite[Theorem 2.64, Step 1]{KM98} (see also \cite[Corollary 3.11]{EV92}),
    the canonical morphism $\cO_X \to
    p_{1*}\cO_{X_1}$ splits, and hence $\cO_X(-L) \to p_{1*}\cO_{X_1}(-p_1^*L)$ also
    splits.
    \par Now $D_1$ corresponds to a section of $\cO_{X_1}(mD')$, and hence we can
    take the associated $m$-th cyclic cover $p_2\colon X_2 \to X_1$ as in
    \cite[Definition 2.50]{KM98} (see also \cite[(3.5)]{EV92}).
    Then, \cite[Lemma 2.51]{KM98} (see also \cite[Lemma 3.15$(b)$]{EV92})
    implies that $X_2$ is regular, the $p_2^*p_1^*D_j$
    are regular, and $\sum_{j > 1} p_2^*p_1^*D_j$ is a simple normal crossings divisor.
    We have the decompositions
    \begin{align*}
      p_{2*}\cO_{X_2} &= \bigoplus_{\ell=0}^{m-1}\cO_{X_1}(-\ell D'),\\
      p_{2*}\cO_{X_2}(-p_2^*p_1^*L+b\,p_2^*D')
      &=
      \bigoplus_{\ell=0}^{m-1}
      \cO_{X_1}\bigl(-p_1^*L+
      (b-\ell)D'\bigr).
    \end{align*}
    The $\ell = b$ summand shows that $\cO_{X_1}(-p_1^*L)$ is a direct summand
    of $p_{2*}\cO_{X_2}(-p_2^*p_1^*L+b\,p_2^*D')$.
    \par We now have the $\QQ$-linear equivalence
    \[
      p_2^*p_1^*L-b\,p_2^*D' \sim_\QQ p_2^*p_1^*M + \sum_{j = 2}^r a_j\,p_2^*p_1^*D_j,
    \]
    which satisfies the hypotheses of Claim \ref{claim:km265}.
    By the inductive hypothesis, there exists a finite surjective morphism $W
    \to X_2$ satisfying the conclusion of Claim \ref{claim:km265} for $X_2$.
    The composition $W \to X_2 \to X$ then satisfies the conclusion of Claim
    \ref{claim:km265} for $X$.
  \end{proof}
  We now apply the finite surjective morphism $p\colon W \to X$ constructed in Claim
  \ref{claim:km265} to prove
  the special case in Step \ref{step:kvvanishingsncfbigsa}.
  \par To prove the vanishing in Step \ref{step:kvvanishingsncfbigsa}, we have an injection
  \[
    H^i_Z\bigl(X,\cO_X(-L)\bigr) \hooklongrightarrow
    H^i_{p^{-1}(Z)}\bigl(W,\cO_W(-M_W)\bigr).
  \]
  Since $M_W \sim_\QQ p^*M$, it is $(f \circ p)$-big (by \cite[Lemma 5.10]{LM}) and
  $(f \circ p)$-semi-ample (by \cite[Lemma 2.11$(i)$]{CT20}).
  Since regular local rings are pseudo-rational \cite[\S4]{LT81}, we see
  the right-hand side vanishes by Theorem
  \ref{thm:maindualvanishing}$(\ref{thm:maindualvanishingkv})$.
  \par To prove the surjectivity in Step \ref{step:kvvanishingsncfbigsa}, we
  have the commutative diagram
  \[
    \begin{tikzcd}
      H^i_Z\bigl(X,\cO_X(-L-D)\bigr) \rar\dar[hook] &
      H^i_Z\bigl(X,\cO_X(-L)\bigr) \dar[hook]\\
      H^i_{p^{-1}(Z)}\bigl(W,\cO_W(-M_W-p^*D)\bigr)
      \rar\arrow[twoheadrightarrow,bend right=30]{u} &
      H^i_{p^{-1}(Z)}\bigl(W,\cO_W(-M_W)\bigr)
      \arrow[twoheadrightarrow,bend right=30]{u}
    \end{tikzcd}
  \]
  where the surjective arrows pointing upwards
  are induced by the projection coming from the split
  injection in Claim \ref{claim:km265}.
  Since $M_W \sim_\QQ p^*M$, it is $f$-semi-ample by \cite[Lemma
  2.11$(i)$]{CT20}.
  Since regular local rings are pseudo-rational \cite[\S4]{LT81}, we see
  the bottom horizontal arrow is surjective by Theorem
  \ref{thm:maindualvanishing}$(\ref{thm:maindualvanishinggr})$.
  The commutativity of the diagram implies the top horizontal arrow is also
  surjective.
  \begin{step}\label{step:kvvanishingdualvanconcl}
    Conclusion of proof for $(\ref{thm:kvvanishingdualvan})$.
  \end{step}
  We first find a log resolution $g\colon \tilde{X} \to X$ on which we can write
  \[
    g^*M \sim_\QQ A + G
  \]
  where $A$ is an $(f \circ g)$-ample $\QQ$-divisor on $\tilde{X}$ and $G$ is an effective
  $\QQ$-divisor on $\tilde{X}$ such that $G\cup \Exc(g) \cup g_*^{-1}\Delta$
  has simple normal crossings support and the coefficients on $G$ are arbitrarily small.
  Applying Chow's lemma \cite[Th\'eor\`eme 5.6.1]{EGAII} and then taking a log
  resolution using \cite[Chapter I, \S3, Main Theorem I$(n)$]{Hir64}, we can
  find a projective log resolution $g_1\colon X_1 \to X$ of the pair $(X,\Delta)$
  such that $f \circ g_1$ is projective.
  Then, we know that $g_1^*M$ is $(f \circ g_1)$-big (by \cite[Lemma 5.10]{LM})
  and $(f \circ g_1)$-nef (by \cite[Lemma 2.17(1)]{Kee03}).
  By Kodaira's lemma \cite[Corollary 5.9]{LM}, we can write $g_1^*M \sim_\QQ A +
  E$ where $A$ is an $(f \circ g_1)$-ample $\QQ$-divisor and $E$ is an effective
  $\QQ$-divisor.
  Let $g_2\colon \tilde{X} \to X_1$ be a log resolution of $(X_1,E+g_1^*\Delta)$ and
  consider the composition
  \[
    g\colon \tilde{X} \overset{g_2}{\longrightarrow} X_1 \overset{g_1}{\longrightarrow} X.
  \]
  Then, we have $g_2^*M \sim_\QQ g_2^*A + g_2^*E$ and $g_2^*E \cup \Exc(g) \cup
  g_*^{-1}\Delta$ has simple normal crossings support.
  Since $g_2$ is constructed as a blowup of $X_1$ along regular centers, there
  exists an effective $g_2$-exceptional $\QQ$-divisor $F$ such that $-F$ is
  $g$-ample.
  After possibly replacing $F$ by a small rational multiple, we therefore see
  that $g^*M - F$ is $(f \circ g)$-ample by \cite[Proposition
  4.6.13$(ii)$]{EGAII} and $g_2^*E \cup F \cup \Exc(g) \cup g_*^{-1}(\Delta)$ is a
  $\QQ$-divisor with simple normal crossings support.
  Finally, for every integer $k \ge 0$, we can write
  \[
    g^*M \sim_\QQ \frac{1}{k+1}(k\,g^*M+g^*M - F) + \frac{1}{k+1}F
  \]
  where $k\,g^*M+g^*M - F$ is $(f \circ g)$-ample by the openness of the relative
  ample cone \citeleft\citen{Kee03}\citemid Theorem 3.9\citepunct
  \citen{Kee18}\citemid Theorem E2.2\citeright.
  By taking $k$ large, we can therefore set $A = \frac{1}{k+1}(k\,g^*M+g^*M - F)$
  and $G = \frac{1}{k+1}F$ to assume that the coefficients on $G$ are
  arbitrarily small.
  \par Now
  let $\omega_X$ and $\omega_{\tilde{X}}$ be canonical sheaves constructed by
  taking the unique cohomology sheaves of $f^!\omega_R^\bullet$ and $(f \circ
  g)^!\omega_R^\bullet$, respectively, and let $K_X$ and $K_{\tilde{X}}$ be
  associated canonical divisors (see \cite[Definition 6.2]{LM}).
  We need to show that
  \[
    H^i_Z\bigl(X,\cO_X(-L)\bigr) = 0
  \]
  for all $i < \dim(X)$,
  which by
  Proposition \ref{prop:aisastar} is implied by
  \[
    R^if_*\bigl(\omega_X(L)\bigr) = 0
  \]
  for all $i > 0$.
  \par As in \cite[Notation 2.6]{Kol13}, write
  \[
    K_{\tilde{X}} + g_*^{-1}\Delta \sim_\QQ g^*(K_X+\Delta) + \sum_i b_iG_i
  \]
  where the coefficients $b_i \in \QQ$ satisfy $b_i > -1$ for all $i$ since the
  pair $(X,\Delta)$ is klt \cite[Corollary 2.13 and Proposition 2.15]{Kol13}
  and the $G_i$ are $g$-exceptional.
  We then have
  \begin{align*}
    K_{\tilde{X}} + A  + g_*^{-1}\Delta+ G+\sum_i\bigl(\lceil b_i \rceil - b_i\bigr)G_i
    &\sim_\QQ
    g^*(K_X+\Delta) + A + G+\sum_i\lceil b_i \rceil G_i\\
    &\sim_\QQ g^*K_X + g^*L + \sum_i\lceil b_i \rceil G_i.
  \end{align*}
  We therefore see that the divisor
  \begin{align*}
    \tilde{L} &\coloneqq \bigl(g^*K_X-K_{\tilde{X}}\bigr)
    + g^*L + \sum_i\lceil b_i \rceil G_i
    \intertext{satisfies}
    \tilde{L} &\sim_\QQ A + g_*^{-1}\Delta + G+\sum_i\bigl(\lceil b_i \rceil -
    b_i\bigr)G_i
  \end{align*}
  which is the sum of the $(f \circ g)$-ample $\QQ$-divisor $A$ and an effective
  $\QQ$-divisor with simple normal crossings support and coefficients in
  $[0,1)$.
  Since $A$ is also $g$-ample by \cite[Proposition 4.6.13$(v)$]{EGAII},
  Step \ref{step:kvvanishingsncfbigsa} and
  Proposition \ref{prop:aisastar}$(\ref{prop:aisastarkv})$ imply
  \begin{equation}\label{eq:vanishingbystep2}
    R^i(f \circ g)_*\bigl(\omega_{\tilde{X}}(\tilde{L})\bigr) = 0
    \qquad \text{and} \qquad
    R^ig_*\bigl(\omega_{\tilde{X}}(\tilde{L})\bigr) = 0
  \end{equation}
  for all $i > 0$.
  \par We are now ready to show the theorem in case $(\ref{thm:kvvanishingdualvan})$.
  Consider the Leray spectral sequence
  \[
    E_2^{i,j} = R^if_*\bigl(R^jg_*\bigl(\omega_{\tilde{X}}(\tilde{L})\bigr)
    \Rightarrow R^{i+j}(f \circ g)_*\bigl(\omega_{\tilde{X}}(\tilde{L})\bigr).
  \]
  By \eqref{eq:vanishingbystep2}, the $E_2$ page of this spectral sequence
  is concentrated in the row $j = 0$.
  We therefore have an isomorphism
  \[
    R^if_*\bigl(g_*\bigl(\omega_{\tilde{X}}(\tilde{L})\bigr)\bigr) \simeq
    R^i(f \circ g)_*\bigl(\omega_{\tilde{X}}(\tilde{L})\bigr) = 0
  \]
  for all $i > 0$ using \eqref{eq:vanishingbystep2} for the vanishing on
  the right-hand side.
  Now by definition of $\tilde{L}$ and the projection formula \cite[Chapitre 0,
  Proposition 12.2.3]{EGAIII1}, we have
  \[
    g_*\bigl(\omega_{\tilde{X}}(\tilde{L})\bigr) \simeq \omega_X(L)
    \otimes_{\cO_X} g_*\biggl(\cO_{\tilde{X}}\biggl(\sum_i\lceil b_i \rceil
    G_i\biggr)\biggr) \simeq \omega_X(L)
  \]
  since the $G_i$ are $g$-exceptional (see \cite[Example 2.1.16]{Laz04a}).
  We therefore have
  \[
    R^if_*\bigl(\omega_X(L)\bigr) \simeq
    R^if_*\bigl(g_*\bigl(\omega_{\tilde{X}}(\tilde{L})\bigr)\bigr) = 0.
  \]
  \begin{step}
    Conclusion of proof for $(\ref{thm:kvvanishingdualinj})$.
  \end{step}
  Let $g\colon \tilde{X} \to X$ be a log resolution of $(X,\Delta)$.
  As in Step \ref{step:kvvanishingdualvanconcl}, the divisor
  \begin{align*}
    \tilde{L} &\coloneqq \bigl(g^*K_X-K_{\tilde{X}}\bigr)
    + g^*L + \sum_i\lceil b_i \rceil G_i
    \intertext{satisfies}
    \tilde{L} &\sim_\QQ g^*M + g_*^{-1}\Delta + \sum_i\bigl(\lceil b_i \rceil -
    b_i\bigr)G_i
  \end{align*}
  which is the sum of the $\QQ$-divisor $g^*M$, which is $(f \circ g)$-semi-ample by
  \cite[Lemma 2.11$(i)$]{CT20} and satisfies $\cO_{\tilde{X}}(g^*D+g^*D') \simeq
  \cO_{\tilde{X}}(n\,g^*M)$, and an effective $\QQ$-divisor with simple
  normal crossings support and coefficients in $[0,1)$.
  Now applying Step \ref{step:kvvanishingsncfbigsa} and
  Proposition \ref{prop:aisastar}$(\ref{prop:aisastargr})$ on $\tilde{X}$, the
  canonical morphisms
  \[
    R^i(f \circ g)_*\bigl(\omega_{\tilde{X}}(\tilde{L})\bigr) \longrightarrow 
    R^i(f \circ g)_*\bigl(\omega_{\tilde{X}}(\tilde{L}+g^*D)\bigr)
  \]
  are injective for all $i$.
  Since $g^*M$ is also $g$-semi-ample (by \cite[Lemma 2.10$(i)$]{CT20}) and
  $g$-big, Step \ref{step:kvvanishingsncfbigsa} and
  Proposition \ref{prop:aisastar}$(\ref{prop:aisastarkv})$ imply
  \[
    R^ig_*\bigl(\omega_{\tilde{X}}(\tilde{L})\bigr) = 0
  \]
  for all $i > 0$.
  The same argument using the Leray spectral sequence as in Step
  \ref{step:kvvanishingdualvanconcl} implies that the canonical morphisms
  \[
    R^if_*\bigl(\omega_{X}(L)\bigr) \longrightarrow 
    R^if_*\bigl(\omega_{X}(L+D)\bigr)
  \]
  are injective for all $i$.
\end{proof}
\subsection{The Kawamata--Viehweg vanishing theorem and Koll\'ar's injectivity
theorem for klt pairs}\label{sect:proofoftheorema}
We can now prove Theorem \ref{thm:kvvanishing}.
\begin{proof}[Proof of Theorem \ref{thm:kvvanishing}]
  For $(\ref{thm:kvvanishinginj})$,
  since the map $\cO_X \to \cO_X(D)$ factors the map $\cO_X \to \cO_X(D+D')$, we
  can replace $D$ by $D+D'$ to assume that $\cO_X(D) \simeq \cO_X(nM)$,
  and in particular, we may assume that $D$ is Cartier.
  \par Let $g\colon \tilde{X} \to X$ be a log resolution of the pair $(X,\Delta)$
  such that $\Exc(g) \cup g_*^{-1}\Delta$ has simple normal crossings support.
  Since $(X,\Delta)$ is klt, we can write
  \[
    K_{\tilde{X}} + g_*^{-1}\Delta \sim_\QQ g^*(K_X+\Delta) + \sum_i a_iE_i
  \]
  as in \cite[Notation 2.6]{Kol13},
  where the coefficients $a_i \in \QQ$ satisfy $a_i > -1$ for all $i$, the
  $E_i$ are exceptional, and $\lfloor g_*^{-1}\Delta \rfloor = 0$.
  We then have
  \begin{align*}
    K_{\tilde{X}} + g^*M + g_*^{-1}\Delta + \sum_i \bigl(\lceil a_i \rceil -
    a_i\bigr)E_i &\sim_\QQ g^*(K_X+M+\Delta) + \sum_i \lceil a_i \rceil E_i\\
    &\sim_\QQ g^*N + \sum_i \lceil a_i \rceil E_i.
  \end{align*}
  We therefore see that the divisor
  \[
    \tilde{N} \coloneqq g^*N + \sum_i \lceil a_i \rceil E_i
  \]
  satisfies
  \[
    \tilde{N} \sim_\QQ K_{\tilde{X}} + g^*M + g_*^{-1}\Delta + \sum_i
    \bigl(\lceil a_i \rceil - a_i\bigr)E_i.\smallskip
  \]
  \par In case $(\ref{thm:kvvanishingvan})$, we have now realized $\tilde{N}$ as
  $\QQ$-linearly equivalent to
  the sum of $K_{\tilde{X}}$, the $(f\circ g)$-big (by \cite[Lemma
  5.10]{LM}) and $(f\circ g)$-nef (by \cite[Lemma 2.17(1)]{Kee03})
  $\QQ$-divisor $g^*M$, and an effective $\QQ$-divisor with simple normal crossings
  support and coefficients in $[0,1)$.
  Since $g^*M$ is also $g$-nef (by the projection formula
  \cite[Proposition B.16]{Kle05}) and $g$-big, Theorem
  \ref{thm:kvvanishingdual}$(\ref{thm:kvvanishingdualvan})$
  and Proposition \ref{prop:aisastar}$(\ref{prop:aisastarkv})$ imply
  \begin{equation}\label{eq:vanishingforkvklt}
    R^i(f\circ g)_*\bigl(\cO_{\tilde{X}}(\tilde{N})\bigr) = 0
    \qquad \text{and} \qquad
    R^ig_*\bigl(\cO_{\tilde{X}}(\tilde{N})\bigr) = 0
  \end{equation}
  for all $i > 0$.
  \par We now consider the Leray spectral sequence
  \[
    E_2^{i,j} = R^if_*\bigl(R^jg_*\bigl(\cO_{\tilde{X}}(\tilde{N})\bigr)\bigr)
    \Rightarrow R^{i+j}(f \circ g)_*\bigl(\cO_{\tilde{X}}(\tilde{N})\bigr).
  \]
  By the vanishing $R^ig_*(\cO_{\tilde{X}}(\tilde{N})) = 0$,
  the $E_2$ page of this spectral sequence is
  concentrated in the row $j = 0$.
  We therefore have a natural isomorphism
  \[
    R^if_*\bigl(g_*\bigl(\cO_{\tilde{X}}(\tilde{N})\bigr)\bigr) \simeq R^i(f
    \circ g)_*\bigl(\cO_{\tilde{X}}(\tilde{N})\bigr) = 0
  \]
  for all $i > 0$ using \eqref{eq:vanishingforkvklt} for the vanishing on the
  right-hand side.
  Now by definition of $\tilde{N}$ and the projection formula \cite[Chapitre 0,
  Proposition 12.2.3]{EGAIII1}, we have
  \[
    g_*\bigl(\cO_{\tilde{X}}(\tilde{N})\bigr) \simeq \cO_X(N) \otimes_{\cO_X}
    g_*\biggl(\cO_{\tilde{X}}\biggl(\sum_i \lceil a_i \rceil E_i \biggr)\biggr)
    \simeq \cO_X(N)
  \]
  since $X$ is normal and the $E_i$ are $g$-exceptional (see \cite[Example
  2.1.16]{Laz04a}).
  We therefore have
  \[
    R^if_*\bigl(\cO_X(N)\bigr) \simeq
    R^if_*\bigl(g_*\bigl(\cO_{\tilde{X}}(\tilde{N})\bigr)\bigr) = 0.\smallskip
  \]
  \par In case $(\ref{thm:kvvanishinginj})$, we have now realized $\tilde{N}$ as
  $\QQ$-linearly equivalent to
  the sum of $K_{\tilde{X}}$, the $(f \circ g)$-semi-ample (by \cite[Lemma
  2.11$(i)$]{CT20}) $\QQ$-divisor $g^*M$ that satisfies $\cO_{\tilde{X}}(g^*D) \simeq
  \cO_{\tilde{X}}(n\,g^*M)$, and an effective $\QQ$-divisor with simple normal crossings
  support and coefficients in $[0,1)$.
  Theorem
  \ref{thm:kvvanishingdual}$(\ref{thm:kvvanishingdualinj})$
  and Proposition \ref{prop:aisastar}$(\ref{prop:aisastargr})$ imply the
  canonical morphisms
  \begin{equation}\label{eq:injectivityforkvklt}
    R^i(f \circ g)_*\bigl(\cO_{\tilde{X}}(\tilde{N})\bigr) \longrightarrow
    R^i(f \circ g)_*\bigl(\cO_{\tilde{X}}(\tilde{N}+g^*D)\bigr)
  \end{equation}
  are injective for all $i$.
  Since $g^*M$ is also $g$-semi-ample (by
  \cite[Lemma 2.10$(i)$]{CT20}) and $g$-big, Theorem \ref{thm:kvvanishingdual}
  and Proposition \ref{prop:aisastar}$(\ref{prop:aisastarkv})$ imply
  \[
    R^ig_*\bigl(\cO_{\tilde{X}}(\tilde{N})\bigr) = 0
  \]
  for all $i > 0$.
  The same argument using the Leray spectral sequence as in the previous
  paragraph implies that the canonical morphisms
  \[
    R^if_*\bigl(\cO_X(N)\bigr) \longrightarrow R^if_*\bigl(\cO_X(N+D)\bigr)
  \]
  are injective for all $i$.
\end{proof}

\begingroup
\makeatletter
\renewcommand{\@secnumfont}{\bfseries}
\part{Applications}\label{part:applications}
\makeatother
\endgroup

\section{Rational singularities}\label{sect:rationalsingularities}
In this section, we apply Theorem \ref{thm:zrmaindualvanishing} to study
rational singularities.
An interesting aspect of our proofs is that we can use the Zariski--Riemann
space to replace resolutions of singularities when they may not exist.
\par We start by proving our version of Boutot's theorem
(Theorem \ref{thm:boutot}) in
\S\ref{sect:boutot}.
In \S\ref{sect:pseudoratdeforms}, we prove that pseudo-rationality deforms in
equal characteristic zero.
We then show a version of Kempf's criterion for rational singularities and
pseudo-rational rings and show that derived splinters and rings with
rational singularities coincide for quasi-excellent $\QQ$-algebras.
In the last two subsections, we recall Bernasconi and Koll\'ar's result that
says that dlt pairs satisfy local rationality properties, and prove a criterion
for the Cohen--Macaulayness of Rees algebras and a Brian\c{c}on--Skoda theorem
for pseudo-rational rings.
All these results extend results known for rings essentially of finite type over
a field of characteristic zero to the context of (quasi-)excellent
$\QQ$-algebras.
\subsection{Boutot's theorem}\label{sect:boutot}
In this subsection, we prove our version of Boutot's theorem for
pseudo-rationality in equal characteristic zero (Theorem \ref{thm:boutot}).
This solves a conjecture of Boutot \cite[Remarque 1 on p.\ 67]{Bou87} and
answers a question of Schoutens \cite[(2) on p.\ 611]{Sch08} in the
affirmative.
\par We have modeled the proof of Theorem \ref{thm:boutot} after the proof of
the Boutot-type theorem for Du Bois singularities due to Godfrey and the author
\cite[Theorem A]{GM}.
The key insight is that $\RR\pi_*\cO_{\ZR(X)}$ can play the role the $0$-th
graded piece of the Deligne--Du Bois
complex $\underline{\Omega}^0_X$ did in our results for Du Bois
singularities.\medskip
\par We start by showing the following version of \cite[Proposition 3.1]{GM}.
As before, we will freely use the fact that the formation of $\ZR(X)$ commutes with
base change by quasi-compact separated \'etale morphisms
\cite[\href{https://stacks.math.columbia.edu/tag/087B}{Tag
087B}]{stacks-project}.
\begin{proposition}\label{prop:gmprop31}
  Let $f\colon W \to X$ be a surjective morphism between integral locally Noetherian
  schemes of equal characteristic zero.
  Suppose that for every $y \in X$, setting $X_y\coloneqq \Spec(\cO_{X,y})$ and
  denoting by $f_y\colon W_y \to X_y$ the base change of $f$ along
  $X_y \to X$, the natural morphism
  \[
    H^i_{\{y\}}(X_y,\cO_{X_y}) \longrightarrow
    \HH^i_{\{y\}}(W_y,\RR f_{y*}\cO_{W_y})
  \]
  is injective.
  If $W$ satisfies the condition in Theorem
  \ref{thm:kempflike}$(\ref{thm:kempflikepseudoratlike})$, then $X$ does also.
\end{proposition}
\begin{proof}
  By the functoriality of $\ZR(-)$ for dominant morphisms between integral
  schemes, we have the commutative diagram
  \[
    \begin{tikzcd}
      \ZR(X_y) \dar[swap]{\pi_{X_y}} & \lar \ZR(W_y) \dar{\pi_{W_y}}\\
      X_y & \lar[swap]{f_y} W_y\mathrlap{.}
    \end{tikzcd}
  \]
  Applying $\HH^i_{\{y\}}(X_y,-)$, we obtain the commutative diagram
  \[
    \begin{tikzcd}
      \HH^i_{\{y\}}\bigl(X_y,\RR\pi_{X_y*}\cO_{\ZR(X_y)}\bigr) \rar
      & \HH^i_{\{y\}}\bigl(X_y,\RR
      (f_{y}\circ\pi_{W_y})_*\cO_{\ZR(W_y)}\bigr)\\
      H^i_{\{y\}}(X_y,\cO_{X,y}) \rar[hook]\uar
      & \HH^i_{\{y\}}(X_y,\RR f_{y*}\cO_{W_y}) \arrow[hook]{u}
    \end{tikzcd}
  \]
  where the right vertical arrow is injective by Theorem
  \ref{thm:kempflike}, and the bottom horizontal arrow is an injection by
  hypothesis.
  By the commutativity of the diagram, we see that the left vertical arrow is an
  injection.
  Finally, since
  \[
    \HH^i_{\{y\}}\bigl(X_y,\RR\pi_{X_y*}\cO_{\ZR(X_y)}\bigr) \simeq
    H^i_{\pi_{X_y}^{-1}\{y\}}\bigl(\ZR(X_y),\cO_{\ZR(X_y)}\bigr),
  \]
  we see that $X_y$ satisfies the condition in Theorem
  \ref{thm:kempflike}$(\ref{thm:kempflikepseudoratlike})$.
\end{proof}
We can now prove the following Boutot-type theorem for pseudo-rationality.
More precisely, we show that the condition in Theorem
\ref{thm:kempflike}$(\ref{thm:kempflikepseudoratlike})$ descends for morphisms
of the type below.
Case $(\ref{thm:boutotnounramdsplit})$ for varieties is due to Kov\'acs
\cite[Theorem 1]{Kov00} and the
case when $X$ is excellent is due to Bernasconi and Koll\'ar
\cite[Proposition 2.14]{BK}.
Cases $(\ref{thm:boutotnounrampure})$--$(\ref{thm:boutotnounrampartpure})$
for varieties is due to Boutot \cite[Th\'eor\`eme on p.\ 65]{Bou87}.
The author of the present paper previously proved that case $(\ref{thm:boutotnounramfflat})$
holds in arbitrary characteristic as well \cite[Proposition 4.20]{Mur}.
\begin{theorem}\label{thm:boutotnounram}
  Let $f\colon W \to X$ be a surjective morphism between integral locally Noetherian
  schemes of equal characteristic zero.
  Assume one of the following holds:
  \begin{enumerate}[label=$(\roman*)$,ref=\roman*]
    \item\label{thm:boutotnounramdsplit}
      The natural morphism $\cO_X \to \RR f_*\cO_W$ admits a left inverse in
      the derived category of $\cO_X$-modules.
    \item\label{thm:boutotnounrampure}
      $f$ is affine, and for every affine open subset $U \subseteq X$, the
      $\cO_X(U)$-module map $\cO_X(U) \to \cO_W(f^{-1}(U))$ is pure.
    \item\label{thm:boutotnounramfflat} $f$ is faithfully flat.
    \item\label{thm:boutotnounrampartpure}
      $f$ is partially pure at every $y \in X$ in the sense that there is a
      $w \in W$ such that $f(w) = y$ and the map $\cO_{X,y} \to \cO_{W,w}$ is
      pure \emph{\cite[p.\ 38]{CGM16}}.
  \end{enumerate}
  If $Y$ satisfies the condition in Theorem
  \ref{thm:kempflike}$(\ref{thm:kempflikepseudoratlike})$, then $X$ does also.
\end{theorem}
\begin{proof}
  Any of the hypotheses above implies the hypothesis in Proposition
  \ref{prop:gmprop31} after possibly replacing $X$ and $W$ by maps on spectra of
  local rings.
  See the proof of \cite[Theorem 3.2]{GM}.
\end{proof}
In \cite[(2) on p.\ 611]{Sch08}, Schoutens asked the following: Given a
cyclically pure
map $R \to R'$ of $\QQ$-algebras, if $R'$ is locally
pseudo-rational, then is $R$ locally pseudo-rational?
Schoutens showed that in this situation, if $R'$ is regular, then
$R$ is locally pseudo-rational \cite[Main Theorem A]{Sch08}.
\par We answer Schoutens's question in the affirmative by showing Theorem
\ref{thm:boutot}.
This result also gives a new proof of Schoutens's theorem \cite[Main Theorem
A]{Sch08}.
\begin{proof}[Proof of Theorem \ref{thm:boutot}]
  By \cite[Corollary 3.12]{Has10}, $R$ is Noetherian and normal and $R \to
  R'$ is pure.
  By Theorem \ref{thm:boutotnounram}$(\ref{thm:boutotnounrampure})$, we know
  that $R$ satisfies the condition in Theorem
  \ref{thm:kempflike}$(\ref{thm:kempflikepseudoratlike})$.
  By Theorem \ref{thm:kempflike}, it therefore suffices to show that $R_\fp$ is
  analytically unramified for every prime ideal $\fp \subseteq R$.
  \par Let $\fp \subseteq R$ be a prime ideal.
  Then, there exists a maximal ideal $\fm \subseteq R'$ such that the map $R_\fp
  \to R'_\fm$ is pure by \cite[Lemma 2.2]{HH95}, and hence the map
  $(R_\fp)^\wedge \to (R'_\fm)^\wedge$
  on completions is pure \cite[Lemme A.2.2]{And18b}.
  Since $R'_\fm$ is analytically unramified, the completion $(R'_\fm)^\wedge$ is
  reduced, and hence $(R_\fp)^\wedge$ is also reduced.
  \par The ``in particular'' statement holds since if $R'$ is regular, then it
  is locally pseudo-rational \cite[\S4]{LT81}.
\end{proof}
\subsection{Deformation for pseudo-rationality}\label{sect:pseudoratdeforms}
We now prove that pseudo-rationality deforms.
This result is due to Elkik \cite{Elk78} when $R$ is essentially of finite type
over a field of characteristic zero and to the author \cite{Mur} when $R$ is a
quasi-excellent $\QQ$-algebra.
Here, the key insight is that the Zariski--Riemann space is functorial enough to
act as a replacement for the embedded resolutions of singularities used in
\cite{Elk78,Mur}.
\begin{theorem}[(cf.\ \citeleft\citen{Elk78}\citemid Th\'eor\`eme 5\citepunct
  \citen{Mur}\citemid Proposition 4.17\citeright)]\label{thm:elkik}
  Let $(R,\fm)$ be a Noetherian local $\QQ$-algebra and let $t \in \fm$ be a
  nonzerodivisor.
  If $R/tR$ is pseudo-rational, then $R$ is pseudo-rational.
\end{theorem}
\begin{proof}
  First, $R$ is normal by \cite[Proposition I.7.4]{Sey72} and Cohen--Macaulay by
  \cite[Theorem 17.3$(ii)$]{Mat89}.
  Next, $t$ maps to a nonzerodivisor in $\hat{R}$, and $\hat{R}/t\hat{R} \simeq
  (R/tR)^\wedge$ is reduced, and hence $\hat{R}$ is reduced by \cite[Proposition
  3.4.6]{EGAIV2}.
  \par Set $X \coloneqq \Spec(R)$ and $X_t \coloneqq \Spec(R/tR)$.
  By Lemma \ref{lem:kempflikepseudorat}, it remains to show that
  \[
    H^d_{\fm}(R) \overset{\delta^d_\pi}{\longrightarrow}
    H^d_{\pi^{-1}(\{\fm\})}\bigl(\ZR(X),\cO_{\ZR(X)}\bigr)
  \]
  is injective.
  Set $\ZR(X)_t \coloneqq \ZR(X) \times_X X_t$ and $\pi'_t\colon \ZR(X_t) \to
  X_t$.
  We claim we have a commutative diagram
  \[
    \begin{tikzcd}
      \ZR(X_t) \arrow[bend left=15]{drr}
      \arrow[bend right=15]{ddr}[swap]{\pi'_t} \arrow[dashed]{dr}\\
      & \ZR(X)_t \rar\dar{\pi_t} & \ZR(X)\dar{\pi}\\
      & X_t \rar[hook] & X
    \end{tikzcd}
  \]
  in the category of locally ringed spaces, where the square is Cartesian.
  By definition of limits, it suffices to show that for every morphism $f\colon
  W \to X$ in the inverse system defining $\ZR(X)$, there exists a morphism
  $f'_t\colon W'_t \to X_t$ in the inverse system defining $\ZR(X_t)$
  such that the diagram
  \[
    \begin{tikzcd}
      W'_t \arrow[bend left=15]{drr}
      \arrow[bend right=15]{ddr}[swap]{f'_t} \arrow[dashed]{dr}\\
      & W_t \rar\dar{f_t} & W \dar{f}\\
      & X_t \rar[hook] & X
    \end{tikzcd}
  \]
  commutes.
  First, choose an irreducible component $(W_t)_0$ of $W_t$ dominating $X_t$.
  Since $X_t$ is normal, $f$ is an isomorphism in codimension $1$, and hence
  $(W_t)_0$ maps birationally to $X_t$.
  We can then apply Chow's lemma \cite[Corollaire 5.6.2]{EGAII} to find a
  projective birational morphism $W'_t \to (W_t)_0$ such that the composition
  $W'_t \to X_t$ is projective birational, which proves the claim.
  \par Next, consider the commutative diagram
  \[
    \begin{tikzcd}
      0\dar & 0 \dar\\
      H^{d-1}_\fm(R/tR)\dar\rar[hook]
      & H^{d-1}_{\pi_t^{-1}(\{\fm\})}\bigl(\ZR(X)_t,\cO_{\ZR(X)_t}\bigr) \rar \dar
      & H^{d-1}_{\pi_t^{\prime-1}(\{\fm\})}\bigl(\ZR(X_t),\cO_{\ZR(X_t)}\bigr)\\
      H^d_\fm(R) \dar[swap]{t \cdot-}\rar{\delta^d_\pi}
      & H^d_{\pi^{-1}(\{\fm\})}\bigl(\ZR(X),\cO_{\ZR(X)}\bigr) \dar{t \cdot-}\\
      H^d_\fm(R) \dar\rar{\delta^d_\pi}
      & H^d_{\pi^{-1}(\{\fm\})}\bigl(\ZR(X),\cO_{\ZR(X)}\bigr) \dar\\
      0 & 0
    \end{tikzcd}
  \]
  with exact columns
  where the left half is obtained from \cite[Functoriality II.9.7]{Ive86}.
  The top left horizontal arrow is injective since the composition in the top
  row is
  injective by Lemma \ref{lem:kempflikepseudorat}, where we
  use the fact that the edge maps in Definition
  \ref{def:pseudorational}$(\ref{def:pseudoratbiratcond})$ behave well under
  composition of morphisms \cite[Proposition 1.12]{Smi97rat}.
  The columns are exact at the top by the fact that $R$ is Cohen--Macaulay and by
  Theorem \ref{thm:zrmaindualvanishing}$(\ref{thm:zrmaindualvanishingkv})$.
  \par Now suppose there exists an element $0 \ne \eta \in
  \ker(\delta^d_\pi)$.
  Since every element in $H^d_\fm(R)$ is annihilated by a power of $t$, after
  multiplying $\eta$ by a power of $t$ we may assume that $t\eta = 0$, in which
  case $\eta$ lies in the image of $H^{d-1}_\fm(R/tR)$ in the left column.
  The commutativity of the diagram implies that the composition
  \[
    H^{d-1}_\fm(R/tR) \longrightarrow
    H^d_{\pi^{-1}(\{\fm\})}\bigl(\ZR(X),\cO_{\ZR(X)}\bigr)
  \]
  is injective.
  Since $\eta \in \ker(\delta^d_\pi)$ by assumption, this shows that $\eta
  = 0$, a contradiction.
\end{proof}
\subsection{Kempf's criterion}\label{sect:kempfwithomega}
We show that Kempf's criterion \cite[Proposition on p.\
50]{KKMSD73} holds for quasi-excellent
schemes of equal characteristic zero with dualizing complexes.
This gives a proof of Kempf's criterion in equal characteristic zero
independent of \cite[Theorem 8.6]{Kov}.
\par The trace morphism below comes from the adjunction $\RR f_*
\dashv f^!$ in Grothendieck duality \citeleft\citen{Har66}\citemid Chapter VI,
Corollary 3.4$(b)$\citepunct \citen{Con00}\citemid Lemma 3.4.3\citeright.
\begin{proposition}[(cf.\ {\citeleft\citen{KKMSD73}\citemid Proposition on p.\
  50\citepunct \citen{Lip94}\citemid Lemma 4.2\citeright})]\label{prop:kempf}
  Let $f\colon X \to Y$ be a proper birational morphism of Noetherian
  schemes of equal characteristic zero
  such that $X$ is regular and such that $Y$ has a dualizing
  complex $\omega_Y^\bullet$.
  Denote by $\omega_X$ the unique cohomology sheaf of $f^!\omega_Y^\bullet$
  (after possibly applying shifts on each connected
  component of $X$).
  \par The following are equivalent:
  \begin{enumerate}[label=$(\roman*)$,ref=\roman*]
    \item\label{prop:kempfkovacstype}
      $Y$ is normal and $R^if_*\cO_X = 0$ for all $i > 0$.
    \item\label{prop:kempfpseudorat}
      $Y$ is Cohen--Macaulay and the trace morphism
      $f_*\omega_X \to \omega_Y$ is an isomorphism.
  \end{enumerate}
\end{proposition}
\begin{proof}
  By \cite[Lemma 4.2]{Lip94}, even without characteristic zero hypotheses,
  $(\ref{prop:kempfkovacstype})$ holds if and only if the trace morphism
  $\RR f_*\omega^\bullet_X \to \omega^\bullet_Y$ is a quasi-isomorphism.
  \par We want to show that $\RR f_*\omega^\bullet_X \to \omega^\bullet_Y$ is a
  quasi-isomorphism if and only if $(\ref{prop:kempfpseudorat})$ holds.
  Since $X$ is regular, after possibly applying shifts on each connected
  component of $X$, $\omega^\bullet_X$ is concentrated in one degree.
  Moreover, $R^if_*\omega_X = 0$ for all $i > 0$ by
  Theorem \ref{thm:mainvanishing}$(\ref{thm:mainvanishingkv})$ applied to
  $\sL = \cO_X$.
  We therefore see that $\RR f_*\omega^\bullet_X \to \omega^\bullet_Y$ is a
  quasi-isomorphism if and only if $f_*\omega_X \to \omega_Y$ is an isomorphism
  and $\omega_Y^\bullet$ is concentrated in one degree.
  But the condition that $\omega_Y^\bullet$ is concentrated in one degree is
  equivalent to $Y$ being Cohen--Macaulay by local duality \cite[Chapter V,
  Corollary 6.3]{Har66}.
\end{proof}
\subsection{Derived splinters}\label{sect:derivedsplinters}
We show that Kov\'acs's splitting criterion for rational singularities
\cite[Theorem 3]{Kov00} holds for quasi-excellent schemes of equal
characteristic zero.
The direction $\Rightarrow$ gives an independent proof of \cite[Theorem 8.7]{Kov}
in equal characteristic zero.
Recall that a scheme $S$ is a \textsl{derived splinter} if for
every proper surjective morphism $f\colon X \to S$, the natural morphism
$\cO_S \to \RR f_*\cO_X$ splits in the derived category of coherent
sheaves on $S$ \cite[Definition 1.3]{Bha12}.
\begin{theorem}[(cf.\ {\citeleft\citen{Kov00}\citemid Theorem 3\citepunct
  \citen{Bha12}\citemid Theorem 2.12\citeright})]\label{thm:derivedsplinters}
  Let $S$ be a quasi-excellent Noetherian scheme of equal characteristic zero.
  Then, $S$ is a derived splinter if and only if $S$ has rational singularities.
\end{theorem}
\begin{proof}
  $\Rightarrow$.
  Let $\pi\colon \ZR(S) \to S$ denote the canonical projection morphism from the
  Zariski--Riemann space of $S$.
  For every admissible blowup $\pi_\sI\colon S_\sI \to S$ and every $s \in S$,
  setting $S_s \coloneqq \Spec(\cO_{S,s})$ and $\pi_{\sI,s}\colon S_{s,\sI_s}
  \to S_s$, we know that the morphism
  \begin{align*}
    H^i_{\{s\}}\bigl( S_s,\cO_{S_s} \bigr) &\xrightarrow{\delta^i_{\pi_{\sI,s}}}
    H^i_{\pi_{\sI,s}^{-1}(\{s\})}\bigl(S_{s,\sI_s},\cO_{S_{s,\sI_s}}\bigr)
    \intertext{is injective for every $i$ since it admits a left inverse.
    Taking colimits over all $\sI \in \AId_S$ and applying Theorem
    \ref{lem:lccolimits}, we see that the morphism}
    H^i_{\{s\}}\bigl( S_s,\cO_{S_s} \bigr) &\xrightarrow{\delta^i_{\pi_{s}}}
    H^i_{\pi_{s}^{-1}(\{s\})}\bigl(\ZR(S_s),\cO_{\ZR(S_s)}\bigr)
  \end{align*}
  is injective for every $i$.
  By Theorem \ref{thm:kempflike}, this shows $S$ has rational singularities.
  \par $\Leftarrow$.
  Let $g\colon Y \to S$ be a proper surjective morphism.
  We want to show that $\cO_S \to \RR g_*\cO_Y$ splits in the derived category
  of coherent sheaves on $S$.
  By Chow's lemma \cite[Corollaire 5.6.2]{EGAII}, we may assume that $g$ is
  projective.
  By homogeneous prime avoidance \cite[Chapter III, \S1, no.\ 4,
  Proposition 8]{BouCA}, we can take repeated hyperplane sections of $Y$ to
  assume that $g$ generically finite.
  By the Raynaud--Gruson flattening theorem \cite[Premi\`ere partie,
  Th\'eor\`eme 5.2.2]{RG71}, we can then find a commutative square that fits
  into a diagram
  \[
    \begin{tikzcd}
      & Y'' \rar\dar[swap]{h} & Y\dar{g}\\
      W \rar{\mu} & Y' \rar{b} & S
    \end{tikzcd}
  \]
  where $b$ is a normalized blowup of $S$, $h$ is the strict transform of $g$
  along $b$, and $h$ is finite flat.
  We take $\mu\colon W \to Y'$ to be a resolution of singularities, which
  exists by \cite[Theorem 1.1]{Tem08}.
  Now consider the corresponding commutative diagram
  \[
    \begin{tikzcd}
      & \RR (b \circ h)_*\cO_{Y''} \dar[bend right=30,dashed] & \lar \RR g_*\cO_Y\\
      \RR (b \circ \mu)_*\cO_W \arrow[bend right=15,dashed]{rr}
      & \lar \RR b_*\cO_{Y'} \uar & \lar \cO_S\mathrlap{.} \uar
    \end{tikzcd}
  \]
  Since we are in equal characteristic zero and $Y'$ is normal, there is a
  section of $\cO_{Y'} \to h_*\cO_{Y''} \simeq \RR h_*\cO_{Y''}$ coming from the
  trace map \cite[Example 2.1]{Bha12}.
  By definition of rational singularities, the composition $\cO_S \to \RR (b
  \circ \mu)_*\cO_W$ is a quasi-isomorphism, and in particular, splits.
  Thus, $\cO_S \to \RR b_*\cO_{Y'}$ splits, which shows that
  $\cO_S \to \RR g_*\cO_Y$ splits by the commutativity of the diagram.
\end{proof}
\subsection{Rationality of dlt pairs}
In \cite{BK}, Bernasconi and Koll\'ar proved the following theorem without
characteristic assumptions by assuming that certain vanishing theorems hold and that
thrifty log resolutions exist.
Since the required vanishing statements hold in equal characteristic zero by Theorem
\ref{thm:kvvanishing}$(\ref{thm:kvvanishingvan})$ and thrifty log resolutions
are known to exist by \cite{Tem18}, we obtain the following
unconditional result.
We note Theorem \ref{thm:bk31} below was already noted in \cite[p.\ 2857]{BK} as a
consequence of Theorem \ref{thm:kvvanishing}$(\ref{thm:kvvanishingvan})$ in this
paper.
\begin{theorem}[(see {{\cite[Theorem 3.1]{BK}}})]\label{thm:bk31}
  Let $X$ be an excellent normal Noetherian scheme of equal characteristic zero
  with a dualizing complex $\omega_X^\bullet$.
  Let $\Delta$ be an effective $\RR$-Weil divisor on $X$ such that $(X,\Delta)$
  is dlt.
  Then, we have the following:
  \begin{enumerate}[label=$(\roman*)$,ref=\roman*]
    \item $X$ is Cohen--Macaulay.
    \item Let $D$ be a Weil divisor on $X$ such that $D+\Delta_D$ is
      $\RR$-Cartier for some $0 \le \Delta_D \le \Delta$.
      Then, $\cO_X(D)$ is Cohen--Macaulay.
    \item For every reduced Weil divisor $B \subseteq \lfloor \Delta \rfloor$,
      the pair $(X,B)$ is rational in the sense of \emph{\cite[Definition
      2.80]{Kol13}}.
      In particular, $X$ has rational singularities.
    \item Log canonical centers of $(X,\Delta)$ are normal and have rational
      singularities.
  \end{enumerate}
\end{theorem}
\begin{proof}
  We can apply \cite[Theorem 3.1]{BK} since Grauert--Riemenschneider vanishing
  holds in this setting by Theorem
  \ref{thm:kvvanishing}$(\ref{thm:kvvanishingvan})$, and thrifty log resolutions
  exist in this setting by \cite[Theorem 1.1.6]{Tem18}.
  Recall that if $X$ is normal and $D \subseteq X$ is a reduced Weil divisor,
  then a resolution $f\colon Y \to X$ is \textsl{thrifty} if $(Y,f_*^{-1}D)$ is
  snc (i.e., $Y$ is regular and $f_*^{-1}D$ has simple normal crossings),
  $f$ is an isomorphism over the generic point of
  every stratum of $\snc(X,D)$, and $f$ is an isomorphism at the generic point of
  every stratum of $(Y,f_*^{-1}D)$ (see \cite[Definition 2.79]{Kol13}).
\end{proof}
\begin{remark}\label{rem:bkcm}
  For the statement of Theorem \ref{thm:kvvanishing}, we need to know that if $X$
  is an excellent normal Noetherian scheme of equal characteristic zero with a
  dualizing complex $\omega_X^\bullet$ and there exists an effective $\QQ$-Weil
  divisor such that $(X,\Delta)$ is klt, then $X$ is Cohen--Macaulay.
  The proof of Theorem \ref{thm:bk31} in \cite{BK} only needs
  Theorem \ref{thm:kvvanishing} in the regular case.
  We may therefore use the regular case of Theorem \ref{thm:kvvanishing} to
  deduce that klt pairs are Cohen--Macaulay in the proof of the klt case of
  Theorem \ref{thm:kvvanishing}.
  Note that the fact that $X$ is Cohen--Macaulay is not used in the proof that
  the regular case of Theorem \ref{thm:kvvanishing} implies the klt case (see
  \S\ref{sect:proofoftheorema}).
\end{remark}
\subsection{Cohen--Macaulayness of Rees algebras and a Brian\c{c}on--Skoda
theorem for rational singularities}
Finally, we give a criterion for when Rees algebras and
associated graded rings of graded sequences of ideals are Cohen--Macaulay.
These results are essentially due to Sancho de Salas \cite{SdS87} and Lipman
\cite{Lip94}, who proved the Cohen--Macaulayness of $G_{\FF^{(e)}}$ and
$R_{\FF^{(e)}}$ assuming certain vanishing theorems hold for the projection morphism
$f\colon X \to \Spec(R)$.
In particular, Sancho de Salas noted the relevant vanishing theorems hold for
rings of finite type over $\CC$ \cite[Theorem 2.8$(a)$]{SdS87}.
We can prove the following statement for quasi-excellent rings
thanks to our version of the
Grauert--Riemenschneider vanishing theorem (Theorem
\ref{thm:maindualvanishing}$(\ref{thm:maindualvanishingkv})$).
See also \citeleft\citen{SdS87}\citemid Theorem 1.4 and Corollary 1.6\citepunct
\citen{Lip94}\citemid Theorems 4.1 and 4.3\citeright\ for partial converses to
these results.
\begin{theorem}[({see \citeleft\citen{SdS87}\citemid Theorem 1.7\citepunct
  \citen{Lip94}\citemid Theorems 4.1 and 4.3\citeright})]\label{thm:sdslip}
  Let $(R,\fm)$ be a Noetherian local $\QQ$-algebra of dimension $d$.
  Let $\FF \coloneqq (F_i)_{i=0}^\infty$ be a graded sequence of ideals in $R$ such
  that $F_0 = R$.
  Suppose that
  \[
    R_\FF \coloneqq \bigoplus_{i=0}^\infty F_i t^i
  \]
  is Noetherian of dimension $d+1$.
  Set $X \coloneqq \Proj_R(R_\FF)$ with
  canonical projection $f\colon X \to \Spec(R)$.
  \begin{enumerate}[label=$(\roman*)$,ref=\roman*]
    \item\label{thm:lip43}
      If $R$ is Cohen--Macaulay and $X$ is locally pseudo-rational, then
      \[
        G_{\FF^{(e)}} \coloneqq \bigoplus_{n=1}^\infty F_{en}/F_{e(n+1)}
      \]
      is Cohen--Macaulay for some $e > 0$.
    \item\label{thm:lip41} If $R$ and $X$ are locally pseudo-rational and
      locally quasi-excellent, then
      \[
        R_{\FF^{(e)}} \coloneqq \bigoplus_{n=1}^\infty F_{en}t^n
      \]
      is Cohen--Macaulay and generated in degree $1$ for some $e > 0$.
  \end{enumerate}
\end{theorem}
\begin{proof}
  In either case, note that $X$ is Cohen--Macaulay by
  definition of pseudo-rationality.
  \par For $(\ref{thm:lip43})$, we apply \cite[Theorem 4.3]{Lip94}.
  It suffices to show that
  \[
    H^i_{f^{-1}(\{\fm\})}(X,\cO_X) = 0
  \]
  for all $i < d$.
  This vanishing holds by Theorem
  \ref{thm:maindualvanishing}$(\ref{thm:maindualvanishingkv})$.
  \par For $(\ref{thm:lip41})$, we apply \cite[Theorem 4.1]{Lip94}.
  It suffices to show that
  \[
    \cO_{\Spec(R)} \longrightarrow \RR f_*\cO_X
  \]
  is a quasi-isomorphism.
  Since $R$ and $X$ are locally pseudo-rational and locally quasi-excellent,
  Theorem \ref{thm:kempflike} implies
  that both the composition
  \[
    \cO_{\Spec(R)} \longrightarrow \RR f_*\cO_X \overset{\sim}{\longrightarrow}
    \RR\pi_*\cO_{\ZR(X)}
  \]
  and the second morphism in this composition are quasi-isomorphisms in the
  derived category of $\cO_{\Spec(R)}$-modules, where $\pi\colon \ZR(X) \to
  \Spec(X)$ is the canonical projection.
  Since the composition $\ZR(X) \to X \to \Spec(R)$ is isomorphic to the
  projection morphism $\ZR(\Spec(R)) \to \Spec(R)$ by definition of the
  Zariski--Riemann space, we see that $\cO_{\Spec(R)} \to \RR f_*\cO_X$ is a
  quasi-isomorphism.
\end{proof}
This yields the following characterization of rational singularities in
equal characteristic zero, which for rings essentially of finite type over a
field is due to Lipman \cite{Lip94}.
\begin{corollary}[(see {\cite[p.\ 149]{Lip94}})]\label{cor:lip149}
  Let $(R,\fm)$ be a quasi-excellent Noetherian local $\QQ$-algebra.
  Then, $R$ has rational singularities if and only if there exists an ideal $I
  \subseteq R$ such that $R[It]$ is Cohen--Macaulay and $\Proj_R(R[It])$ is
  regular.
\end{corollary}
\begin{proof}
  $\Rightarrow$.
  Let $X \to \Spec(R)$ be any projective resolution of singularities, which
  exists by \cite[Chapter I, \S3, Main Theorem I$(n)$]{Hir64}.
  Then, $X \simeq \Proj_R(R[Jt])$ for some ideal $J \subseteq R$ by
  \cite[Corollaire 2.3.7]{EGAIII1}.
  By Theorem \ref{thm:sdslip}$(\ref{thm:lip41})$, $R[J^et]$ is Cohen--Macaulay for
  some $e > 0$.
  Since
  \[
    \Proj_R\bigl(R[Jt]\bigr) \simeq \Proj_R\bigl(R[J^et]\bigr),
  \]
  setting $I = J^e$ works.
  \par $\Leftarrow$.
  By \cite[Theorem 4.1]{Lip94}, we know that $\cO_{\Spec(R)} \to \RR f_*\cO_X$
  is a quasi-isomorphism.
  This implies $R$ has rational singularities since the definition of rational
  singularities is independent of the resolution of singularities chosen
  \cite[Remark 4.13]{Mur}.
\end{proof}
We also obtain the following version of the Brian\c{c}on--Skoda theorem for
rings with rational singularities.
This statement strengthens the Brian\c{c}on--Skoda theorem shown by Lipman and
Teissier for pseudo-rational rings in arbitrary characteristic \cite[Theorem
2.1]{LT81}.
The statement below is due to Huneke when $R$ is essentially of finite type over
a field of characteristic zero \cite{Hun00}.
\begin{corollary}[({cf.\ \cite[Corollary 4.8]{Hun00}})]\label{cor:bsrational}
  Let $R$ be a normal locally excellent Noetherian $\QQ$-algebra with rational
  singularities.
  Then, for every ideal $I \subseteq R$ of analytic spread $\ell$, we have
  \[
    \overline{I^n} \subseteq I^{n-\ell}
  \]
  for all $n \ge \ell$.
\end{corollary}
\begin{proof}
  Since the question is local, we can localize at every maximal ideal in $R$ to
  assume that $R$ is local.
  By \cite[Corollary 4.8]{Hun00}, it suffices to show there exists an ideal $J
  \subseteq R$ such that $\Proj_R(R[Jt])$ is regular and $R[Jt]$ is
  Cohen--Macaulay.
  This condition is equivalent to $R$ having rational singularities by Corollary
  \ref{cor:lip149}.
\end{proof}

\section{Complete local UFDs of dimension
\texorpdfstring{$\le$}{\unichar{"02264}} 4 with residue field
\texorpdfstring{$\CC$}{C}
are Gorenstein}\label{sect:applications}
In this section, we give another application of our vanishing theorems that is
not related to rational singularities.
In \cite[p.\ 17]{Sam61}, Samuel asked whether every UFD is Cohen--Macaulay.
While the answer is no in general \cite[Proposition on p.\ 655]{Ber67} (see also
\cite[\S4]{Lip75} for a survey), all complete local UFDs $(R,\fm)$
such that $R/\fm \simeq \CC$ are $S_3$.
This result is due to Raynaud (unpublished), Danilov \cite[Theorem 2]{Dan70},
Boutot \cite[Corollaire on p.\ 693]{Bou73}, and
Hartshorne--Ogus (when $R$ is algebrizable)
\cite[Theorem 2.5]{HO74}.
Hartshorne and Ogus showed that in dimensions $\le4$, such rings are in fact
Gorenstein \cite[Corollary 2.6]{HO74}.\smallskip
\par We remove the algebrizability condition in Hartshorne and Ogus's proof
using our vanishing theorems.
Hartshorne and Ogus's proof uses resolutions of singularities
\cite{Hir64} and the relative Grauert--Riemenschneider vanishing theorem \cite[Satz
2.3]{GR70}.
\begin{theorem}[(see \citeleft\citen{HO74}\citemid Corollary 2.6\citeright)]\label{thm:ho}
  Let $(R,\fm)$ be a complete local UFD of dimension $\le4$ such that $R/\fm
  \simeq \CC$.
  Then, $R$ is Gorenstein.
\end{theorem}
\begin{proof}
  Let $f\colon Z \to \Spec(R)$ be a resolution of singularities, and set $M^i
  \coloneqq R^if_*\cO_Z$ for every $i$.
  By \cite[Proposition 2.4]{HO74}, we have $M^1 = 0$.
  Replacing Hartshorne and Ogus's dual version of the Grauert--Riemenschneider
  vanishing theorem \cite[Proposition 2.2]{HO74} by Theorem
  \ref{thm:maindualvanishing}$(\ref{thm:maindualvanishingkv})$ in the proof of
  \cite[Theorem 2.3]{HO74}, we then see that $R$ satisfies $S_3$.
  Since $\dim(R) \le 4$, we see that
  $R$ satisfies Hartshorne and Ogus's condition $C$ \cite[Definition
  1.7]{HO74}, i.e., for every prime ideal $\fp \subseteq R$, we have
  \[
    \depth(R_\fp) \ge \min\biggl\{\dim(R_\fp),\frac{1}{2}\dim(R_\fp)+1\biggr\}.
  \]
  This implies $R$ is Gorenstein by \cite[Corollary 1.8]{HO74}.
\end{proof}

\section{Extensions to other categories}
In this section, we prove Theorem \ref{thm:kvvanishingothercats}, which extends
our main vanishing and injectivity theorem to other categories of spaces.
While most objects and notions appearing in the statement of Theorem
\ref{thm:kvvanishingothercats} were either defined or checked to be compatible
with existing definitions and relative GAGA in \cite[\S\S23--25]{LM}, the notion of bigness
in \cite[Defintion 25.5]{LM} requires projectivity assumptions to apply relative
GAGA theorems.
We therefore define relative bigness below before proving Theorem
\ref{thm:kvvanishingothercats} in the next subsection.
\begin{remark}
  The projectivity assumption on $f$ is necessary in cases
  $(\ref{setup:introformalqschemes})$,
  $(\ref{setup:introberkovichspaces})$,
  $(\ref{setup:introrigidanalyticspaces})$, and
  $(\ref{setup:introadicspaces})$
  since as far as we aware, appropriate versions of Chow's lemma for
  proper bimeromorphic morphisms in these categories do not exist in the
  literature.
  See \cite[\S1.6]{AT19} and \cite[Th\'eor\`eme 6.6]{Duc21}.
  We have written our proof of Theorem \ref{thm:kvvanishingothercats} so that if
  such a version of Chow's lemma becomes available in one of these
  categories, then Theorem \ref{thm:kvvanishingothercats} would also be true in
  that category.
\end{remark}
\subsection{Relatively big invertible sheaves and locally Moishezon morphisms}
In this subsection, we define relatively big invertible sheaves and locally
Moishezon morphisms in the
categories of spaces appearing in Theorem
\ref{thm:kvvanishingothercats}.\medskip
\par To be able to state Theorem
\ref{thm:kvvanishingothercats}$(\ref{thm:kvvanishingothercatsvan})$ without
projectivity assumptions, we need to define relative bigness without
relying on GAGA or the existence of a relatively ample invertible sheaf.
Instead, we adapt Nakayama's definition of relative bigness for
complex-analytic spaces \cite[Definition
on p.\ 568]{Nak87}, which uses relative Proj.
We define relative Proj as in
\cite[\href{https://stacks.math.columbia.edu/tag/084C}{Tag
084C}]{stacks-project} for algebraic spaces,
\cite[Chapter II, \S1.b]{Nak04}
for complex analytic spaces,
\cite[5.4]{Duc21} for Berkovich spaces,
\cite[Definition 2.3.3]{Con06} for rigid-analytic spaces,
and
\cite[Definition 6.7]{Zav} for adic spaces.
All these constructions are compatible with relative GAGA over affinoid
subdomains in the base, since locally, relative Proj is a closed subscheme of a
projective space over an affinoid space.
\par We now define relative bigness for invertible sheaves and Cartier divisors.
We restrict to the case of morphisms between integral schemes for simplicity.
The idea is to use functorial properties of Proj to discuss bigness instead of
generic fibers as Nakayama does in the complex-analytic case in \cite[Definition
on p.\ 568]{Nak87} (see also \cite[p.\ 1666]{Kol22})
since generic fibers to do not always exist in the categories we
are interested in.
\begin{definition}\label{def:relativelybig}
  Let $f\colon X \to Y$ be a proper morphism such that $X$ is integral.
  Consider an invertible sheaf $\sL$ on $X$.
  For each integer $n > 0$, consider the adjunction morphism
  $f^*f_*\sL^{\otimes n} \to \sL$.
  Let $\fb \subseteq \cO_X$ be the coherent ideal sheaf such that this
  adjunction morphism factors as
  \[
    f^*f_*\sL^{\otimes n} \longtwoheadrightarrow \fb \cdot \sL^{\otimes n}
    \hooklongrightarrow \sL^{\otimes n}.
  \]
  Taking symmetric algebras and relative Proj, we obtain the
  commutative diagram
  \begin{equation}\label{eq:bigdiagram}
    \begin{tikzcd}[column sep=0]
      \Bl_\fb X \arrow[start anchor={[xshift=-1em]south east},
      end anchor={[xshift=-.5em]north}]{dr}[swap]{\pi}
      \arrow[hook]{rr} & & \PP_X(f^*f_*\sL^{\otimes n})
      \arrow[start anchor={[xshift=1em]south west},
      end anchor={[xshift=.5em]north}]{dl} \arrow{r}{\sim}
      &[1.475em] X \times_Y \PP_Y(f_*\sL^{\otimes n})
      \dar{\textup{pr}_2}\\
      & X \arrow[dashed]{rr} & & \PP_Y(f_*\sL^{\otimes n}) \rar & Y
    \end{tikzcd}
  \end{equation}
  where the composition in the bottom row is $f$.
  We say that $\sL$ is \textsl{$f$-big} if there exists an integer $n > 0$
  such that the dashed arrow in the diagram above is \textsl{generically finite}
  onto the closure of its image, i.e., is finite away from a nowhere dense set
  in the closure of its image.
  We extend this definition to Cartier divisors $L$ on $X$ by asking that its
  associated invertible sheaf $\cO_X(L)$ is $f$-big.
  If $D$ is a $\QQ$-Cartier divisor, then we say that $D$ is \textsl{$f$-big}
  if some positive integer multiple of $D$ is $f$-big.
\end{definition}
We note the partially defined map $X \dashrightarrow \PP_Y(f_*\sL^{\otimes n})$
in \eqref{eq:bigdiagram}
is meromorphic in the sense of Remmert \citeleft\citen{Rem57}\citemid Def.\
15\citepunct \citen{Pet94}\citemid Definition 1.7\citeright\ 
in the complex analytic case and 
is meromorphic in the sense of Morrow and Rosso \cite[Definition 3.2]{MR23} in
the non-Archimedean case.
\begin{remark}
  Suppose $f$ as in Definition \ref{def:relativelybig} is projective.
  Then, we see that $X \dashrightarrow \PP_Y(f_*\sL^{\otimes n})$ is generically
  finite if and only if for every affinoid subdomain $U \subseteq Y$, this map
  is the relative analytification of a generically finite rational map of
  schemes over $U$.
  As a result, we see that our definition of relative bigness coincides with
  our previous definition in \cite[Definition 25.5]{LM} for projective
  morphisms, and hence is compatible with GAGA.
\end{remark}
We define (locally) Moishezon morphisms as follows.
Our definition is an adaptation of the definition for complex analytic spaces
in \citeleft\citen{Moi74}\citemid Definition 2\citepunct
\citen{Kol22}\citemid Definition 11\citeright.
\begin{definition}
  Let $f\colon X \to Y$ be a proper morphism.
  We say that $f$ is \textsl{Moishezon} if the morphism $f$ is bimeromorphic
  (over $Y$) to a projective morphism.
  We say that $f$ is \textsl{locally Moishezon} if for every point $y \in Y$,
  there exists an affinoid subdomain $U \subseteq Y$ containing $y$ such that
  the morphism $f^{-1}(U) \to U$ is Moishezon.
\end{definition}
We now show that the existence of an $f$-big invertible sheaf implies that $f$
is locally Moishezon.
\begin{lemma}\label{lem:bigimplieslocallymoishezon}
  Let $f\colon X \to Y$ be a proper morphism such that $X$ is integral.
  Consider an $f$-big invertible sheaf $\sL$ on $X$.
  For all $n > 0$
  such that the dashed arrow in \eqref{eq:bigdiagram} is generically finite onto
  the closure of its image, let $f^p$ be the finite part of the Stein factorization
  \[
    \Bl_\fb X \longrightarrow X^p \overset{f^p}{\longrightarrow}
    \PP_Y(f_*\sL^{\otimes n})
  \]
  of the morphism $\Bl_\fb X \to \PP_Y(f_*\sL^{\otimes n})$.
  Then, the composition
  \[
    X^p \overset{f^p}{\longrightarrow}
    \PP_Y(f_*\sL^{\otimes n}) \longrightarrow Y
  \]
  is a locally projective morphism bimeromorphic (over $Y$) to $f$.
  In particular, $f$ is locally Moishezon.
\end{lemma}
\begin{proof}
  Let $n > 0$ be an integer such that the dashed arrow in \eqref{eq:bigdiagram}
  is generically finite onto the closure of its image.
  The morphism $\Bl_\fb X \to \PP_Y(f_*\sL^{\otimes n})$ is then generically
  finite onto its image.
  Now consider the Stein factorization of this morphism
  \[
    \Bl_\fb X \longrightarrow X^p \longrightarrow \PP_Y(f_*\sL^{\otimes n}),
  \]
  which exists for algebraic spaces by
  \cite[\href{https://stacks.math.columbia.edu/tag/0A1B}{Tag
  0A1B}]{stacks-project}, for semianalytic germs of complex analytic spaces
  by applying \cite[10.6.1]{GR84} to a representative of this morphism, for
  Berkovich spaces by \cite[Proposition 3.3.7]{Ber90}, for rigid analytic
  spaces by \cite[Proposition 9.6.3/5]{BGR84}, and for adic spaces locally
  of weakly finite type over a field by \cite[Theorem 3.9]{Man23}.
  The morphism $X^p \to \PP_Y(f_*\sL^{\otimes n})$ is finite, and the morphism
  $\Bl_\fb X \to X^p$ is surjective and bimeromorphic.
  We therefore see that the morphism
  \[
    X^p \overset{\phi}{\longrightarrow} \PP_Y(f_*\sL^{\otimes n}) \longrightarrow Y
  \]
  is bimeromorphic (over $Y$) to $f$ and
  is a locally projective morphism in the sense that for every $y \in Y$, there
  exists an affinoid subdomain $U \subseteq Y$ containing $y$ such that
  $X^p \times_Y U \to U$ is projective.
  In particular, $f$ is locally Moishezon.
\end{proof}
\subsection{Proof of Theorem
\texorpdfstring{\ref{thm:kvvanishingothercats}}{A'}}
We can now prove Theorem \ref{thm:kvvanishingothercats}.
\begin{proof}[Proof of Theorem \ref{thm:kvvanishingothercats}]
  Since the statement is local on $Y$, we may assume that $Y$ is
  affinoid.
  We will also be able to replace $Y$ by smaller affinoid subdomains during the
  proof below.\smallskip
  \par We first show that Theorem \ref{thm:kvvanishingothercats} holds when $f$
  is projective.
  The case for algebraic spaces follows by flat base change
  \cite[\href{https://stacks.math.columbia.edu/tag/073K}{Tag
  073K}]{stacks-project} applied to an \'etale cover of $Y$.
  For the other cases,
  by the GAGA-type results in \cite[\S\S23--25]{LM}, we know that
  $f$ is the analytification of a projective morphism of schemes and that the
  hypotheses in Theorem \ref{thm:kvvanishingothercats} are compatible under the
  GAGA correspondence.
  Since the vanishing and injectivity statements on the scheme side hold by 
  Theorem \ref{thm:kvvanishing}, the compatibility of analytification with
  higher direct images \citeleft\citen{EGAIII1}\citemid Proposition
  5.1.2\citepunct \citen{AT19}\citemid Theorem C.1.1\citepunct
  \citen{Poi10}\citemid Th\'eor\`eme A.1\citepunct \citen{Kop74}\citemid
  Folgerung 6.6\citepunct \citen{Hub07}\citemid Corollary 6.4\citeright\ shows
  that the desired vanishing and injectivity statements are preserved under
  analytification.\smallskip
  \par For the locally Moishezon case, we first reduce case
  $(\ref{setup:introberkovichspaces})$ to case
  $(\ref{setup:introrigidanalyticspaces})$.
  Note that in both cases, $Y$ is a point.
  Let $k \subseteq K_r$ be a field extension where $K_r$ is a complete
  non-trivially valued non-Archimedean field such that $X
  \mathbin{\hat{\otimes}}_k K_r$ is a strictly $K_r$-analytic space
  (such a $K_r$ exists as in the proof of \cite[Proposition 2.2.4]{Ber90}).
  Since coherent cohomology is compatible with the field extension $k
  \subseteq K_r$ by \cite[Proposition 2.1.2$(ii)$]{Ber90} (see the proof of
  \cite[Proposition 3.3.5]{Ber90}), we can detect the desired vanishing and
  injectivity statements after base change to $K_r$.
  Note that the formation of $\omega_X$ is compatible with base change to $K_r$
  \cite[Proposition 3.3.3$(ii)$]{Ber93}.
  Bigness is compatible with this extension since blowups and relative
  Proj are compatible with ground field extensions (locally, they are defined as
  the scheme-theoretic notions on affinoid subdomains, which are compatible with
  base change).
  Nefness is compatible with this extension because proper Berkovich curves are
  always projective
  \citeleft\citen{dJ95}\citemid Proposition 3.2 and Remark 3.3\citepunct
  \citen{Duc}\citemid Th\'eor\`eme 3.7.2\citeright, and hence are
  algebraizations of projective curves, to which we we can apply the
  scheme-theoretic result in \cite[Lemma 2.18(1)]{Kee03}.
  Finally, the comparison between (Berkovich) strictly $K_r$-analytic spaces and
  rigid $K_r$-analytic spaces in \cite[Theorem 1.6.1]{Ber93} is compatible with
  coherent cohomology \cite[p.\ 37]{Ber93}, the formation of $\omega_X$ (by
  definition on the smooth locus, which is preserved by \cite[Proposition
  3.3.1$(iii)$]{Ber90}),
  and both bigness (since finite morphisms are by \cite[Proposition
  3.3.2]{Ber90}) and nefness (since coherent cohomology, and hence the computation
  of Euler characteristics, is compatible as above).
  This shows we may reduce case
  $(\ref{setup:introberkovichspaces})$ to case
  $(\ref{setup:introrigidanalyticspaces})$.
  \par It remains to reduce to the case when $f$ is projective in cases
  $(\ref{setup:introalgebraicspaces})$,
  $(\ref{setup:introcomplexanalyticgerms})$, and
  $(\ref{setup:introrigidanalyticspaces})$,
  where in the last case we assume
  that $Y$ is a point.
  We claim that after possibly replacing $Y$ with an affinoid subdomain,
  we can construct a commutative diagram
  \begin{equation}\label{eq:redlocmoitoproj}
    \begin{tikzcd}[sep=1.475em]
      & \makebox[\widthof{$Y$}][c]{$\tilde{X}$}\dar{\mu}\\
      & \makebox[\widthof{$Y$}][c]{$\Bl_\fb X$}\arrow{dl}[swap]{\pi}\arrow{dr}{\phi}\\
      X \arrow{dr}[swap]{f}
      & & \makebox[\widthof{$X$}][l]{$X^p$} \arrow{dl}{f^p}\\
      & Y
    \end{tikzcd}
  \end{equation}
  of proper morphisms, where $\pi$, $\phi$, and $\mu$ are bimeromorphic, $f^p$
  is projective, and $\pi \circ \mu$ is a projective log resolution of
  $(X,\Delta)$ such that $f^p \circ \phi \circ \mu$ is projective.
  The bottom square exists in case $(\ref{thm:kvvanishingothercatsvan})$ by
  Lemma \ref{lem:bigimplieslocallymoishezon}, and exists in case
  $(\ref{thm:kvvanishingothercatsinj})$ by the locally Moishezon assumption.
  To construct the desired projective log resolution, we first apply Chow's
  lemma for bimeromorphic/birational morphisms, and then take a projective log
  resolution.
  Projective log resolutions exist in these categories by
  \citeleft\citen{AHV77}\citemid Theorem 5.3.1\citepunct
  \citen{Sch99}\citemid Theorem 3.2.3 and Remark on p.\ 327\citepunct
  \citen{Tem18}\citemid Theorem 1.1.13\citeright.
  For algebraic spaces, Chow's lemma (without birationality assumptions)
  holds by \cite[\href{https://stacks.math.columbia.edu/tag/088U}{Tag
  088U}]{stacks-project}.
  For semianalytic germs of complex analytic spaces,
  we can apply the complex-analytic
  version of Chow's lemma \citeleft\citen{Moi74}\citemid Par.\ 2\citepunct
  \citen{Hir75}\citemid Corollary 2\citeright\ to a representative of $\phi$.
  Finally, for rigid analytic spaces,
  we can apply
  Chow's lemma for proper
  Moishezon rigid analytic spaces \cite[Corollary 4.1.2]{Con10} since $Y$ is a
  point.
  \par Next, we construct appropriate divisors on $\tilde{X}$ to which we will
  apply the projective case of Theorem \ref{thm:kvvanishingothercats} proved
  above.
  As in the proof of Theorem \ref{thm:kvvanishing} in
  \S\ref{sect:proofoftheorema}, we can write
  \[
    K_{\tilde{X}} + (\pi \circ \mu)^*M + (\pi \circ \mu)_*^{-1}\Delta + \sum_i
    \bigl(\lceil a_i \rceil - a_i\bigr)E_i \sim_\QQ (\pi \circ \mu)^*N + \sum_i
    \lceil a_i \rceil E_i,
  \]
  in which case setting
  \[
    \tilde{N} \coloneqq (\pi \circ \mu)^*N + \sum_i \lceil a_i \rceil E_i,
  \]
  we have
  \[
    \tilde{N} \sim_\QQ K_{\tilde{X}} + (\pi \circ \mu)^*M + (\pi \circ
    \mu)_*^{-1}\Delta + \sum_i \bigl(\lceil a_i \rceil - a_i\bigr)E_i.
  \]
  Since $(\pi \circ \mu)^*M$ is $(\pi \circ \mu)$-big
  and $(\pi \circ \mu)$-nef by the projection formula,
  we know that
  \[
    R^i(\pi \circ \mu)_*\bigl(\cO_{\tilde{X}}(\tilde{N})\bigr) = 0
  \]
  for all $i > 0$ by the projective case of $(\ref{thm:kvvanishingothercatsvan})$.
  \par For $(\ref{thm:kvvanishingothercatsvan})$, we know that
  $(\pi \circ \mu)^*M$ is $(f^p \circ \phi \circ \mu)$-nef and
  $(f^p \circ \phi \circ \mu)$-big by the projection
  formula, the normality of $X$, and the commutativity of the diagram
  \eqref{eq:redlocmoitoproj}.
  We therefore see that
  \[
    R^i(f \circ \pi \circ \mu)_*\bigl(\cO_{\tilde{X}}(\tilde{N})\bigr)
    = R^i(f^p \circ \phi \circ \mu)_*\bigl(\cO_{\tilde{X}}(\tilde{N})\bigr) = 0
  \]
  for all $i > 0$ by the projective case of $(\ref{thm:kvvanishingothercatsvan})$.
  We then consider the Leray spectral sequence
  \[
    E_2^{i,j} = R^if_*\bigl(R^j(\pi \circ \mu)_*\bigl(\cO_{\tilde{X}}(\tilde{N})\bigr)\bigr)
    \Rightarrow R^{i+j}(f \circ \pi \circ \mu)_*\bigl(\cO_{\tilde{X}}(\tilde{N})\bigr).
  \]
  By the vanishing $R^i(\pi \circ \mu)_*(\cO_{\tilde{X}}(\tilde{N})) = 0$,
  the $E_2$ page of this spectral sequence is
  concentrated in the row $j = 0$.
  We therefore have a natural isomorphism
  \[
    R^if_*\bigl((\pi \circ \mu)_*\bigl(\cO_{\tilde{X}}(\tilde{N})\bigr)\bigr)
    \simeq R^i(f \circ \pi \circ \mu)_*\bigl(\cO_{\tilde{X}}(\tilde{N})\bigr) = 0
  \]
  for all $i > 0$.
  Now by definition of $\tilde{N}$ and the projection formula, we have
  \[
    (\pi \circ \mu)_*\bigl(\cO_{\tilde{X}}(\tilde{N})\bigr) \simeq \cO_X(N) \otimes_{\cO_X}
    (\pi \circ \mu)_*\biggl(\cO_{\tilde{X}}\biggl(\sum_i \lceil a_i \rceil E_i \biggr)\biggr)
    \simeq \cO_X(N)
  \]
  since $X$ is normal and the $E_i$ are $(\pi \circ \mu)$-exceptional (cf.\ \cite[Example
  2.1.16]{Laz04a}).
  We therefore have
  \[
    R^if_*\bigl(\cO_X(N)\bigr) \simeq
    R^if_*\bigl((\pi \circ \mu)_*\bigl(\cO_{\tilde{X}}(\tilde{N})\bigr)\bigr) = 0.
  \]
  \par For $(\ref{thm:kvvanishingothercatsinj})$, since the map $\cO_X \to
  \cO_X(D)$ factors the map $\cO_X \to \cO_X(D+D')$, we can replace $D$ by
  $D+D'$ to assume assume that $\cO_X(D) \simeq \cO_X(nM)$,
  and in particular, we may assume that $D$ is Cartier.
  The projective case implies that the canonical morphisms
  \[
    R^i(f \circ \pi \circ \mu)_*\bigl(\cO_{\tilde{X}}(\tilde{N})\bigr)
    \longrightarrow 
    R^i(f \circ \pi \circ \mu)_*\bigl(\cO_{\tilde{X}}\bigl(\tilde{N}+(\pi \circ
    \mu)^*D\bigr)\bigr)
  \]
  are injective for all $i$ since $(\pi \circ \mu)^*M$ is $(f \circ \pi \circ
  \mu)$-semi-ample (by the projection formula, the normality of $X$, and the
  commutativity of the diagram \eqref{eq:redlocmoitoproj})
  and $\cO_{\tilde{X}}((\pi \circ \mu)^*D) \simeq
  \cO_{\tilde{X}}(n\,(\pi \circ \mu)^*M)$.
  The same argument using the Leray spectral sequence as in the previous
  paragraph then implies that the canonical morphisms
  \[
    R^if_*\bigl(\cO_X(N)\bigr) \longrightarrow
    R^if_*\bigl(\cO_X(N+D)\bigr)
  \]
  are injective for all $i$.
\end{proof}

\addtocontents{toc}{\protect\bigskip}
\bookmarksetup{startatroot}

\end{document}